\documentclass[a4paper,11pt,final]{amsart}
\usepackage{amsthm,amssymb,amsmath}
\usepackage{tikz,tikz-cd}
\tikzcdset{scale cd/.style={every label/.append style={scale=#1},
    cells={nodes={scale=#1}}}}
\usepackage[shortlabels]{enumitem}
\usepackage{hang} % for the glossary
\usepackage{centernot}
\usepackage[small]{caption} % sets small as default font size in figure captions
\usepackage{xcolor}
\usepackage{hyperref}
  \hypersetup{
      colorlinks,
      linkcolor={red!50!black},
      citecolor={blue!50!black},
      urlcolor={blue!80!black}
  }
  \usepackage{thmtools}
\usepackage[capitalise,noabbrev]{cleveref}
\usepackage{booktabs}
\usepackage{mathtools} % this is used only for \coloneqq -- if it breaks, use: \newcommand{\coloneqq}{\mathrel{\mathop:}=}

\usepackage{pgfplots}
\usepackage{mathrsfs}
\usetikzlibrary{arrows}

\usepackage{fullpage}

\newcommand{\defn}[1]{\textbf{\emph{#1}}}
\newcommand{\defnn}[1]{\textbf{\emph{#1}}}

\newtheoremstyle{boldhead}%               % head notes to remarks and examples in Boldface
  {}%                                     % Space above
  {}%                                     % Space below
  {}%                                     % Body font
  {}%                                     % Indent amount
  {\bfseries}%                            % Theorem head font
  {.}%                                    % Punctuation after theorem head
  { }%                                    % Space after theorem head, ' ', or \newline
  {\thmname{#1}\thmnumber{ #2}\thmnote{ (#3)}}% % Theorem head spec (can be left empty, meaning `normal')

\theoremstyle{plain}
\newtheorem{theorem}{Theorem}[section]

\newtheorem{lemma}[theorem]{Lemma}
\newtheorem{corollary}[theorem]{Corollary}
\newtheorem{proposition}[theorem]{Proposition}

\theoremstyle{boldhead}

\newtheorem{remark}[theorem]{Remark}
\newtheorem{example}[theorem]{Example}
\newtheorem{question}[theorem]{Question}

\newtheorem{definition}[theorem]{Definition}
\newtheorem{remarks}[theorem]{Remarks}

\newtheorem{construction}[theorem]{Construction}

\newcommand{\Step}[1]{\noindent \emph{Step #1:}}

\newcommand{\ie}{i.e.\ }
\newcommand{\eg}{e.g.\ }

\newcommand{\timplies}{$\Longrightarrow$}  % text mode implies: for use in (i) \timplies (ii), also less spacing than math mode \implies

\DeclarePairedDelimiterX{\norm}[1]{\lVert}{\rVert}{#1} % norm
\DeclarePairedDelimiterX{\abss}[1]{\lvert}{\rvert}{#1} % absolute value

\DeclareMathOperator{\id}{{\mathsf{id}}} % identity map
\newcommand{\kk}{{\mathbf{k}}}           % ground field
\newcommand{\PSL}{\mathrm{PSL}}
\newcommand{\GL}{\mathrm{GL}_2^+(\R)}    % this macro specifically for GL_2^+(IR)

\newcommand{\Ginzburg}{\Gamma_{\!2}}      % special typesetting for Ginzburg algebra 

\newcommand{\EG}[1]{\mathrm{EG}(#1)}

\renewcommand{\Re}{\mathrm{Re}\,}
\renewcommand{\Im}{\mathrm{Im}\,}

\newcommand{\rk}[1]{\mathrm{rk}\, #1} 
\newcommand{\rank}{\mathrm{rk}}

\DeclareMathAlphabet{\mathpzc}{OT1}{pzc}{m}{it}

\newcommand{\cat}[1]{\mathsf{#1}}
\newcommand{\Hom}[2]{{\mathrm{Hom}}(#1,#2)}
\newcommand{\Homm}[3]{{\mathrm{Hom}}_{#1}(#2,#3)}
\newcommand{\End}[2]{{\mathrm{End}}_{#1}(#2)}
\newcommand{\Aut}[1]{\mathrm{Aut}(#1)}
\newcommand{\Aaut}[2]{\mathrm{Aut}_{#1}(#2)}

\newcommand{\inj}{\hookrightarrow}
\newcommand{\too}{\longrightarrow}

\newcommand{\isom} {\xrightarrow{\raisebox{-0.5ex}[0ex][0ex]{$\sim$}}}

\newcommand{\CC}{\cat{C}}
\newcommand{\HH}{\cat{H}}
\newcommand{\NN}{\cat{N}}
\newcommand{\MM}{\cat{M}}

\newcommand{\TT}{\cat{T}}

\DeclareMathOperator{\coh}{\mathsf{coh}}          % coherent sheaves
\newcommand{\Db}{\cat{D}^b}                       % bounded derived category
\newcommand{\nilrep}  {\cat{nilrep}}              % nilpotent representations
\newcommand{\orth}{^\perp}                         % orthogonal complement
\newcommand{\Thick}[1]{\mathrm{Thick}(#1)}			% poset of thick subcategories in #1
\newcommand{\thick}[2]{\mathsf{thick}_{#1}(#2)}    % thick subcategory generated by #2 in #1
\newcommand{\triang}[2]{\mathsf{triang}_{#1}(#2)}  % triangulated subcategory generated by #2 in #1
\newcommand{\Clext}[1]{{\langle\, #1 \,\rangle}}  % extension closure of #1, with extra space
\newcommand{\clext}[1]{{\langle #1 \rangle}}      % extension closure of #1, without extra space

\newcommand{\twist}{\textsc{tw}} %spherical twist

\renewcommand{\mod}[1]{\mathsf{mod}(#1)}

\newcommand{\Slice}[1]{\mathrm{Slice}(#1)}
\newcommand{\Stab}[1]{\mathrm{Stab}(#1)}              % space of stability conditions
\newcommand{\LaxStab}[1]{\mathrm{LaxStab}(#1)}        % space of lax stability conditions, w/o closure 
\newcommand{\LaxStabN}[2]{\mathrm{LaxStab}(#1,#2)}    % supsace of lax stability conditions with given massless category
\newcommand{\StabL}[1]{\mathrm{Stab^L}(#1)}           % space of lax stability conditions, with closure condition
\newcommand{\StabLN}[2]{\mathrm{Stab^L}(#1,#2)}       % stratum of space of lax stability conditions
\newcommand{\StabQ}[1]{\mathrm{Stab^Q}(#1)}           % space of quotient stability conditions
\newcommand{\StabQN}[2]{\mathrm{Stab^Q}(#1,#2)}       % stratum of space of quotient stability conditions
\newcommand{\Stabvw}[1]{\mathrm{Stab_\Lambda^{\text{vw}}(#1)}} % space of very weak stability conditions (Piyaratne & Toda)

\newcommand{\PStabL}[1]{\mathbb{P}\mathrm{Stab}^L(#1)}
\newcommand{\PStabQ}[1]{\mathbb{P}\mathrm{Stab}^Q(#1)}

\newcommand{\PStab}[1]{\mathbb{P}\mathrm{Stab}(#1)}

\newcommand{\hatstab}[1]{\mathrm{Stab}\!\left(#1\right)^{\wedge{}}} % was:  {\widehat{\mathrm{Stab}}\!\left(#1\right)} %Bolognese's completion
\newcommand{\Bstab}[1]{\mathrm{Stab}^B\!\left(#1\right)} %subspace of limits of Bolognese's Cauchy sequences

\newcommand{\cm}{\mathcal{Z}}        % charge map

\newcommand{\restrict}[1]{R_{#1}} % restriction map: massless part
\newcommand{\massive}[1]{M_{#1}}   % massive part
\newcommand{\deform}[1]{D_{#1}}   % deformation map

\renewcommand{\phi}{\varphi}
\renewcommand{\epsilon}{\varepsilon}
\renewcommand{\emptyset}{\varnothing}
\renewcommand{\setminus}{\backslash}
\renewcommand{\geq}{\geqslant}
\renewcommand{\leq}{\leqslant}

\newcommand{\N}{\mathbb{N}}
\newcommand{\Z}{\mathbb{Z}}
\newcommand{\Q}{\mathbb{Q}}
\newcommand{\R}{\mathbb{R}}
\newcommand{\C}{\mathbb{C}}
\newcommand{\U}{\mathbb{H}}  % upper half-plane
\newcommand{\PP}{\mathbb{P}} % projective space
\newcommand{\pd}{\mathbb{D}} % Poincare disk

\newcommand{\cE}{\mathcal{E}}

\newcommand{\cO}{\mathcal{O}}

\pgfmathsetmacro{\ra}{2.5}

\newcommand{\hgline}[3]{
  \pgfmathsetmacro{\thetaone}{#2}
  \pgfmathsetmacro{\thetatwo}{#3}
  \pgfmathsetmacro{\theta}{(\thetaone+\thetatwo)/2}
  \pgfmathsetmacro{\phi}{abs(\thetaone-\thetatwo)/2}
  \pgfmathsetmacro{\close}{less(abs(\phi-90),0.0001)}
  \ifdim \close pt = 1pt
      \draw[red] (\theta+180:\ra) -- (\theta:\ra);
  \else
      \pgfmathsetmacro{\R}{tan(\phi)*\ra}
      \pgfmathsetmacro{\distance}{sqrt(\ra^2+\R^2)}
      \draw[thick, #1] (\theta:\distance) circle (\R)
  \fi
}

\newcommand{\hgfill}[3]{
  \pgfmathsetmacro{\thetaone}{#1}
  \pgfmathsetmacro{\thetatwo}{#2}
  \pgfmathsetmacro{\theta}{(\thetaone+\thetatwo)/2}
  \pgfmathsetmacro{\phi}{abs(\thetaone-\thetatwo)/2}
  \pgfmathsetmacro{\close}{less(abs(\phi-90),0.0001)}
  \ifdim \close pt = 1pt
      \draw[red] (\theta+180:\ra) -- (\theta:\ra);
  \else
      \pgfmathsetmacro{\R}{tan(\phi)*\ra}
      \pgfmathsetmacro{\distance}{sqrt(\ra^2+\R^2)}
      \fill[#3] (\theta:\distance) circle (\R)
  \fi
}

\usetikzlibrary{calc,positioning,arrows,matrix}
\usepackage{ifthen}

\tikzstyle{matrix of math nodes}=[
  matrix of nodes,
  nodes={%
   execute at begin node=$,%
   execute at end node=$%
  }%
]

\newcommand{\halfscale}{1/4} 

\newcommand{\yyy}[6]{ % the six parameters are (x1,y1,x2,y2,x3,y3), where 1=point on diagram, 2=simple S, 3=simple T
   \begin{scope}[black]  % the half plane: positive imaginary axis, dashed negative real axis and positive real axis
          \draw[dashed] (#1,#2) -- +($\halfscale*(-1.5,0)$); 
          \draw         (#1,#2) -- +($\halfscale*(1.5,0)$);  
          \draw         (#1,#2) -- +($\halfscale*(0,1.5)$); 
   \end{scope}
   \begin{scope}[ultra thick]
   \ifthenelse{ \equal{#3}{#5} \AND \equal{#4}{#6} }   % both simples map to the same point...
   {  \ifthenelse{ \equal{#3}{0} }                      
        { % ...if on imaginary axis:
             \draw[arrows = {-left to},  transform canvas={shift={(-0.03,0)}}, red!80]  (#1,#2) -- +($\halfscale*(#3,#4)$);
             \draw[arrows = {-right to}, transform canvas={shift={( 0.03,0)}}, blue!80] (#1,#2) -- +($\halfscale*(#5,#6)$); }
        { % ...else have to distinguish between positive /negative on real axis:
          \ifthenelse{ \equal{#3}{-1} }
          { \draw[arrows = {-left to}, transform canvas={shift={(0,-0.03)}}, red!80]  (#1,#2) -- +($\halfscale*(#3,#4)$);
            \draw[arrows = {-right to}, transform canvas={shift={(0, 0.03)}}, blue!80] (#1,#2) -- +($\halfscale*(#5,#6)$); }
          { \draw[arrows = {-left to}, transform canvas={shift={(0, 0.03)}}, red!80]  (#1,#2) -- +($\halfscale*(#3,#4)$);
            \draw[arrows = {-right to}, transform canvas={shift={(0,-0.03)}}, blue!80] (#1,#2) -- +($\halfscale*(#5,#6)$); }
        }    
   }       
   { % simples map to different points; need to check if one of them maps to 0
     \ifthenelse{ \equal{#3}{0} \AND \equal{#4}{0} } {} { \draw[arrows = {->},red!80]  (#1,#2) -- +($\halfscale*(#3,#4)$); }
     \ifthenelse{ \equal{#5}{0} \AND \equal{#6}{0} } {} { \draw[arrows = {->},blue!80] (#1,#2) -- +($\halfscale*(#5,#6)$); }
     \ifthenelse{ \equal{#3}{0} \AND \equal{#4}{0} } { \filldraw[red!80]  (#1,#2) circle (0.04); } {}
     \ifthenelse{ \equal{#5}{0} \AND \equal{#6}{0} } { \filldraw[blue!80] (#1,#2) circle (0.04); } {}
   }
   \end{scope} 
}

\title{Partial Compactification of Stability Manifolds \\ via Massless Semistable Objects}
\author{Nathan Broomhead, David Pauksztello, David Ploog, Jon Woolf}

\dedicatory{Dedicated to Mike Prest on the occasion of his seventieth birthday.}

\makeatletter
\@namedef{subjclassname@2020}{%
  \textup{2020} Mathematics Subject Classification}
\makeatother

\begin{document}

\begin{abstract}
We introduce two extensions of the space of Bridgeland stability conditions of a triangulated category. First we consider lax stability conditions where semistable objects are allowed to have mass zero but still have a phase. The subcategory of massless objects is thick and there is an induced Bridgeland stability on the quotient category. We study deformations of lax stability conditions. Second we consider the space arising by identifying lax stability conditions which are deformation-equivalent with fixed charge. This second space is stratified by stability spaces of Verdier quotients of the triangulated category by thick subcategories of massless objects. We illustrate our results through examples in which the Grothendieck group has rank $2$. For these, our extended stability spaces can be explicitly described and related to the wall-and-chamber structure of the stability space.
\end{abstract}



%18G80: Derived categories, triangulated categories
%16E35:	Derived categories and associative algebras
%14F08:	Derived categories of sheaves, dg categories, and related constructions in algebraic geometry
\maketitle

{\small
  \tableofcontents
}

\addtocontents{toc}{\protect\setcounter{tocdepth}{0}}

\vfill

{\centering
  \begin{tabular}{ccc}
  \scalebox{0.80}{\begin{tikzpicture}
%draw Poincare disk
\draw[thick, dashed, fill=blue!05] (0,0) circle (\ra);

\begin{scope}
%clip
\clip (0,0) circle (\ra);
%fill regions
\foreach \theta in {0,120,240} \hgfill{\theta}{\theta+120}{blue!10};
\foreach \theta in {0,60,...,300} \hgfill{\theta}{\theta+60}{blue!15};
\foreach \theta in {0,30,...,330} \hgfill{\theta}{\theta+30}{blue!20};
\foreach \theta in {0,15,...,345} \hgfill{\theta}{\theta+15}{blue!30};

%draw walls

\foreach \theta in {0,120,240} \hgline{blue}{\theta}{\theta+120};
\foreach \theta in {0,60,...,300} \hgline{blue}{\theta}{\theta+60};
\foreach \theta in {0,30,...,330} \hgline{blue}{\theta}{\theta+30};
\foreach \theta in {0,15,...,345} \hgline{blue}{\theta}{\theta+15};

\end{scope}

\end{tikzpicture}} &\hspace*{3em}&
  \scalebox{0.80}{\begin{tikzpicture}
%draw Poincare disk
\draw[thick, dashed, fill=red!5] (0,0) circle (\ra);

\begin{scope}
%clip
\clip (0,0) circle (\ra);
%fill chambers
\hgfill{90}{230}{red!10};
\hgfill{90}{170}{red!15};
\hgfill{90}{130}{red!20};
\hgfill{90}{110}{red!25};

\hgfill{90}{-50}{red!10};
\hgfill{90}{10}{red!15};
\hgfill{90}{50}{red!20};
\hgfill{90}{70}{red!25};

\hgfill{95}{100}{red!30};
\hgfill{100}{110}{red!30};
\hgfill{110}{130}{red!30};
\hgfill{130}{170}{red!30};
\hgfill{170}{230}{red!30};

\hgfill{85}{80}{red!30};
\hgfill{80}{70}{red!30};
\hgfill{70}{50}{red!30};
\hgfill{50}{10}{red!30};
\hgfill{10}{-50}{red!30};

\hgfill{230}{310}{red!30};

%draw walls
\hgline{red}{95}{100};
\hgline{red}{100}{110};
\hgline{red}{110}{130};
\hgline{red}{130}{170};
\hgline{red}{170}{230};

\hgline{red}{85}{80};
\hgline{red}{80}{70};
\hgline{red}{70}{50};
\hgline{red}{50}{10};
\hgline{red}{10}{-50};

\hgline{red}{230}{310};

\hgline{red}{90}{95};
\hgline{red}{90}{100};
\hgline{red}{90}{110};
\hgline{red}{90}{130};
\hgline{red}{90}{170};
\hgline{red}{90}{230};
\hgline{red}{90}{-50};
\hgline{red}{90}{10};
\hgline{red}{90}{50};
\hgline{red}{90}{70};
\hgline{red}{90}{80};
\hgline{red}{90}{85};

\end{scope}

\end{tikzpicture}}\\
  Ginzburg algebra $\CC = \Db(\Ginzburg A_2)$ &&
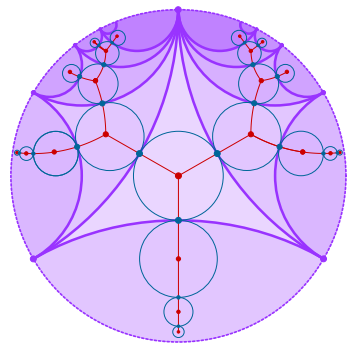
  Bound path algebra $\CC = \Db(\Lambda_{1,2,0})$
\end{tabular}

\medskip
Examples of quotient stability spaces $\StabQ{\CC}^*$ up to $\C$ action; see \cref{fig:A2}.
\par
}
\newpage

\section*{Glossary}

\newcommand{\glossp}[1]{\hfill p~\pageref{#1}}

\subsection*{Slicings}

\begin{compacthang}
\item $\Slice{\CC}$, the set of locally finite slicings on $\CC$ \glossp{def:space of slicings}
\item $P \in \Slice{\CC}$ is \defn{adapted} to thick $\NN \subset \CC$ if $P$ restricts to $\NN$ and $P(I) \cap \NN \subset P(I)$ a Serre subcategory for all $I=[\phi,\phi+1)$ and $I=(\phi,\phi+1]$ \glossp{def:adapted slicing}
\item $P \in \Slice{\CC}$ is \defn{well adapted} if also the quotient slicing $P_{\CC/\NN}$ is locally finite \glossp{def:well-adapted slicing}
\end{compacthang}

\subsection*{Charges}

\begin{compacthang}
\item $v\colon K(\CC) \to \Lambda$, a surjective homomorphism onto lattice $\Lambda$ with fixed inner product \glossp{subsec:charges}
\item $\Hom{\Lambda/\Lambda_\NN}{\C}$, the subset of charge maps $\Hom{\Lambda}{\C}$ vanishing on $\Lambda_\NN$ \glossp{subsec:charges}
\item $\Hom{\Lambda_\NN}{\C}$, the charges on $\Lambda_\NN$, subset of $\Hom{\Lambda}{\C}$ via the inner product \glossp{subsec:charges}
\end{compacthang}

\subsection*{The spaces}

\begin{compacthang}
\item $\Stab{\CC} \subseteq \Slice{\CC} \times \Hom{\Lambda}{\C}$, the set of stability conditions
      whose charge map factors as $Z \colon K(\CC) \xrightarrow{v} \Lambda \to \C$; also called \defn{strict} stability conditions here
      \glossp{subsec:stability conditions}
\item $\cm\colon \Stab{\CC} \to \Hom{\Lambda}{\C}$, the charge map, also for larger spaces
\item $\sigma = (P,Z)$ satisfies \defn{$\delta$-lax support} if $\exists K>0$ with $|Z(c)| \geq \norm{c}/K$ for all massive, indecomposable $c\in\CC$ with $\phi^+_\sigma(c) - \phi^-_\sigma(c) < 2\delta$ \glossp{def:lax stability condition}
\item $\LaxStab{\CC} \subseteq \Slice{\CC} \times \Hom{\Lambda}{\C}$, the space of \defn{lax stability conditions} (objects can have mass 0, and the $\delta$-lax support property holds for some $\delta>0$) \glossp{def:space of lax stability conditions}
\item $\StabL{\CC} = \LaxStab{\CC} \cap \overline{\Stab{\CC}}$, the \defnn{lax closure} of $\Stab{\CC}$ \glossp{def:StabL}
\item $\StabQ{\CC} = \StabL{\CC}/{\sim}$, the \defnn{quotient stability space} of $\Stab{\CC}$: $\sigma \sim \sigma' \iff \NN_\sigma = \NN_{\sigma'}$ and $\sigma, \sigma'$ induce the same strict stability condition on $\CC/\NN$ \glossp{def:StabQ}
\end{compacthang}
\[
\begin{tikzcd}
  \Stab{\CC} \ar[hook]{r} & \StabL{\CC} \ar[hook]{r} \ar[twoheadrightarrow]{d} & \LaxStab{\CC} \ar[hook]{r} &
  \Slice{\CC} \times \Hom{\Lambda}{\CC} \ar{d}{\cm} \\
  & \StabQ{\CC} & & \Hom{\Lambda}{\C}
\end{tikzcd}
\]

\begin{compacthang}
\item $\LaxStabN{\CC}{\NN}$ the set of lax stability conditions with massless subcategory $\NN$ \glossp{def:space of lax stability conditions}
\item $\StabL{\CC}^* = \StabL{\CC} \setminus \{ (P,0) \in \StabL{\CC} \}$, the entirely massless stratum removed \glossp{sec:orbit closures}
\item $\PStab{\CC} = \Stab{\CC} / \C$ and $\PStabL{\CC} = \big(\StabL{\CC} \setminus \cm^{-1}(0)\big) / \C$
   \glossp{rem:projective stability space} and p~\pageref{sec:orbit closures}
\item Analogously $\StabLN{\CC}{\NN}$ and $\StabQN{\CC}{\NN}$ and $\PStabQ{\CC}$.
\end{compacthang}

\subsection*{Analysis}

\begin{compacthang}
\item $\norm{-}_{\sigma,\delta} \colon \Hom{\Lambda}{\C} \to [0,\infty]$, generalised semi-norms ($\delta\geq0$) \glossp{subsec:semi-norms}
\item $B_\epsilon^\delta(\sigma) \coloneqq \{ (Q,W) : d(P,Q)<\epsilon \ \text{and}\ \norm{W-Z}_{\sigma,\delta} < \sin(\pi \epsilon) \} \subset \Slice{\CC}\times \Hom{\Lambda}{\C}$ \glossp{subsec:semi-norm neighbourhoods}
\end{compacthang}

\subsection*{The maps}

\begin{compacthang}
\item $\massive{\NN} \colon \StabLN{\CC}{\NN} \to \Stab{\CC/\NN}$, the map sending a lax stability condition $(P,Z)$ with massless subcategory $\NN$ to the massive stability condition $\massive{\NN}(P,Z) = (P_{\CC/\NN},Z)$ on the quotient. It extends to a continuous map $\massive{\NN} \colon \overline{\StabLN{\CC}{\NN}} \to \StabL{\CC/\NN}$
  \glossp{prop:massive part} and p~\pageref{lem:rho and mu restrict}
\item $\restrict{\NN} \colon \StabLN{\CC}{\NN} \to \StabLN{\NN}{\NN}$, $(P,Z) \mapsto (P\cap\NN, 0)$, the restriction map \glossp{lem:rho and mu restrict}
  \newline
  It extends locally to $\restrict{\NN} \colon B_\epsilon^\delta(\sigma) \cap \StabL{\CC} \to \StabL{\NN}$ by \cref{lem:rho and mu restrict}.
\item $V(\CC,\NN) \subset \StabL{\CC}$ \defn{decoupling neighbourhood} of $\StabLN{\CC}{\NN}$ \glossp{def:open nbhd U}
\item $U(\CC,\NN) \subseteq \StabLN{\CC}{\NN} \times_{\Slice{\NN}} \StabL{\NN}$ \glossp{def:open nbhd U}
\item The local stratum-wise product decomposition of \cref{thm:StabL structure}:
\[
\begin{tikzcd}[column sep=large]
  \StabLN{\CC}{\NN} \ar[equals]{r} \ar[hook]{d}
& \StabLN{\CC}{\NN} \ar{d}{\id \times \restrict{\NN}} \ar{r}{\massive{\NN}\times \restrict{\NN}}
& \Stab{\CC/\NN} \times \StabLN{\NN}{\NN} \ar[hook]{d}
\\
 V(\CC,\NN)
& U(\CC,\NN) \ar[swap,hook']{l}{\deform{\NN}} \ar[hook]{r}{\massive{\NN}\times\id}% was: \StabLN{\CC}{\NN} \times_{\Slice{\NN}} \StabL{\NN} 
& \Stab{\CC/\NN}\times \StabL{\NN}
\end{tikzcd}
\]
\end{compacthang}

\addtocontents{toc}{\protect\setcounter{tocdepth}{1}}

\newpage

\section{Introduction}

%set theorem counter to capital roman letters
\renewcommand\thetheorem{\Alph{theorem}}

\noindent
The space of stability conditions on a non-zero triangulated category is non-compact when non-empty: the mass of an object may tend to zero or infinity, and the phases of objects may tend to infinity. As such, the study of extensions of spaces of stability conditions has become an area of intensive investigation, see for example \cite{BDL20,Bolognese23,BMS24,HJR24,HR25,Liu26}. We introduce two  such extensions or `partial compactifications': the {\em lax closure of the stability space} and the {\em quotient stability space}. We compare these with some of the above cited approaches in \cref{sec:comparisons}. Our motivation stems from the following features:
\begin{itemize}
\item neighbourhoods of boundary strata in the lax closure have a simple product structure (see \cref{thm:A} below), which provides useful information about the boundary of the stability space;
\item the quotient stability space is always contractible (see \cref{thm:B} below), so may be a useful stepping stone in establishing the conjectured contractibility of stability spaces;
\item in the rank $2$ case, there is a close connection between our boundary strata and the walls in the wall-and-chamber structure of the stability space (see \cref{thm:C} below); and,
\item both our partial compactifications carry the same group actions as the stability space.
\end{itemize}
Indeed, the second feature constitutes our primary motivation: we sought a better geometric setting for our previous proofs that certain stability spaces are contractible (see \cite{BPP16,QW18}).
We anticipate that the connections between the boundary strata and wall-and-chamber structure can be generalised to higher rank cases.

In this article, we extend the stability manifold by allowing the mass of a semistable object to be zero while it still has a well-defined phase.
Doing so, we add boundary strata parametrised by the thick subcategories of massless objects.
The points of these boundary strata can be interpreted as stability conditions on the quotient categories. 
In some sense, our approach is the closest to Bridgeland's original definition of stability condition and therefore the topology of our extension maintains the closest direct link to that of the stability manifold.
However, the extension comes with a cost: our extension loses its complex manifold structure, but interesting topology is maintained as a `stratified space', and generalisations of  Bridgeland's key deformation theorem hold.
Our approach of generalising stability conditions by allowing massless objects generalises the `very weak stability conditions' of \cite{BMS16}, \cite{PT} and is closely related to the `weak stability conditions' of \cite{CLSY}; see \cref{subsec:very weak}.

\subsection{Overview}

Let $\kk$ be a field and $\CC$ be a Hom-finite $\kk$-linear triangulated category. Let $v\colon K(\CC) \to \Lambda$ be a surjective homomorphism from its Grothendieck group onto a lattice, \ie a free abelian group of finite rank. Write $\Slice{\CC}$ for the space of locally finite slicings of $\CC$. 
The space of stability conditions $\Stab{\CC}$ is the subspace of $\Slice{\CC}\times \Hom{\Lambda}{\C}$ consisting of pairs $(P,Z)$ with $Z(c) \in \R_{>0}e^{i\pi \phi}$ whenever $0\neq c\in P(\phi)$, and which satisfy the \defn{support property} (see \cref{subsec:stability conditions}). We will often call these \defn{strict} stability conditions to distinguish them from the following lax generalisation.

\subsubsection*{The lax closure}

A \defn{lax pre-stability condition} consists of a pair $(P,Z)$ with $Z(c) \in \R_{\geq 0}e^{i\pi \phi}$ whenever $c\in P(\phi)$, \ie semistable objects may be \defn{massless} ($Z(c)=0$). 
It is a \defn{lax stability condition} if it additionally satisfies a modified version of the support property. 
Because uniform mass dilation is permitted, allowing masses to tend to zero is equivalent to allowing masses to tend to infinity.

The full subcategory $\NN$ of massless objects for a lax stability condition $(P,Z)$, called the \defn{massless subcategory}, is thick. 
There is an induced stability condition on the quotient $\CC/\NN$ with the same charge $Z$ and for which the semistable objects of phase $\phi$ are those in the isomorphism closure of $P(\phi)$ in $\CC/\NN$. 
More precisely, the induced stability condition has charge $Z$ considered as an element of the subspace $\Hom{\Lambda/\Lambda_\NN}{\C} \subset \Hom{\Lambda}{\C}$ where $\Lambda_\NN$ is the saturation of the subgroup $\{ v(c) \mid c\in \NN\}$ of $\Lambda$. It satisfies the classic support property with respect to the homomorphism $K(\CC/\NN) \to \Lambda/\Lambda_\NN$ induced from $v$. 
In particular, a lax stability condition can be thought of as a compatible pair of a locally finite slicing on a thick subcategory $\NN$ and a stability condition on $\CC/\NN$.
This provides one approach to resolving the technical difficulties in glueing stability conditions in \cite{CP10}.

As we are interested in the boundary of $\Stab{\CC}$, we further restrict to the \defn{lax closure}, $\StabL{\CC}$, which is the subset of the closure of $\Stab{\CC}$ in $\Slice{\CC}\times \Hom{\Lambda}{\C}$ consisting of lax stability conditions.
To each thick subcategory $\NN$ corresponds the \defn{stratum} $\StabLN{\CC}{\NN}$, the subset of pairs $(P,Z) \in \StabL{\CC}$ with massless subcategory $\NN$.
The extreme cases $\NN=0$ and $\NN=\CC$ correspond to the strict stability manifold  and the set of locally finite slicings  arising as limits of slicings of strict stability conditions.
Each stratum fibres over $\Stab{\CC/\NN}$ with the fibres encoding massless phases.
The following theorem encapsulates the structure of $\StabL{\CC}$; see \cref{thm:StabL structure summary} for more details.

\begin{theorem}
\label{thm:A}
The lax closure $\StabL{\CC}$ of $\Stab{\CC}$ is a topological space with a decomposition 
\[
\StabL{\CC} = \bigsqcup_{\NN \in \Thick{\CC}} \StabLN{\CC}{\NN}
\]
into (possibly empty) strata $\StabLN{\CC}{\NN}$ indexed by thick subcategories $\NN \subseteq \CC$. There is a continuous map $\cm \colon \StabL{\CC}\to\Hom{\Lambda}{\CC}$ sending $\StabLN{\CC}{\NN}$ to $\Hom{\Lambda/\Lambda_\NN}{\C}$. 

 The strata $\StabLN{\CC}{\NN}$ are locally closed subsets of $\StabL{\CC}$, with $\StabLN{\CC}{0} = \Stab{\CC}$ being open and dense. The set of strata forms a poset with height bounded by $\rk(\Lambda)$ under 
\[
\StabLN{\CC}{\NN} \leq \StabLN{\CC}{\MM} 
\iff \StabLN{\CC}{\NN} \cap \overline{\StabLN{\CC}{\MM}} \neq \emptyset
\iff \StabLN{\NN}{\MM} \neq \emptyset.
\]
The stratum $\StabLN{\CC}{\NN}$ has an almost-stratum-preserving deformation retract neighbourhood homeomorphic to an open subset of $\Stab{\CC/\NN} \times \StabL{\NN}$ via a homeomorphism mapping points in $\StabL{\CC,\MM}$ to $\Stab{\CC/\NN}\times \StabL{\NN,\MM}$.
\end{theorem}

To interpret the final sentence in \cref{thm:A}, first decompose a lax stability condition with massless subcategory $\NN$ into its massive and massless part, on $\CC/\NN$ and $\NN$ respectively, leading to an open embedding
$\StabLN{\CC}{\NN} \inj \Stab{\CC/\NN} \times \StabLN{\NN}{\NN}$. Then we claim that this product decomposition extends to a neighbourhood of the stratum $\StabLN{\CC}{\NN}$. Intuitively, the massive and massless parts of the theory `de-couple' and can be treated independently of each other when objects in $\NN$ have sufficiently small mass.

In the case that $\CC$ has only hearts with finitely many indecomposable objects, the lax closure is the the closure of the stability manifold in $\Slice{\CC} \times \Hom{K(\CC)}{\C}$ 
(see \cref{prop:boundary support}). However, this is not true in general, and even fails in codimension one --- see \cref{subsec:nilrep}.

The union of all strata of (real) codimension at most one  is a manifold with boundary. There is a boundary stratum for each massless subcategory $\NN$ with $\rk \Lambda_\NN=1$. It is an open subset of $\Stab{\CC/\NN} \times \R$ consisting of pairs of an induced stability condition on $\CC/\NN$ and a `massless phase'. The latter is the common phase of the non-zero objects in a length heart  in $\NN$. It must be compatible with the slicing of the induced stability condition; for each induced stability condition there is a (possibly empty) open subset of compatible massless phases.

\subsubsection*{Quotient stability spaces}

The \defn{quotient stability space} $\StabQ{\CC}$ is the topological quotient of the lax closure obtained by forgetting the phases of massless objects. More precisely, two points of $\StabL{\CC}$ are identified if they have the same massless subcategory $\NN$ and induce the same point of $\Stab{\CC/\NN}$.
It is again a stratified space but has a nicer structure than the lax closure --- strata are complex manifolds and the frontier condition holds. 
It is contractible, and boundary points can be interpreted as stability conditions on certain quotient categories of $\CC$. The following theorem encapsulates its structure; see \cref{thm:stabq structure} for more details.

\begin{theorem}
\label{thm:B}
The quotient stability space $\StabQ{\CC}$ is a topological space with a decomposition into (possibly empty) strata $\StabQN{\CC}{\NN}$ indexed by thick subcategories $\NN \subseteq \CC$:
\[
\StabQ{\CC} = \bigsqcup_{\NN \in \Thick{\CC}} \StabQN{\CC}{\NN} .
\]
There is a continuous map $\cm \colon \StabQ{\CC}\to \Hom{\Lambda}{\C}$ with discrete fibres mapping $\StabQN{\CC}{\NN}$ to $\Hom{\Lambda/\Lambda_\NN}{\C}$. The space $\StabQ{\CC}$ is contractible: uniform mass rescaling defines an almost-stratum-preserving deformation retraction of $\StabQ{\CC}$ onto $\StabQN{\CC}{\CC} \cong \{*\}$. 

The stratum $\StabQN{\CC}{\NN}$ is locally closed, canonically embeds as an open subset of $\Stab{\CC/\NN}$, is a complex manifold of dimension $\rk(\Lambda/\Lambda_\NN)$, and $\StabQN{\CC}{0} = \Stab{\CC}$ is open and dense. The decomposition satisfies the frontier condition, \ie the closure of a stratum is a union of strata, and the resulting poset of strata ordered by inclusion of closures has height at most $\rk(\Lambda)$.
\end{theorem}

Unfortunately, we do not know whether taking the quotient of $\StabL{\CC}$ respects the local product structure required to have decoupling of the massive and massless parts of the theory for quotient stability spaces, see \cref{rem:quotient-decoupling}. Even the situation when $\rk \Lambda_\NN=1$ can be subtle: each point of the associated stratum corresponds to an induced stability condition on $\CC/\NN$ and the topology in a neighbourhood depends on the subset of compatible massless phases for objects of $\NN$ --- see \cref{sec:codim1}.

We consider stability spaces of smooth projective curves to illustrate the restrictions we impose on boundary points; see also \cref{subsec:dense phase case}. 
The stability space $\Stab{X}$ of a strictly positive genus smooth complex projective curve $X$ is isomorphic to $\U\times \C$, where $\U$ is the strict upper half-plane in $\C$ (\cite[Thm.~2.7]{Macri07}). The boundary points $\Q\times \C$ where the masses of line bundles vanish do not give lax stability conditions because the slicings do not converge as we approach them. Nor do the points $(\R-\Q)\times \C$ appear because no objects become massless at these points.
Here $\Stab{X} = \StabL{X} = \StabQ{X}$, and our constructions don't add anything. This shows that any extension of the stability space including points where our lax support property fails is likely to have rather complicated local geometry. 
It would be pleasant to have a partial compactification including the points $\Q\times \C$, but the topology on it would have to allow for the slicings to vary discontinuously. Our techniques rely on the convergence of slicings so prohibit consideration of such boundary points.

In contrast, in the genus zero case, there is one boundary stratum in $\StabL{\PP^1}$ for each $n\in\Z$, corresponding to the vanishing mass of the line bundle $\mathcal{O}(n)$, and finally the deepest stratum where the charge vanishes entirely; see \cref{subsec:projective line} for details.

\subsubsection*{The rank $2$ case}

The stability manifold $\Stab{\CC}$ comes equipped with natural actions of the universal cover $G$ of $\GL$ and of $\Aaut{\Lambda}{\CC}$. These actions extend continuously to both the lax closure and the quotient stability space, see \cref{subsec:group actions}.
In \cref{sec:orbit closures} we describe the closure of the $G$-orbit of a stability condition $\sigma$ in $\StabQ{\CC}$ in terms of the phase diagram of $\sigma$, \ie the set of `occupied' phases for which there is a non-zero semistable object. 
This is a key ingredient of \cref{sec:2d case} in which we illustrate our results in various rank two examples.

The $G$-action preserves semistable objects, hence so does the free action of the subgroup $\C$. Thus the wall-and-chamber structure of $\Stab{\CC}$ descends to a partition of the \defn{projective stability space} $\PStab{\CC} = \Stab{\CC}/\C$ into open chambers and real codimension one walls between them.

Suppose $\CC$ contains a length heart with exactly two simple objects up to isomorphism.
We call the images of free $G$-orbits in $\PStab{\CC}$ \defn{cells}; these are biholomorphic to $G/\C \cong \U$.
The images of non-free $G$-orbits consist of single points and the heart in such an orbit consists of semistable objects of a single phase, and so is a length abelian category, which by assumption has two simple objects $s$ and $t$, say. There is a one-parameter family consisting of images of orbits in which $s$ and $t$ are semistable of the same phase, parametrised by the ratio $m(s)/m(t) \in \R_{>0}$, which defines a real analytic curve in $\PStab{\CC}$, which we call a \defn{cell-wall}.
Thus, in the rank two case, we obtain a refinement of the wall-and-chamber structure, which encapsulates the HRS tilting theory and provides a relationship with the boundary strata of the \defn{projective quotient stability space}, $\PStabQ{\CC}$.
The following theorem summarises the structure and relationship, for more details and illustrations see \cref{sec:2d case}.
This description is of independent interest: note that (1) and (2) below are statements about the Bridgeland stability space $\PStab{\CC}$.

\begin{theorem} \label{thm:C}
Suppose $\CC$ is a triangulated category containing a length heart with exactly two simple objects up to isomorphism. Then
\begin{enumerate}
\item (\cref{lem:cell-walls-algebraic-hearts}) There is a one-to-one correspondence between the cell-walls of $\PStab{\CC}$ and the length hearts of $\CC$ up to shift.
\item (\cref{lem:linear-chains}) Each chamber of $\PStab{\CC}$ is a linear chain of cells separated by cell-walls. If the chain is doubly infinite, the chamber is biholomorphic to $\C$; otherwise, it is biholomorphic to $\U$.
\item (\cref{prop:massless stables = algebraic simples}) The massless stable objects in $\StabL{\CC} \setminus \StabLN{\CC}{\CC}$ are precisely those which are simple in a length heart whose (right) simple HRS tilt is also length. 
\item (\cref{cor:rank 2 boundary points}) The boundary points of $\PStabQ{\CC}$ are in bijection with the massless stable objects, up to isomorphism and shift.
\end{enumerate}
\end{theorem}

To interpret the final statement, the cell-wall parametrised by $m(s)/m(t) \in \R_{>0}$ has two boundary points corresponding to $m(s) = 0$ and $m(t) = 0$.

To see the relationship with HRS tilting theory, recall that the \defn{exchange graph} has one vertex for each length heart in $\CC$ and an edge whenever two hearts are related by a(n irreducible) simple HRS tilt (see \cite{HRS96,KQ15,Woolf10}). The \defn{projective exchange graph} is obtained by taking the quotient of the exchange graph under the action of the shift. This quotient has one vertex for each length heart up to shift and thus for each cell-wall, and one edge for each cell containing two cell-walls. As such, there is a natural embedding of the projective exchange graph in $\PStab{\CC}$.

\subsection{Methods}

To describe the local structure of the lax closure, we generalise the deformation theorem for stability conditions to the lax setting. This requires us to study slicings glued from a slicing on a thick subcategory and on its quotient, which may be of independent interest; cf.\ \cite{CP10,GKR04}. 

We say a pair $(P_\NN,P_{\CC/\NN})$ of slicings of $\NN$ and $\CC/\NN$ is compatible with a slicing $P$ of $ \CC$ if $P_\NN(\phi) \subseteq P(\phi) \subseteq P_{\CC/\NN}(\phi)$ as full subcategories of $\CC$ for each $\phi\in \R$. The key result we use to construct deformed slicings is the following.

\begin{theorem}
\label{thm:D}
Let $\NN$ be a thick subcategory of $\CC$. Then:
\begin{enumerate}
\item (\cref{cor:uniqueness of compatibility}) There is at most one slicing $P$ compatible with each pair $(P_\NN,P_{\CC/\NN})$.
\item (\cref{prop:quotient slicing} and \cref{cor:compatible slicings}) A slicing $P$ is compatible with a pair if and only if the HN-factors of any object of $\NN$ are in $\NN$ and the intersection $P(I)\cap \NN$ is a Serre subcategory of $P(I)$ for each interval $I\subset \R$ of strict length one.
\item (\cref{prop:glueing slicings}) The subset of pairs in $\Slice{\NN}\times \Slice{\CC/\NN}$ compatible with some slicing in $\Slice{\CC}$ is open. 
\end{enumerate}
\end{theorem}

This result enables the following generalisation of Bridgeland's deformation theorem for stability conditions.
The proof is quite technical. One issue is that our notion of lax support is tested on massive objects of sufficiently small HN-width, rather than just on semistable objects as in the strict notion; see \cref{def:lax stability condition}. Moreover, we have no control over the number of Jordan--H\"older factors of semistable objects because slices are quasi-abelian categories. We address this by making the assumption that $\CC$ is Hom-finite, see \eg \cref{lem:lax and strict support}. We do not assume that thick quotients of $\CC$ are Hom-finite, and indeed they need not be. 
The following combines the more detailed \cref{prop:deformations in boundary} and \cref{prop:deformation off stratum}.
\begin{theorem}
\label{thm:E}
Let $(P,Z) \in \StabL{\CC,\NN}$ be a lax stability condition. Then given 
\begin{enumerate}
\item a charge $W$ sufficiently close to $Z$ and 
\item a lax stability condition $(Q_\NN,W|_\NN) \in \StabL{\NN,\MM}$ with $Q_\NN$ sufficiently close to $P|_\NN$
\end{enumerate}
there is a unique lax stability condition $(Q,W) \in \StabL{\CC,\MM}$ with $d(P,Q)<1$ and $Q|_\NN=Q_\NN$. 
\end{theorem}

When $\NN=0$, hence also $\MM=0$, there is a unique choice of $(Q_\NN,W|_\NN)$ because $\StabL{0,0}$ is a point. In this case \cref{thm:D} reduces to Bridgeland's theorem that deformations of the charge lift uniquely to deformations of stability conditions. More generally, when $\NN\neq 0$ one has to specify not only the deformation of the charge but also of the slicing on $\NN$, since the latter is not governed by the charge. This extra information is provided by the choice of lax stability condition $(Q_\NN,W|_\NN)$ on $\NN$.

\subsubsection*{Discrepancies between strict and lax stability conditions}

In many ways lax stability conditions behave much as strict stability conditions do. However, in some respects lax stability conditions have weaker categorical and analytical properties than strict stability conditions. We highlight the differences. Let $\sigma = (P,Z)$ be a pair of a slicing $P$ on $\CC$ and a charge $Z\in\Hom{\Lambda}{\C}$. 
\begin{enumerate}
\item The definitional distinction is that the mass of a non-zero semistable object $c\in P(\phi)$ is required to be positive in the strict case but can be zero if $\sigma$ is lax.
\item The support property for a stability condition $\sigma$ implies that the slicing $P$ is locally finite. By contrast, the support property for a lax stability condition $\sigma$ does not imply local finiteness of the slicing, which we therefore impose as a separate condition. 
Another difference is that the lax support property tests over objects with small HN-width rather than semistable objects only.
\item The slices $P(\phi)$ are always abelian length categories if $\sigma$ is a strict stability condition. If $\sigma$ is a lax stability condition then we only know that $P(\phi)$ is a quasi-abelian length category; in particular, we don't know whether the Jordan--H\"older property holds.
\item A strict stability condition $\sigma$ induces a norm $\norm{\cdot}_\sigma$ on $\Hom{\Lambda}{\C}$ whereas a lax stability condition induces a family of semi-norms $\norm{-}_{\sigma,\delta}$, indexed by$\delta\geq0$.
\item Finally, the space of strict stability conditions $\Stab{\CC}$ is a complex manifold modelled on $\Hom{\Lambda}{\C}$. The space of lax stability conditions $\StabL{\CC}$ is a stratified space. 
\end{enumerate}

\begin{table}
\begin{center}
\begin{tabular}{llcc} \toprule
    &                                 & Stability conditions    & Lax stability  conditions  \\ \midrule
(1) & Masses are                      & positive;               & non-negative. \\
(2) & Slicings are locally finite     & automatically;          & by additional condition. \\
(3) & Slices $P(\phi)$ are            & abelian categories;     & quasi-abelian categories. \\  
(4) & Analysis on $\Hom{\Lambda}{\C}$ & a norm $\norm{-}_\sigma$ & family of semi-norms $\norm{-}_{\sigma,\delta}$. \\
(5) & Geometric structure:            & complex manifold;       & stratified space.\\ \bottomrule
\end{tabular}
\end{center}
\caption{Summary of discrepancies between ordinary and lax stability conditions} \label{table:discrepancies}
\end{table}

\subsection*{Acknowledgments}
We are indebted to Arend Bayer who kindly provided a counterexample to a claim in a previous version.
We are grateful to the London Mathematical Society and the Mathematisches Forschungsinstitut Oberwolfach for financial support through their `Research in Pairs' schemes, grants no.\ 41434 and 1815p. We thank Lasse Rempe for helpful discussions about Riemann surface theory and Maria Azam, Stefan Roberts, Antonios-Alexandros Robotis and the anonymous referee for comments.
The second named author was supported by EPRSC grant no.\ EP/V050524/1, which also supported research visits by the other authors.

% reset theorem counter
\renewcommand\thetheorem{\arabic{section}.\arabic{theorem}}

\section{Notation and preliminaries}
\label{sec:notation}

\subsection{Quasi-abelian categories}
\label{subsec:quasi-abelian}

It has been known since Tom Bridgeland's original article \cite{Bridgeland08} that quasi-abelian categories are important in the theory of stability conditions. In this text they appear even more prominently because slices of lax stability conditions are in general not abelian categories (as with stability conditions) but only quasi-abelian; see slicing $P_t$ in \cref{ex:slicings}. 

A \defn{quasi-abelian category} is an additive category with kernels and cokernels and such that the pullback of a strict epimorphism is a strict epimorphism, and the pushout of a strict monomorphism is a strict monomorphism. A \defn{strict morphism} is one for which the canonical morphism from its coimage to its image is an isomorphism. A quasi-abelian category is \defn{length} if it satisfies the bi-chain condition, \ie is both artinian and noetherian. See \eg \cite{Schneiders99}.

\begin{example}[{\cite[Exs.~3.5 and 6.9]{BHLR20}}] \label{ex:non-unique decomposition}
Let $\cE$ be the full $\kk$-linear subcategory of finite-dimensional $\kk$-vector spaces generated by $\kk^2$ and $\kk^3$. In this example, any non-zero map $\kk^2\to\kk^3$ has kernel and cokernel $0$. Therefore the coimage $\kk^2$ and image $\kk^3$ are not isomorphic and the morphism is not strict. In particular, $\kk^2$ and $\kk^3$ are simple objects of $\cE$, and $\cE$ is a length quasi-abelian category. However, the Jordan--H\"older property fails: $\kk^6 = \kk^2\oplus\kk^2\oplus\kk^2 = \kk^3\oplus\kk^3$ has two Jordan--H\"older filtrations of different lengths, with non-isomorphic factors. Likewise, the Krull--Schmidt property fails because of the non-unique decompositions.
\end{example} 

As the above example shows, decompositions into indecomposable objects exist in arbitrary length quasi-abelian categories but may lack the Krull--Schmidt properties. However, when the length quasi-abelian category arises as in the next section from a slicing of a Hom-finite triangulated category we recover the Krull--Schmidt property, see \cref{prop:KRS property}.

\subsection{Slicings}
\label{subsec:slicings}

Let $\CC$ be a triangulated category with shift functor $c\mapsto c[1]$. A \defn{slicing} $P$ on $\CC$ is a collection of full additive subcategories $P(\phi)$ for each $\phi \in \R$ such that
\begin{enumerate}
\item $P(\phi+1) = P(\phi)[1]$ for all $\phi\in \R$;
\item $\Homm{\CC}{c}{c'}=0$ whenever $c\in P(\phi)$ and $c'\in P(\phi')$ with $\phi > \phi'$;
\item each $0\neq c\in \CC$ admits a finite filtration \ie a finite sequence of morphisms
\[
\begin{tikzcd}
0=c_0 \ar{r} & c_1 \ar{r} \ar{d}        & c_2 \ar{r} \ar{d}       & \cdots \ar{r} & c_{n-1} \ar{r} & c_n = c \ar{d}\\
                    & a_1 \ar[dashed]{ul}  & a_2 \ar[dashed]{ul} &                    &                       & a_n \ar[dashed]{ul}
\end{tikzcd}
\]
with cones $a_i \in P(\phi_i)$ where $\phi_1 > \phi_2 > \cdots > \phi_n$.
\end{enumerate}
The subcategory $P(\phi)$ is called the \defn{slice of phase $\phi$}, its objects are \defn{semistable of phase $\phi$}, the filtration is the \defn{Harder--Narasimhan filtration} (henceforth abbreviated to HN filtration) of $c$, and the objects $a_i$ are the \defn{semistable factors} of $c$. The filtration, in particular the semistable factors, are determined uniquely up to isomorphism when they exist. The \defn{maximal and minimal phases} of $0\neq c \in \CC$ are $\phi^+(c) \coloneqq \phi_1$ and $\phi^-(c) \coloneqq \phi_n$, respectively. 

For any interval $I \subseteq \R$, we denote by $P(I)$ the full subcategory of $\CC$ on those objects whose semistable factors have phases in $I$. When $I=(a,b)$ we omit the outer brackets and simply write $P(a,b)$; likewise for other types of intervals. Equivalently, $P(I)$ is the extension closure of the slices $P(\phi)$ for all $\phi\in I$.

An interval $I\subset \R$ is said to have \defn{strict length one} if either $I = (\phi,\phi+1]$ or $I = [\phi,\phi+1)$ for some $\phi\in \R$.
The category $P(I)$ is abelian when $I$ is an interval of strict length one (see below) and quasi-abelian when $I$ is contained in such an interval \cite[\S4]{Bridgeland07}. 

A \defn{stable} object is a semistable object that is simple in its slice, \ie a semistable object of some phase $\phi$ with no proper strict subobjects in the quasi-abelian category $P(\phi)$. 

For any $\phi\in \R$, the inclusions of $P(-\infty, \phi)$ and $P(-\infty,\phi]$ into $\CC$ have respective left adjoints $H_P^{<\phi}$ and $H_P^{\leq \phi}$. Dually, the inclusions of $P(\phi,\infty)$ and $P[\phi,\infty)$ into $\CC$ have respective right adjoints $H_P^{>\phi}$ and $H_P^{\geq \phi}$. We use the notation 
\[
  H_P^{[\phi,\psi]} \coloneqq H_P^{\geq\phi} \circ H_P^{\leq\psi} = H_P^{\leq\psi} \circ H_P^{\geq\phi}
  \qquad\text{and}\qquad
  H_P^\phi \coloneqq H_P^{[\phi,\phi]}
\] 
and similarly for open and semi-open intervals.
For $c\in\CC$ the $P$-semistable factors of $H_P^{[\phi,\psi]}(c)$ are precisely the $P$-semistable factors of $c$ with phases in the interval $[\phi,\psi]$.
The functor $H_P^\phi \colon \CC \to P(\phi)$ maps an object to its semistable factor of phase $\phi$, if present, and to zero else.

When $I$ is an interval of strict length one, the subcategory $P(I)$ is the heart of a bounded t-structure on $\CC$ and $H^I_P \colon \CC \to P(I)$ is the associated cohomological functor taking triangles in $\CC$ to long exact sequences in $P(I)$. 

\begin{remark}
The right adjoint to the inclusion $P(0,\infty) \inj \CC$ is the truncation \emph{below} associated to the bounded t-structure on $\CC$ with heart $P(0,1]$, \ie it is the functor classically denoted $\tau^{\leq 0}$. This unfortunate clash of notation arises because the factors in a HN filtration are ordered by decreasing phase. To avoid confusion we use the notation $H^{>0}_P$ instead.
\end{remark}

If $Q$ is another slicing with $Q(-\infty,\phi) \subseteq P(-\infty,\psi)$ then 
$H_Q^{<\phi} = H_Q^{<\phi} H_P^{<\psi}$
and similarly for closed intervals, and the dual cases. We omit the subscript $P$ from the notation when the slicing is understood from the context. 

The slicing $P$ is \defn{locally finite} if there is some $\epsilon>0$ such that $P(\phi-\epsilon,\phi+\epsilon) $ is a length quasi-abelian category for each $\phi\in \R$. Let $\Slice{\CC}$ \label{def:space of slicings}denote the space of locally finite slicings on $\CC$, topologised by the metric
\[
d(P,Q) = \sup_{0\neq c\in\CC} \max \left\{ |\phi_P^-(c)-\phi_Q^-(c)|, |\phi_P^+(c)-\phi_Q^+(c)| \right\}.
\] 
For a locally finite slicing each slice $P(\phi)$ is a quasi-abelian length category. It follows that each semistable object has a finite composition series whose factors are stable objects. However, we do not know in general that the set of these stable factors, nor the multiplicity with which each occurs, are well defined --- see \cite{Enomoto21} for a discussion of when a quasi-abelian category satisfies the Jordan--H\"older Theorem. 

We will often assume that the category $\CC$ is Hom-finite $\kk$-linear over a field $\kk$.
Then the coarser additive decomposition into indecomposable objects is essentially unique; more precisely the Krull--Schmidt property holds:

\begin{proposition}
\label{prop:KRS property}
Let $P$ be a slicing of a Hom-finite $\kk$-linear triangulated category $\CC$ and $I\subset \R$ contained in an interval of strict length one. Then the following hold:
\begin{enumerate}
\item An object $c\in P(I)$ is indecomposable if and only if $\End{\CC}{c}$ is a local algebra.
\item Each object of $P(I)$ decomposes as a finite direct sum of indecomposable objects, and such a decomposition is unique up to permutation of the summands.
\end{enumerate}
\end{proposition}

\begin{proof}
By assumption, we have $P(I) \subseteq P(\phi,\phi+1]$ or $P(I) \subseteq P[\phi,\phi+1)$. Without loss of generality we assume the former. Since $P(\phi,\phi+1]$ is a Hom-finite $\kk$-linear abelian category, each object satisfies the bi-chain condition by \cite[Lem.~5.2]{Krause15}; an object in $P(\phi,\phi+1]$ is indecomposable if and only if its endomorphism algebra is local by \cite[Prop.~5.4]{Krause15}; each object of $P(\phi,\phi+1]$ decomposes as a finite direct sum of indecomposable objects by \cite[Thm.~5.5]{Krause15}; and such a decomposition is unique up to permutation of the summands by \cite[Thm.~4.2]{Krause15}. The same then holds for $P(I)$ because it is a full additive subcategory which is closed under direct summands.
\end{proof}

\begin{remark}
\label{rem:hom-finite quotients}
Hom-finite categories may have quotients which are not Hom-finite, as happens in \cref{ex:thick subcategories}. When we assume $\CC$ is Hom-finite in later sections we do not assume that any of its proper quotients are so.
\end{remark}

\subsection{Charges}
\label{subsec:charges}

Fix a finite rank lattice $\Lambda$ and a surjective homomorphism $v\colon K(\CC) \to \Lambda$. For $Z\in \Hom{\Lambda}{\C}$ and $c\in \CC$ we abuse notation by writing $Z(c)$ for $Z(v([c]))$. 

We also fix an inner product $\langle \cdot , \cdot \rangle$ on $\Lambda \otimes \R$ and denote the associated norm by $\norm{\cdot}$. Again we abuse notation by writing $\norm{c}$ for $\norm{v([c])}$. The induced operator norm on $\Hom{\Lambda}{\C}$ is
\[ \norm{Z} \coloneqq \sup \{ |Z(\lambda)| : \lambda\in \Lambda\otimes\R, \ \norm{\lambda}=1 \} . \]
Throughout this article, we equip the product $\Slice{\CC} \times \Hom{\Lambda}{\C}$ with the metric
\[
d( (P,Z) , (Q,W) ) \coloneqq \max\{ d(P,Q) , \norm{W-Z} \}
\]
made from the slicing metric and the operator norm. All subspaces of $\Slice{\CC} \times \Hom{\Lambda}{\C}$ carry the same metric, and are topologised as metric spaces. The second projection is denoted $\cm \colon \Slice{\CC} \times \Hom{\Lambda}{\C}$, $\sigma = (P,Z) \mapsto \cm(\sigma) = Z$, and likewise for all subspaces of the product.

Suppose $\NN$ is a thick subcategory of $\CC$. The Verdier quotient $\CC/\NN$ is a triangulated category with the same objects as $\CC$ and a morphism $c' \to c$ in $\CC/\NN$ given by a roof $c' \leftarrow c'' \to c$ where the cone of the first morphism $c' \leftarrow c''$ is in $\NN$; see, for example, \cite{GM96, Hartshorne66}.

The homomorphism $K(\NN) \to K(\CC)$ induced from the inclusion $\NN \to \CC$ need not be injective, but its cokernel is $K(\CC/\NN)$. Let $\Lambda_\NN \subset \Lambda$ be the minimal primitive sublattice containing the image of $K(\NN) \to K(\CC) \to \Lambda$, so that $\Lambda/\Lambda_\NN$ is again a lattice. Let $v_\NN \colon K(\NN) \to \Lambda_\NN$ and $v_{\CC/\NN} \colon K(\CC/\NN) \to \Lambda/\Lambda_\NN$ denote the induced homomorphisms. The map $v_\NN$ may not be surjective but it always has finite index.

The norm of $\Lambda\otimes\R$ restricts to a norm on $\Lambda_\NN\otimes \R$ and also induces a norm on the quotient
\[
\Lambda/\Lambda_\NN \otimes \R \cong (\Lambda\otimes \R) / (\Lambda_\NN\otimes \R)
\]
defined by $\norm{ \lambda + \Lambda_\NN\otimes \R } = \inf\{ \norm{ \lambda+\alpha } \colon \alpha \in \Lambda_\NN \otimes \R \}$. 
Alternatively, this is given by identifying $\Lambda/\Lambda_\NN \otimes \R$ with the orthogonal complement of $\Lambda_\NN\otimes \R$ and taking the restriction of $\norm{\cdot}$.
 
\begin{remark}
\label{rmk:charge space splitting}
The orthogonal projection $\Lambda\otimes \R \to \Lambda_\NN\otimes \R$ induces a splitting $\Hom{\Lambda_\NN}{\C} \inj \Hom{\Lambda}{\C}$ and we use this throughout to identify $\Hom{\Lambda_\NN}{\C}$ with its image in $\Hom{\Lambda}{\C}$.
\end{remark}

\subsection{Spaces of stability conditions}

We work with stability conditions on $\CC$ whose charges factor through $v \colon K(\CC) \to \Lambda$ and satisfy the support property. (See \cref{subsec:stability conditions} for the definition.) We denote the space of these by $\Stab{\CC}$, leaving the lattice $\Lambda$ implicit. 

Stability conditions on thick subcategories $\NN$ of $\CC$, and on the quotients $\CC/\NN$ by these play a prominent role. The charges of these are always understood to factor through $v_\NN$ and $v_{\CC/\NN}$ respectively. We denote the respective spaces of stability conditions by $\Stab{\NN}$ and $\Stab{\CC/\NN}$, again omitting the lattices from the notation.

\section{Restriction, descent and glueing of slicings}
\label{sec:slicings and thick subcategories}

\noindent
Fix a thick subcategory $\NN \subseteq \CC$. We investigate the relationship between slicings of $\CC$, $\NN$ and $\CC/\NN$. In this section we do not assume that slicings are locally finite unless stated otherwise. 

\begin{definition}
A slicing $P$ of $\CC$ is \defn{compatible} with a pair $(P_\NN,P_{\CC/\NN})$ of slicings of $\NN$ and $\CC/\NN$ if there are inclusions of objects $P_\NN(\phi) \subseteq P(\phi) \subseteq P_{\CC/\NN}(\phi)$ for each $\phi \in \R$.
\end{definition}

\begin{proposition}
\label{prop:uniqueness of compatibility}
Let $P$ be a slicing of $\CC$ compatible with a pair of slicings $(P_\NN,P_{\CC/\NN})$ of $\NN$ and $\CC/\NN$. Then for any $\phi\in\R$, the following hold:
\begin{enumerate}
\item $P_\NN(\phi)=P(\phi)\cap\NN$ and $P_{\CC/\NN}(\phi)$ is the isomorphism closure of $P(\phi)$ in $\CC/\NN$;
\item $P(\phi) = P_{\CC/\NN} \cap P_\NN({>}\phi)\orth \cap {}\orth P_\NN({<}\phi)$.
\end{enumerate}
\end{proposition}

\begin{corollary}
\label{cor:uniqueness of compatibility}
For a slicing $P$ of $\CC$, there is at most one pair $(P_\NN,P_{\CC/\NN})$ compatible with $P$.

Conversely, given a pair $(P_\NN,P_{\CC/\NN})$, there is at most one slicing $P$ compatible with the pair.
\end{corollary}

\begin{proof} Fix $\phi \in \R$.
\begin{enumerate}[wide, labelwidth=!, labelindent=1em]
\item
By compatibility, $P_\NN(\phi) \subseteq P(\phi)$ so that the HN filtration of any $c\in \NN$ with $P_\NN$-semistable factors is also the HN filtration with $P$-semistable factors. The uniqueness of HN filtrations implies $P(\phi)\cap \NN \subseteq P_\NN(\phi)$, hence $P_\NN(\phi)=P(\phi)\cap \NN$ for each $\phi\in \R$. 

Let $P(\phi)_{\CC/\NN}$ be the isomorphism closure of $P(\phi)$ in $\CC/\NN$. We claim $P_{\CC/\NN}(\phi) = P(\phi)_{\CC/\NN}$. Since $P(\phi) \subseteq P_{\CC/\NN}(\phi)$ it is clear that $P(\phi)_{\CC/\NN} \subseteq P_{\CC/\NN}(\phi)$. Moreover, again by uniqueness, the HN filtration of any $c\in \CC$ with $P$-semistable factors descends to the HN filtration of $c$ with $P_{\CC/\NN}$-semistable factors if we simply ignore any factors in $\NN$. Thus if $c\in P_{\CC/\NN}(\phi)$ it has a HN filtration in $\CC$ with all factors in $\NN$ except for one factor, say $c'$, in $P(\phi)$. Thus $c\cong c'$ in $\CC/\NN$ and $P_{\CC/\NN}(\phi) \subseteq P(\phi)_{\CC/\NN}$, establishing the claim.
\item
The right-hand side of the equation translates to
\[ c\in P(\phi) \iff c\in P_{\CC/\NN}(\phi) \text{ and } \Homm{\CC}{P_\NN({>}\phi)}{c} = 0 = \Homm{\CC}{c}{P_\NN({<}\phi)} . \]
By definition $P(\phi) \subseteq P_{\CC/\NN}(\phi)$. We saw above that any $c\in P_{\CC/\NN}(\phi)$ has a HN filtration in $\CC$ with all factors in $\NN$ apart from one factor in $P(\phi)$. Hence the truncations $H_P^{>\phi}(c)$ and $H_P^{<\phi}(c)$ are in $\NN$ and so vanish precisely when $\Homm{\CC}{P_\NN({>}\phi)}{c} = 0 = \Homm{\CC}{c}{P_\NN({<}\phi)}$ holds. \qedhere
\end{enumerate}
\end{proof}

\begin{lemma}
\label{lem:local-finiteness of compatible slicing}
If a slicing $P$ of $\CC$ is compatible with a pair $(P_\NN,P_{\CC/\NN})$ of locally finite slicings then $P$ is also locally finite.
\end{lemma}

\begin{proof}
Let $I \subset \R$ be an interval such that both $P_\NN(I)$ and $P_{\CC/\NN}(I)$ are quasi-abelian length categories. Let $a_0 \inj a_1 \inj \cdots \inj a$ be an increasing sequence of strict subobjects of $a$ in $P(I)$. 
This can be considered as an increasing sequence of strict subobjects in $P_{\CC/\NN}(I)$ via the quotient functor $\CC \to \CC/\NN$ since the strict exact structures on $P_\NN(I)$ and $P_{\CC/\NN}(I)$ are induced by the triangulated structures on $\CC$ and $\CC/\NN$, respectively, and the quotient functor is exact.
Since $P_{\CC/\NN}(I)$ is length this chain stabilises, \ie there is some $n\in \N$ such that $a_n\cong a_{n+1}\cong \cdots \cong a$ in $P_{\CC/\NN}(I)$. Pushing the sequence of strict monomorphisms out along $a_{n+1}\to a_{n+1}/a_n$ we obtain a commutative diagram in $P_\CC(I)$
\[
\begin{tikzcd}
a_{n+1} \ar[hook]{r} \ar[->>]{d} 	& a_{n+2}\ar[hook]{r} \ar[->>]{d} 	& \cdots \ar[hook]{r} & a \ar[->>]{d}\\
a_{n+1}/a_n \ar[hook]{r} 		& a_{n+2}/a_n\ar[hook]{r} 		& \cdots \ar[hook]{r} & a/a_n
\end{tikzcd}
\]
whose bottom row is an increasing sequence of strict subobjects of $a/a_n$ in $P_\NN(I)$. Since the latter is length this bottom row also stabilises. It follows that the original sequence stabilises so that $P(I)$ is noetherian. The proof that it is artinian is dual.
\end{proof}

In the next sections we discuss the more subtle question of when compatible slicings exist.

\subsection{Restriction and descent}

We say that a slicing $P$ of $\CC$ \defn{restricts} to the thick subcategory $\NN$ if the semistable factors of objects of $\NN$ lie in $\NN$. Then the full subcategories $P_\NN(\phi) \coloneqq P(\phi)\cap \NN$ form a slicing of $\NN$. Deciding when $P$ \defn{descends} to a compatible slicing of $\CC/\NN$ is more involved. 

\begin{definition}
\label{def:adapted slicing}
A slicing $P$ of $\CC$ is \defn{adapted} to a thick subcategory $\NN$ if it restricts to $\NN$ and $P(I) \cap \NN$ is a Serre subcategory of $P(I)$ for each interval $I \subset \R$ of strict length one.
\end{definition}

\begin{lemma}
\label{lem:quotient semistables}
Let $P$ be a slicing of $\CC$ adapted to the thick subcategory $\NN$. The following two conditions are equivalent for objects $c\in P(\phi)$ and $b\in \CC$:
\begin{enumerate}[label=(\roman*)]
\item $b \cong c$ in $\CC/\NN$
\item the only semistable factor of $b$ not in $\NN$ is a factor $b_0\in P(\phi)$ with $b_0\cong c$ in $\CC/\NN$.
\end{enumerate}
\end{lemma}

\begin{proof}
(ii) \timplies (i): Clearly if $b$ has a unique semistable factor $b_0 \in P(\phi)$ not in $\NN$ then $b\cong b_0$ in $\CC/\NN$, and thus if $b_0\cong c$ in $\CC/\NN$ then also $b\cong c$ in $\CC/\NN$.

(i) \timplies (ii): We show the more general statement that, given $c\in \CC$ with $H^{<\phi}(c)$ and $H^{>\phi}(c)$ in $\NN$ and a morphism in $\Homm{\CC}{c}{b}$ or $\Homm{\CC}{b}{c}$ with cone in $\NN$, then $H^{<\phi}(b)$ and $H^{>\phi}(b)$ are also in $\NN$. The cases are similar, and we only consider the first in which there is an exact triangle $n[-1]\to c\to b \to n$ with $n \in \NN$. Applying the cohomological functor $H^{(\phi,\phi+1]}$ yields a long exact sequence in the heart $P(\phi,\phi+1]$
\[
\cdots \to H^{(\phi,\phi+1]}(c) \to H^{(\phi,\phi+1]}(b) \to H^{(\phi,\phi+1]}(n) \to \cdots.
\]
The assumptions on $c$ and $n$ imply that the first and third terms are in $P(\phi,\phi+1]\cap \NN$. Since this is a Serre subcategory of $P(\phi,\phi+1]$ so too is the middle term. For the same reason $H^{(\phi+i,\phi+i+1]}(b) \in \NN$ for all $i\in \N$, which implies $H^{>\phi}(b)\in \NN$. To show that $H^{<\phi}(b)\in \NN$ one proceeds similarly using the cohomological functor $H^{[\phi-1,\phi)}$. 
\end{proof}

\begin{proposition}
\label{prop:quotient slicing}
Let $P$ be a slicing of $\CC$. Then the following conditions are equivalent:
\begin{enumerate}[label=(\roman*)]
\item $P$ is adapted to $\NN$.
\item There is a pair $(P_\NN,P_{\CC/\NN})$ of slicings of $\NN$ and $\CC/\NN$ compatible with $P$.
\end{enumerate}
\end{proposition}

\begin{proof}
(ii) \timplies (i): Given a compatible pair $(P_\NN,P_{\CC/\NN})$, we saw above that $P_\NN(\phi)= P(\phi)\cap \NN$ so that $P$ restricts to $\NN$. Moreover, the quotient functor $\CC\to \CC/\NN$ restricts to an exact functor $P(\phi,\phi+1] \to P_{\CC/\NN}(\phi,\phi+1]$ between abelian categories with $P(\phi,\phi+1]\cap \NN$ as kernel. Hence the latter is a Serre subcategory for each $\phi\in \R$. The argument showing $P[\phi,\phi+1) \cap \NN$ is a Serre subcategory is similar. Thus $P$ is adapted to $\NN$.

(i) \timplies (ii): If $P$ is adapted to $\NN$ then the subcategories $P_\NN(\phi) \coloneqq P(\phi)\cap\NN$ define a slicing of $\NN$. We must show that $P$ also descends to a slicing of $\CC/\NN$. \Cref{prop:uniqueness of compatibility} shows that we must define $P_{\CC/\NN}(\phi)$ to be the closure of $P(\phi)$ under isomorphisms in $\CC/\NN$. By construction $P_{\CC/\NN}(\phi)$ is a full additive subcategory of $\CC/\NN$ for each $\phi\in \R$, satisfying $P_{\CC/\NN}(\phi+1) = P_{\CC/\NN}(\phi)[1]$. Moreover, ignoring any factors in $\NN$, the image in $\CC/\NN$ of the HN filtration of $0\neq c \in \CC$ with respect to the slicing $P$ provides a finite filtration with factors in these subcategories and with strictly decreasing phases. Therefore to show that $P_{\CC/\NN}$ is a slicing we must show that there are no non-zero morphisms in $\CC/\NN$ from $c\in P_{\CC/\NN}(\phi)$ to $c'\in P_{\CC/\NN}(\phi')$ when $\phi > \phi'$. 

It is enough to show that $\Homm{\CC/\NN}{c}{c'}=0$ for $c\in P(\phi)$ and $c'\in P(\phi')$. A morphism in $\Homm{\CC/\NN}{c}{c'}$ is represented by a roof $c \leftarrow b \to c'$ in $\CC$ where the cone on the left hand morphism is in $\NN$. By \cref{lem:quotient semistables}, $H^{>\phi}(b)$ and $H^{<\phi} (b)$ are in $\NN$ and moreover, $\Homm{\CC}{H^{>\phi}(b)}{c} = 0 = \Homm{\CC}{H^{>\phi}(b)}{c'}$ because $c\in P(\phi)$ and $c'\in P(\phi')$ for $\phi'<\phi$. Thus we can construct dashed arrows to form a commutative diagram in $\CC$
\[
\begin{tikzcd}
{} & b\ar{dl} \ar{d} \ar{dr} & {}\\
 c & H^{\leq \phi}(b)  \ar[dashed]{l}  \ar[dashed]{r}         & c' \\
{} & H^\phi(b) \ar{u} \ar[dashed]{ul} \ar[dashed]{ur} \ar{u} & {}
\end{tikzcd}
\]
in which the morphisms $b\to H^{\leq \phi}(b)$ and $H^\phi(b)\to H^{\leq \phi}(b)$ are the canonical ones. These are isomorphisms in $\CC/\NN$. Therefore the bottom roof $c \leftarrow H^\phi(b) \to c'$ represents the same morphism in $\CC/\NN$. Since $\Homm{\CC}{H^\phi(b)}{c'}=0$ we deduce $\Homm{\CC/\NN}{c}{c'}=0$ as required. 
\end{proof}

\begin{corollary}
\label{cor:slicing inequalities}
Let $P$ and $Q$ be slicings on $\CC$ that are adapted to $\NN$, let $P_\NN$, $Q_\NN$ be the restricted slicings of $\NN$ and $P_{\CC/\NN}$, $Q_{\CC/\NN}$ be the descended slicings of $\CC/\NN$. Then
\[ d(P_\NN,Q_\NN) \leq d(P,Q) \qquad\text{and}\qquad d(P_{\CC/\NN},Q_{\CC/\NN}) \leq d(P,Q) . \]
\end{corollary}

\begin{proof}
The HN filtrations of $c\in\NN$ with respect to $P$ and to $P_\NN$ coincide. The first statement follows since the supremum defining $d(P_\NN,Q_\NN)$ is taken over a subset of that defining $d(P,Q)$. The second statement follows because
\begin{align*}
d(P, Q) < \epsilon 
& \iff P(\phi) \subseteq Q(\phi-\epsilon,\phi+\epsilon) & \forall \phi \in \R \\
& \implies P(\phi) \subseteq Q_{\CC/\NN}(\phi-\epsilon,\phi+\epsilon) & \forall \phi \in \R \\
& \iff P_{\CC/\NN}(\phi) \subseteq Q_{\CC/\NN}(\phi-\epsilon,\phi+\epsilon) & \forall \phi \in \R\\
& \iff d(P_{\CC/\NN}, Q_{\CC/\NN}) < \epsilon 
\end{align*}
where we have used the fact that $P_{\CC/\NN}(\phi)$ is the isomorphism closure of $P(\phi)$ in $\CC/\NN$, and similarly for $Q_{\CC/\NN}(\phi)$.
\end{proof}

If $P$ is a locally finite slicing adapted to a thick subcategory $\NN$ then the restriction $P_\NN$ is clearly locally finite. However, we do not know whether the slicing on the quotient $P_{\CC/\NN}$ is locally finite when $P$ is so. Therefore we introduce the following enhancement of \cref{def:adapted slicing}.

\begin{definition}
\label{def:well-adapted slicing}
A locally finite slicing $P$ of $\CC$ is \defn{well adapted} to a thick subcategory $\NN$ if it is adapted to it and the quotient slicing $P_{\CC/\NN}$ is locally finite.
\end{definition}

\begin{corollary}
\label{cor:compatible slicings}
Let $P$ be a slicing of $\CC$. Then the following conditions are equivalent:
\begin{enumerate}[label=(\roman*)]
\item $P$ is locally finite and well adapted to $\NN$.
\item There is a pair $(P_\NN,P_{\CC/\NN})$ of locally finite slicings of $\NN$ and $\CC/\NN$ compatible with $P$.
\end{enumerate}
\end{corollary}

\begin{proof}
If $P$ is locally finite and well adapted to $\NN$ then the restricted slicing $P_\NN$ and quotient slicing $P_{\CC/\NN}$ exist and are locally finite. The slicing $P$ is compatible with them. 

Conversely, if $P$ is compatible with the pair $(P_\NN,P_{\CC/\NN})$ of locally finite slicings then $P$ is locally finite by \cref{lem:local-finiteness of compatible slicing}, it is adapted to $\NN$ by \cref{prop:quotient slicing} and indeed is well adapted since $P_{\CC/\NN}$ is locally finite by assumption.
\end{proof}

We do not know whether being `well adapted' is a genuine condition on a locally-finite slicing, so we have the following question.
A partial answer is given in \cref{prop:massive stability condition}.

\begin{question} \label{q:locally finite}
Let $P$ be a locally finite slicing on $\CC$ and $\NN \subset \CC$ a thick subcategory such that $P$ descends to $\CC/\NN$. Is the slicing $P_{\CC/\NN}$ on the quotient also locally finite?
\end{question}

\subsection{Glueing}

Here we prove that a pair of slicings of $\NN$ and $\CC/\NN$ can be glued to a compatible locally finite slicing on $\CC$ if and only if a sufficiently close pair does. An important consequence is that the set of pairs of slicings which glue to a locally finite slicing is open. 

\begin{proposition}
\label{prop:glueing slicings}
Let $(Q_\NN,Q_{\CC/\NN})$ be a pair of slicings of $\NN$ and $\CC/\NN$. Then the following conditions are equivalent:
\begin{enumerate}[label=(\roman*)]
\item There is a compatible locally finite slicing $Q$ of $\CC$.
\item There is a locally finite slicing $P$ of $\CC$ compatible with a pair $(P_\NN,P_{\CC/\NN})$ such that $d(P_\NN,Q_\NN)<\epsilon$ and $d(P_{\CC/\NN},Q_{\CC/\NN})<\epsilon$ where $\epsilon>0$ is sufficiently small that the categories $P(\phi-2\epsilon,\phi+2\epsilon)$ are length for each $\phi\in \R$.
\end{enumerate}
\end{proposition}
%
% Note (leave in text): This statement also works without `locally finite' in (i) and (ii).

The implication (i) \timplies (ii) is trivial: if there is a compatible locally finite slicing $Q$ then we set $P=Q$ and are done. The proof of the other direction is long, and we break it down into a number of results. We retain the notation and assumptions of the statement throughout this section. By \cref{prop:uniqueness of compatibility} there is at most one choice for the slicing: for $\phi\in \R$ the slice must be the full subcategory
 $Q(\phi) \coloneqq Q_{\CC/\NN}(\phi) \cap Q_\NN({>}\phi)\orth \cap {}\orth Q_\NN({<}\phi)$.
Clearly $Q(\phi+1) = Q(\phi)[1]$ for all $\phi \in \R$. Moreover, $Q(\phi)\cap \NN = Q_\NN(\phi)$ and $Q(\phi) \subseteq Q_{\CC/\NN}(\phi)$. We extend the notation to intervals $I\subset \R$ by defining $Q(I)$ to be the extension-closure of the $Q(\phi)$ for $\phi\in I$. By definition $Q(I) \subseteq Q_{\CC/\NN}(I)$ for any interval $I$.

\begin{lemma}
\label{lem:slicing distance}
For each $\phi \in \R$ we have $Q(\phi) \subseteq P(\phi-\epsilon,\phi+\epsilon)$.
\end{lemma}

\begin{proof}
Suppose $c\in Q(\phi)$. Then as $Q(\phi) \subseteq Q_{\CC/\NN}(\phi)\subseteq P_{\CC/\NN}(\phi-\epsilon, \phi+\epsilon)$ we know that $H_P^{\geq \phi+\epsilon}(c) , H_P^{\leq \phi-\epsilon}(c) \in \NN$. Thus $H_P^{\geq \phi+\epsilon}(c) \in Q_\NN({>}\phi)$ and $H_P^{\leq \phi-\epsilon}(c) \in Q_\NN({<}\phi)$ and so by the definition of $Q(\phi)$ we have
\[
\Homm{\CC}{H_P^{\geq \phi+\epsilon}(c)}{c} = 0 = \Homm{\CC}{c}{H_P^{\leq \phi-\epsilon}(c)}.
\]
It follows that $H_P^{\geq \phi+\epsilon}(c) = 0 = H_P^{\leq \phi-\epsilon}(c)$ so that $c\in P(\phi-\epsilon,\phi+\epsilon)$ as claimed.
\end{proof}

\begin{lemma}
\label{lem:reconstructed slicing}
If $c\in Q(\phi)$ and $c'\in Q(\phi')$ with $\phi > \phi'$ then $\Homm{\CC}{c}{c'}=0$.
\end{lemma}

\begin{proof}
Suppose $\gamma \in \Homm{\CC}{c}{c'}$. Since $c \in Q_{\CC/\NN}(\phi)$ and $c' \in Q_{\CC/\NN}(\phi')$ the morphism $\gamma$ vanishes in $\CC/\NN$. Hence it must factor through some $b\in \NN$. We therefore have a diagram 
\[
\begin{tikzcd}
c \ar{rr}{\gamma} \ar{dr} \ar[dashed]{dd} && c' \\
& b \ar{ur} \ar{dr} &\\
H_{Q_\NN}^{\geq \phi}(b) \ar{ur} && H_{Q_\NN}^{<\phi}(b) \ar{ll}[swap]{[1]}
\end{tikzcd}
\]
in which the upper triangle commutes and the lower triangle is exact. One of the vanishing requirements in the definition of $Q(\phi)$ implies $\Homm{\CC}{c}{H^{<\phi}(b)}=0$, and hence that there is a dashed morphism making the left hand triangle commute. We deduce $\Homm{\CC}{H^{\geq\phi}(b)}{c'}=0$ from the other Hom vanishing together with $\phi>\phi'$. Hence $\gamma=0$ as claimed.
\end{proof}

It remains to check that each $c\in \CC$ has a HN filtration with respect to $Q$. We do so by induction on the length of the HN filtration of $c$ in $\CC/\NN$ with respect to $Q_{\CC/\NN}$. The next result provides the base case.

\begin{lemma}
\label{lem:HN base case}
Suppose $c\in Q_{\CC/\NN}(\phi)$. Then $c$, considered as an object of $\CC$, has a HN filtration with respect to $Q$ with all factors in $\NN$ except for a single factor in $Q(\phi)$.
\end{lemma}

\begin{proof}
Let $0 < \epsilon < \frac{1}{2}$ be such that all $P(\phi-2\epsilon,\phi+2\epsilon)$ are locally finite. Set $\cat{A} \coloneqq P(\phi-\epsilon, \phi+\epsilon)$ and to begin with, assume $c\in \cat{A}$.

\Step{1} $c\in Q(\phi) \iff \Homm{\CC}{b}{c} = 0 = \Homm{\CC}{c}{d}$ for all $b\in \cat{A} \cap Q_\NN({>}\phi)$, $d\in \cat{A} \cap Q_\NN({<}\phi)$.

One direction is clear: when $c\in Q(\phi)$ the vanishing conditions follow from the definition of $Q$. For the other direction let $b\in Q_\NN({>}\phi)$. Applying $\Homm{\CC}{-}{c}$ to the exact triangle
\[
H_P^{\geq \phi+\epsilon}(b) \to b \to H_P^{< \phi+\epsilon}(b) \to H_P^{\geq\phi+\epsilon}(b)[1]
\]
and using $\Homm{\CC}{H_P^{{\geq} \phi+\epsilon}(b)}{c} = 0$ from $c \in \cat{A} \subseteq P(< \phi+\epsilon)$ shows that any morphism $b \to c$ factors through $h \coloneqq H_P^{< \phi+\epsilon}(b)$. The assumption $d(P_\NN,Q_\NN)<\epsilon$ implies
\[
H_P^{\geq\phi+\epsilon}(b)[1] \in P_\NN(\geq \phi+\epsilon+1) \subseteq Q_\NN(\geq \phi+1) \subseteq Q_\NN({>}\phi).
\]
Since $Q_\NN({>}\phi)$ is extension-closed, the above triangle shows that $h \in Q_\NN({>}\phi)$. Moreover, $h \in \cat{A}$ because $b\in Q_\NN({>}\phi) \subseteq P_\NN({>} \phi-\epsilon)$. Therefore $\Homm{\CC}{h}{c}=0$ by assumption, and hence $\Homm{\CC}{b}{c}=0$ too. The dual argument gives $\Homm{\CC}{c}{d}=0$ for all $d\in Q_\NN({<}\phi)$. Hence $c\in Q(\phi)$ as claimed.

\Step{2} $c \in \cat{A}$ has a HN filtration with respect to $Q$ with all factors in $\NN$ except for a single factor in $Q(\phi)$ isomorphic to $c$ in $\CC/\NN$.

Let $b$ be a maximal strict subobject of $c$ in the subcategory $\cat{A}\cap Q_\NN({>}\phi)$ of $\cat{A}$. We can always find such a $b$ (possibly zero) because $\cat{A}$ is a quasi-abelian length category. Then $c'=c/b$ has no non-zero strict subobjects in $\cat{A}\cap Q_\NN({>}\phi)$ because if $b' \inj c'$ is such a strict subobject we can pullback to obtain a commutative diagram
\[
\begin{tikzcd}
b \ar[equals]{d} \ar[hook]{r} & b'' \ar[hook]{d}\ar[->>]{r} & b'\ar[hook]{d} \\
b \ar[hook]{r} & c\ar[->>]{r} & c'
\end{tikzcd}
\]
whose rows are strict short exact sequences and whose vertical morphisms are strict monomorphisms. In particular $b''$ is a strict subobject of $c$ in $\cat{A}\cap Q_\NN({>}\phi)$ and by maximality of $b$ we deduce that $b\cong b''$ and therefore that $b'=0$. By assumption $b$ has a HN filtration with respect to $Q_\NN$. Concatenating this with $b \inj c$, we obtain a finite filtration of $c$ in $\cat{A}$ whose quotients are a sequence of $Q_\NN$-semistable objects of strictly decreasing phases in $\cat{A}\cap Q_\NN({>}\phi)$, except for the final quotient $c'$ which has no non-zero strict subobjects in $\cat{A} \cap Q_\NN({>}\phi)$. 

The dual argument constructs a finite filtration of this final quotient $c'$ whose first term $c''$ has no non-zero strict quotients in $\cat{A} \cap Q_\NN({<}\phi)$ and whose other quotients form a sequence of $Q_\NN$-semistable objects of strictly decreasing phases in $\cat{A}\cap Q_\NN({<}\phi)$. It follows that $c''$ cannot have any non-zero strict subobjects in $\cat{A} \cap Q_\NN({>}\phi)$ either, for any such would lift to a subobject of $c'$. To summarise we have constructed a strict subquotient $c''$ of $c$ in $\cat{A}$ such that
\begin{enumerate}
\item $c''\cong c$ in $\CC/\NN$, in particular $c''\in Q_{\CC/\NN}(\phi)$;
\item $c''$ has no non-zero strict subobjects in $\cat{A} \cap Q_\NN({>}\phi)$;
\item $c''$ has no non-zero strict quotients in $\cat{A} \cap Q_\NN({<}\phi)$.
\end{enumerate}
Because the image of any morphism in the quasi-abelian category $\cat{A}$ is a strict subobject of the target, and dually that the coimage is a strict quotient of the source, we conclude that $\Homm{\CC}{b}{c''} = 0 = \Homm{\CC}{c''}{d}$ for all $b\in \cat{A} \cap Q_\NN({>}\phi)$ and $d\in \cat{A} \cap Q_\NN({<}\phi)$. Hence $c''\in Q(\phi)$ by Step 1 and concatenating the filtrations of $b$ and of $c'$ using iterated applications of the octahedral axiom yields the desired HN filtration.

%
% commented out: details of how to concatenate -- leave this comment in 
%
%% Consider the HN filtration of $c'$,
%% \[
%% \begin{tikzcd}
%% 0 \ar{r} & c'' = c'_1 \ar{r} \ar{d}  & c'_2 \ar{r} \ar{d}      &\cdots \ar{r} & c'_{t-1} \ar{r} & c'_t = c', \ar{d}\\
%%             & q_1 \ar[dashed]{ul}   &  q_2 \ar[dashed]{ul} &                   &                       & q_t \ar[dashed]{ul}
%% \end{tikzcd}
%% \]
%% and the triangle $b \to c \to c' \to b[1]$. We can inductively construct a tower
%% \[
%% \begin{tikzcd}
%% b \ar{r} & x_i \ar{r} \ar{d}  & x_{i+1} \ar{r} \ar{d}                &\cdots \ar{r} & x_{t-1} \ar{r} & x_t = c, \ar{d}\\
%%             & c'_i \ar[dashed]{ul}   &  q_{i+1} \ar[dashed]{ul} &                   &                       & q_t \ar[dashed]{ul}
%% \end{tikzcd}
%% \]
%% where the inductive step is given by the following commutative diagram constructed using the octahedral axiom:
%% \[
%% \begin{tikzcd}
%% x_i \ar{r} \ar[equal]{d} & c'_i \ar{r} \ar{d}      & b[1] \ar{d} \\
%% x_i \ar{r}                    & q_i \ar{r} \ar{d}       & x_{i-1}[1] \ar{d} \\
%%                                  & c'_i[1] \ar[equal]{r}  & c'_i[1]
%%  \end{tikzcd}
%% \]
%% Concatenating the resulting filtration with that for $b$ gives a HN filtration for $c$ with all factors in $\NN$ except for a single factor $c'' \in Q(\phi)$.

\Step{3} The general case, \ie remove the assumption that $c\in \cat{A}$ from Step 2.

By Step 1, $c_0 \coloneqq H_P^{(\phi-\epsilon,\phi+\epsilon)}(c) \in \cat{A}$ has a HN $Q$-filtration. 
From $H_P^{\geq \phi + \epsilon} H_P^{> \phi - \epsilon}(c) = H_P^{\geq \phi + \epsilon}(c)$ and the octahedral axiom, we get a commutative diagram
\[
\begin{tikzcd}
H_P^{\geq \phi+\epsilon}(c) \ar{r} \ar[equals]{d}& c_1 \ar{r} \ar{d}& H_Q^{>\phi}(c_0) \ar{d} \\
H_P^{\geq \phi+\epsilon}(c) \ar{r} \ar{d}& H_P^{>\phi-\epsilon}(c) \ar{r} \ar{d}& c_0 \ar[]{d}\\
0 \ar{r} & H_Q^{\leq \phi}(c_0) \ar[equals]{r}& H_Q^{\leq \phi}(c_0)
\end{tikzcd}
\]
whose rows and columns extend to exact triangles. Moreover $H_P^{\geq \phi+\epsilon}(c) \in \NN$ because 
\[
c\in Q_{\CC/\NN}(\phi) \subseteq P_{\CC/\NN}(\phi-\epsilon, \phi+\epsilon).
\]
 Indeed $H_P^{\geq \phi+\epsilon}(c) \in Q_\NN({>}\phi)$ because $d(P_\NN,Q_\NN)<\epsilon$. Considering the top row, and recalling that $H_Q^{>\phi}(c_0)\in \NN$ too, shows that $c_1 \in Q_\NN({>}\phi)$. Therefore, by considering the middle column, we can construct a HN filtration for $c$ with $Q$-semistable factors by concatenating the filtrations of $c_1$ and of $H_Q^{\leq \phi}(c_0)$.
\end{proof}

\begin{proof}[Proof of \cref{prop:glueing slicings}]
Suppose $c\in \CC$ has a HN filtration of length $k\in \N$ in $\CC/\NN$ with $Q_{\CC/\NN}$-semistable factors. We show that $c$ has a HN filtration in $\CC$ with $Q$-semistable factors. If $k=0$ then $c\in \NN$ and we simply take the HN filtration with respect to $Q_\NN$. If $k=1$ then the result holds by \cref{lem:HN base case}. Therefore we assume that $k>1$ and that the result holds for any object with a strictly shorter HN filtration in $\CC/\NN$. Choose a representative $b\in Q_{\CC/\NN}(\psi)$ for the highest phase factor of $c$ so that there is an exact triangle $b\to c \to d \to b[1]$ in $\CC$. By induction we may assume both $b$ and $d$ have HN filtrations with $Q$-semistable factors. In particular there is an exact triangle 
\[
H^{\geq \psi}_Q(b) \to b \to H^{<\psi}_Q(b) \to H^{\geq \psi}_Q(b)[1]
\]
in $\CC$. Since $H^{\geq \psi}_Q(b) \in Q({\geq}\psi) \subseteq Q_{\CC/\NN}({\geq}\psi)$ and $H^{<\psi}_Q(b)\in Q({<}\psi) \subseteq Q_{\CC/\NN}({<}\psi)$ we deduce that $H^{<\psi}_Q(b)\in \NN$ and that $H^{\geq \psi}_Q(b) \to b$ is an isomorphism in $\CC/\NN$. Therefore we may assume $b\in Q({\geq}\psi)$. Having done so we argue similarly with $d$. There is an exact triangle
\[
H^{\geq \psi}_Q(d) \to d \to H^{<\psi}_Q(d) \to H^{\geq \psi}_Q(d)[1]
\]
where now $H^{\geq \psi}_Q(d)\in \NN$ and $d \to H^{<\psi}_Q(d)$ is an isomorphism in $\CC/\NN$. Hence there is a commutative diagram
\[
\begin{tikzcd}
b \ar{r} \ar[equals]{d}& b' \ar{r} \ar{d}& H_Q^{\geq\psi}(d) \ar{d} \\
b\ar{r} \ar{d}& c\ar{r} \ar{d}& d \ar[]{d}\\
0 \ar{r} & H_Q^{<\psi}(d) \ar[equals]{r}& H_Q^{<\psi}(d) 
\end{tikzcd}
\]
whose rows and columns extend to exact triangles. By considering the top row we see that $b'\in Q({\geq}\psi)$. Thus the middle column shows that we may assume, by judicious choice of representatives, that $b\in Q({\geq}\psi)$ and $d\in Q({<}\psi)$. Having done so, we obtain a $Q$ HN filtration for $c$ by concatenating those of $b$ and $d$.

Clearly $d(P,Q)<\epsilon$ because $Q(\phi) \subseteq P(\phi-\epsilon,\phi+\epsilon)$ for all $\phi\in \R$ by \cref{lem:slicing distance}. Finally $Q$ is locally finite because $Q(\phi-\epsilon,\phi+\epsilon) \subseteq P(\phi-2\epsilon,\phi+2\epsilon)$ and the latter is length.
\end{proof}

\begin{example}
\label{ex:slicings}
Let $\CC = \Db(\PP^1)$ be the bounded derived category of coherent sheaves on the projective line $\PP^1$. All complexes in $\CC$ decompose into direct sums of their cohomologies (this holds for any smooth curve) and, moreover, all coherent sheaves decompose into direct sums of line bundles $\cO(n)$ and torsion sheaves; the latter have the skyscraper sheaves $\cO_x$ for $x\in\PP^1$ as their minimal non-zero subsheaves.

For a slicing in $\CC$, up to shift and direct sums, various $\cO(n)$ and $\cO_x$ occur as factors of HN filtrations and, conversely, the decomposition properties of $\CC$ imply that HN filtrations exist trivially for any family of subcategories $P(\phi)$ with $P(\phi+1) = P(\phi)[1]$ and Hom-vanishing $\Hom{P({>}\phi)}{P(\phi)} = 0$ and such that all $\cO_x$ and $\cO(n)$ are in the heart $P(0,1]$, again up to shift. Thus the following assignments for $\phi\in(0,1]$ give slicings on $\CC$:
% The non-zero morphisms in $\CC$ among these two classes of sheaves are $\cO(n) \to \cO(m)$ for $n \leq m$; $\cO(a) \to \cO(b)[1]$ for $a \geq b-2$; and $\cO(n) \to \cO_x$ as well as $\cO_x \to \cO(n)[1]$ for any $n\in\Z$ and $x\in\PP^1$.
\[
P_t(\phi) = 
\begin{cases}
  \Clext{ \cO(n) : n\in \Z } & \phi=\frac{1}{2}\\
  \Clext{ \cO_x : x\in \PP^1 } & \phi=1\\
  0 & \textrm{else};
\end{cases}
, \qquad
P_g(\phi) = 
\begin{cases}
  \Clext{ \cO(n) : n\in \Z } & \phi = \frac{1}{\pi} \arg(-n+i) \in (0,1)\\
  \Clext{ \cO_x : x\in \PP^1 } & \phi = 1\\
  0 & \textrm{else};
\end{cases}
\]
\[
P_b(\phi) = 
\begin{cases}
  \Clext{ \cO } & \phi=\frac{1}{2}\\
  \Clext{ \cO_x, \cO(n), \cO(-n)[1] : x\in\PP^1, n\in\N_{>0} } & \phi=1\\
  0 & \textrm{else};
\end{cases}
, \qquad
P_c(\phi) =
\begin{cases}
  \coh(\PP^1) & \phi=1\\
  0 & \textrm{else}.
\end{cases}
\]
$P_t$ separates torsion sheaves and line bundles into two slices. $P_g$ is the geometric slicing induced by the classical slope of coherent sheaves $\mu(A) = \deg(A)/\rk(A)$; see \cref{ex:seminorm balls asymmetric}. The slicing $P_b$ occurs in the boundary of the stability space; see \cref{ex:norm example} and \cref{subsec:projective line}.
% $P_b(1)$ has infinite chains but not of sub/quotient objects. Its simple objects are $O(1)$ and $O(-1)[1]$.

These three slicings are locally finite. Note that the slice $P_t(\frac{1}{2})$ contains the infinite chain $\cdots\to\cO(-1)\to\cO\to\cO(1)\to\cdots$. Nonetheless $P_t$ is locally finite because any non-zero morphism $\cO(n)\to\cO(m)$ for $n<m$ is not strict in $P_t(\frac{1}{2})$ since it has image $\cO(n)$ and coimage $\cO(m)$. In particular, each $\cO(n)$ is simple in the quasi-abelian but not abelian category $P_t(\frac{1}{2})$.

In $P_c$, a single slice contains the whole heart, which is not length and therefore the slicing is not locally finite.
See \cite{GKR04} for the classification of stability conditions and bounded t-structures on $\CC$.
\end{example}

\begin{example}
\label{ex:thick subcategories}
We continue the above example with two thick subcategories of $\CC = \Db(\PP^1)$: first the subcategory $\NN = \NN_\cO = \thick{}{\cO}$ generated by the trivial line bundle. Because every line bundle is an exceptional object, $\NN \cong \Db(\kk)$ and $\NN$ is an admissible subcategory, \ie the inclusion $\NN \inj \CC$ has both adjoints. Hence there is a canonical equivalence $\CC/\NN \cong \NN^\perp = \thick{}{\cO(-1)}$.

Next, let $\TT$ be the thick subcategory of $\CC$ generated by all torsion objects in $\coh(\PP^1)$; it is thick but not admissible. The quotient $\CC/\TT$ is not Hom-finite: the objects $\cO$ and $\cO(-1)$ are isomorphic in $\CC/\TT$ but the morphisms $x,y\colon \cO(-1) \to \cO$ for $x\neq y \in \PP\Hom{\cO(-1)}{\cO} \simeq \PP^1$, induce non-equivalent elements of $\End{\CC/\TT}{\cO}$.
%
%% Verification commented out but kept in tex file.
%We verify non-equivalence below. In order to have equivalence, we require that there is a commutative diagram
%\[
%\begin{tikzcd}
%                                                                     & \cO \arrow[d,"s"] & \\
%\cO(-1) \arrow[ur,"x"] \arrow[dr,"y"'] \arrow[r,"z"] & Z                    & \arrow[l,"s"'] \arrow[dl,"1_{\cO}"] \arrow[ul, "1_{\cO}"'] \cO \\
%                                                                     & \cO \arrow[u,"s"'] &
%\end{tikzcd}
%\]
%with the cone, $C$, of $s \colon \cO \to Z$ in $\TT$. Now $C \simeq C_{\leq 0} \oplus C_{> 0}$ with $C_{\leq 0} \in \add{\cO_u[\leq 0] : u \in \PP^1}$ and $C_{> 0} \in \add{\cO_u[>0] : u \in \PP^1}$. As $s(x-y) = 0$, we have $x-y$ factors through $C[-1] \to \cO$.
%However, $\Hom{\cO(-1)}{C_{\leq 0}[-1]} = 0$ and $\Hom{C_{>0}[-1]}{\cO} = 0$ together imply that $x - y = 0$. Hence, for $y \neq x$, there is no such commutative diagram and the morphisms are not equivalent in $\End{\CC/\TT}{\cO}$.
\end{example}

\section{Lax stability conditions and quotient categories}
\label{sec:lax and quotient stability conditions}

\subsection{Stability conditions}
\label{subsec:stability conditions}

\noindent
Let $\CC$ be a triangulated category and $v\colon K(\CC) \to \Lambda$ a surjective homomorphism from the Grothendieck group to a finite rank lattice.

A \defn{pre-stability condition} on $\CC$ is a pair $(P, Z)$ consisting of a slicing $P$, which is not assumed to be locally finite, and a charge $Z\in \Hom{\Lambda}{\C}$ such that $0\neq c \in P(\phi)$ implies $Z(c) = m(c) \exp(i\pi \phi)$ with $m(c) \in \R_{>0}$. An object $c\in P(\phi)$ is said to be \defn{semistable of phase $\phi$} and $m(c)$ is its \defn{mass}. The \defn{mass} of any $c \in \CC$ is defined to be
$ m(c) \coloneqq m(c_1) + \cdots + m(c_k) $
where $c_i \in P(\phi_i)$ for $i=1,\ldots,k$ are the semistable factors of $c$ with respect to the slicing $P$.

A \defn{stability condition} is a pre-stability condition $\sigma=(P,Z)$ such that there is $K>0$ with
\[
m(c) = |Z(c)| \geq \frac{1}{K}\norm{c} 
\]
for all semistable $c\in \CC$. This latter condition is referred to as the \defn{support property} \cite[p7]{KS08}; it is independent of the choice of norm because $\dim (\Lambda\otimes \R) <\infty$. The support property has three important consequences. First, and most obviously, it implies that the \defn{infimal mass} 
\[
\mu_\sigma = \inf \{ m(c) : 0\neq c\in \CC\} \geq \frac{1}{K} \inf\{ \norm{\lambda} \colon 0\neq \lambda \in \Lambda\}
\]
is strictly positive. Second, it implies that the slicing $P$ is locally finite; in fact that $P(I)$ is a length category for any interval $I\subset \R$ of length $|I|<1$; see \cite[p8]{KS08}. 
Third, it implies that the generalised norm 
\[
U \mapsto \norm{U}_\sigma \coloneqq \sup\left \{ \frac{|U(c)|}{|Z(c)|} : 0\neq c\in P(\phi), \phi\in \R \right\}
\]
defined in \cite{Bridgeland07} is actually a norm on $\Hom{\Lambda}{\C}$ because it is bounded above as $\norm{U}_\sigma \leq K\norm{U}$ where 
$
\norm{U} = \sup \{ |U(\lambda)| : \lambda\in \Lambda\otimes\R,\ \norm{\lambda}=1\}
$
is the operator norm. In fact the support property is equivalent to $\norm{\cdot}_\sigma$ being a norm, see \cite[App.~B]{BM11}. 

The central result in the theory of stability conditions is the following deformation theorem.

\begin{theorem}[{\cite[Thm.~7.1 and Lem.~6.2]{Bridgeland07}}]
\label{thm:deformation theorem}
Let $\sigma = (P,Z)$ be a pre-stability condition, $0 < \epsilon < \tfrac{1}{8}$ and $W \in \Hom{\Lambda}{\C}$.
If $\norm{ W-Z}_\sigma< \sin(\pi\epsilon)$ then there is a unique pre-stability condition $\tau = (Q,W)$ with $d(P,Q)<\epsilon$. Moreover, if $\sigma$ is a stability condition then so is $\tau$. 
\end{theorem}

\noindent
The set of stability conditions $\Stab{\CC}$ is topologised as a subset of $\Slice{\CC} \times \Hom{\Lambda}{\C}$.
\Cref{thm:deformation theorem} shows that the charge map $\cm \colon \Stab{\CC} \to \Hom{\Lambda}{\CC}$, $(P,Z) \mapsto Z$ is a local homeomorphism, and therefore that $\Stab{\CC}$ can be given the structure of a, possibly empty, complex manifold of dimension $\rk(\Lambda)$.

\subsection{Lax pre-stability conditions}
\label{subsec:lax pre-stability conditions}

We start with a modified version of Bridgeland's notion of pre-stability condition without the requirement that the masses of non-zero objects have to be positive. Late we will add a concomitantly modified support property. See \cite[\S3]{CLSY} for almost the same notion in another shape; the only difference is that in \cite{CLSY} the slicing is not required to be locally finite.

\begin{definition}
\label{def:lax pre-stability condition}
A \defn{lax pre-stability condition} on $\CC$ is a pair $(P, Z) \in \Slice{\CC}\times \Hom{\Lambda}{\C}$ such that $0\neq c \in P(\phi)$ implies $Z(c) = m(c) \exp(i\pi \phi)$ for some $m(c) \in \R_{\geq 0}$.
\end{definition}

\noindent
Recall that $\Slice{\CC}$ is defined as the set of \emph{locally finite} slicings, \ie in contrast to the case of a pre-stability condition we assume that the slicing of a lax pre-stability condition is locally finite from the outset. The reason for the difference is that local finiteness will not be a consequence of the weaker support property we define for lax pre-stability conditions below.

As before, we refer to $m(c) = |Z(c)|$ as the \defn{mass} of a semistable object $c\in P(\phi)$. The \defn{mass} of any $c \in \CC$ is again defined as $m(c) \coloneqq m(c_1)+\cdots+m(c_k)$ where $c_1 \in P(\phi_1),\ldots,c_k \in P(\phi_k)$ 
are the semistable factors of $c$ with respect to the slicing $P$.
 
\begin{definition}
An object $c\in \CC$ is called \defn{massive} if $m(c)>0$, and \defn{massless} if $m(c)=0$. Note that $0\in\CC$ has no semistable factors, so that $m(0)=0$, \ie $0$ is always a massless object. 

The \defn{massless subcategory} $\NN_\sigma$ of a lax pre-stability condition $\sigma$ is the full subcategory of the massless objects. Thus $\sigma$ is a pre-stability condition precisely when $\NN_\sigma=0$; sometimes for emphasis we say it is a \defn{strict} pre-stability condition. 
\end{definition}

\begin{proposition}
\label{prop:massless thick}
The massless subcategory of a lax pre-stability condition is a thick subcategory to which the slicing is adapted. 
\end{proposition}

\begin{proof}
Let $\sigma = (P,Z)$ be a lax pre-stability condition and write $\NN = \NN_\sigma$.
It is clear that $\NN$ is closed under shifts. Suppose that $b \in \CC$ sits in a triangle $a \to b \to c \to a[1]$ with $a, c \in \NN$. Taking cohomology with respect to the t-structure with heart $P(0,1]$ we obtain a long exact sequence $\cdots \to H^i(a) \to H^i(b) \to H^i(c) \to \cdots$ in $P(0,1]$. By assumption $m(a)=0=m(c)$, so that $m( H^i(a)) = 0 = m(H^i(c))$ for each $i\in \Z$ since the HN filtration of an object of $\CC$ is a refinement of the decomposition of that object into its cohomology with respect to the heart $P(0,1]$. 
Because $\NN \cap P(0,1]$ is a Serre subcategory, we get $m(H^i(b))=0$ too. Hence $m(b) = \sum_{i\in \Z} m(H^i(b)) = 0$, and $\NN$ is extension-closed, hence a triangulated subcategory.

Since the set of semistable factors of $a\oplus b$ is the union of the sets of semistable factors of $a$ and $b$, we obtain that $\NN$ is thick from the following chain of equivalences:
\[
  a\oplus b \in \NN \iff m(a\oplus b)=0 \iff m(a)+m(b)=0 
  \iff m(a)=0=m(b) \iff a, b \in \NN.
\]
Every semistable factor of a massless object is massless since the mass of an object is the sum of the masses of its semistable factors. 
Moreover, for any interval $I$ of the form $(\phi, \phi+1]$ or $[\phi, \phi+1)$ the full subcategory $P(I)$ is the heart of a t-structure, and hence abelian. The intersection $P(I) \cap \NN$ is a Serre subcategory because $m(c)=0 \iff Z(c)=0$ for $c\in P(I)$. Therefore $P$ is adapted to $\NN$.
\end{proof}

\begin{lemma}
\label{lem:massless simples}
If $\sigma = (P,Z)$ is a lax pre-stability condition then its massless subcategory is the triangulated closure $\NN_\sigma = \triang{\CC}{S}$ of the set $S$ of stable massless objects in the heart $P(0,1]$. 
\end{lemma}

\begin{proof}
Evidently $\triang{\CC}{S} \subseteq \NN_\sigma$. For the other inclusion, let $c\in\NN_\sigma$ be a massless object. Its HN filtration has finitely many semistable factors, and each factor is massless and has a finite composition series in its slice with stable massless objects. Up to shift each of these stable objects has phase in $(0,1]$. Therefore $c\in \triang{\CC}{S}$ and $\NN_\sigma \subseteq \triang{\CC}{S}$.
\end{proof}

The semistable factors of a massless object are, by definition, massless. In fact, this property persists in an open neighbourhood in the following sense.

\begin{lemma}
\label{lem:local persistence of massless factors}
Let $\sigma =(P,Z)$ be a lax pre-stability condition and $Q$ a slicing. If $d(P,Q) < \tfrac{1}{6}$ then $Q$ restricts to a slicing on the massless subcategory $\NN_\sigma$.
\end{lemma}

\begin{proof}
Let $d(P,Q) < \epsilon < \tfrac{1}{6}$ and fix $c \in \NN \coloneqq \NN_\sigma$. We must show that the $Q$-semistable factors of $c$ lie in $\NN$, so let $b\in Q(\phi)$ be one such $Q$-semistable factor of $c$. Then $b$ is also a $Q$-semistable factor of $c' \coloneqq H_P^{(\phi-\epsilon,\phi+\epsilon)}(c)$ because $Q(\phi) \subseteq P(\phi-\epsilon,\phi+\epsilon)$ so that
\[
H_Q^\phi (c') = H_Q^\phi H_P^{(\phi-\epsilon,\phi+\epsilon)}(c) = H_Q^\phi (c) = b.
\]
Moreover, $c'\in \NN$ too, since each of its $P$-semistable factors lies in $\NN$. In particular, $Z(c')=0$. 

Let $b_1,\ldots,b_m$ be the $Q$-semistable factors of $c'$, so that $b=b_i$ for some $1\leq i \leq m$. Since $c'\in P(\phi-\epsilon,\phi+\epsilon) \subseteq Q(\phi-2\epsilon,\phi+2\epsilon)$ we have
\[
b_1, \ldots,b_m \in Q(\phi-2\epsilon,\phi+2\epsilon) \subseteq P(\phi-3\epsilon,\phi+3\epsilon).
\] 
Now let $b_{i,j}$ for $j=1, \ldots, n_i$ be the $P$-semistable factors of $b_i$ for $1\leq i \leq m$. The equation
\[
\sum_{i,j}Z(b_{i,j}) = \sum_i Z(b_i) = Z(c') = 0
\]
implies $Z(b_{i,j})=0$ because of $6 \epsilon < 1$. Hence $b_{i,j}\in\NN$, for each $1\leq i \leq m$ and $1\leq j\leq n_i$, so that all $P$-semistable factors of $b$ lie in $\NN$. In particular, $b\in \NN$ as claimed.
\end{proof}

\begin{remark}
When $\sigma=(P,Z)$ is a strict pre-stability condition the slices $P(\phi)$ are abelian categories \cite[Lem.~5.2]{Bridgeland07}. However, this may fail for lax pre-stability conditions since the argument relies on the charge of each semistable object being non-zero --- see \cref{ex:lax supported versus weak}. Nevertheless, each slice $P(\phi)$ is a quasi-abelian length category, and so each semistable object $c$ has at least one finite composition series with stable factors. Indeed, every composition series of $c$ must have finite length, but the lengths need not be the same, and the multi-sets of stable factors need not be unique up to isomorphism. However, we are not aware of an example of a slice $P(\phi)$ in which there is an object with Jordan--H\"older filtrations of different lengths.
\end{remark}

\subsection{Lax stability conditions: the lax support property}
\label{subsec:lax stability conditions}

Evidently the support property fails when there are massless objects. We remedy this with a new notion generalising to the lax setting. To phrase it concisely, we make the following definition.

\begin{definition}
An object $c\in \CC$ is \defn{$(P,\delta)$-slim} for a slicing $P$ and $\delta > 0$ if $c\in P(\phi-\delta,\phi+\delta)$ for some $\phi\in\R$. We extend this to $\delta=0$ by saying that $c\in \CC$ is $(P,0)$-slim if $c\in P(\phi)$ for some $\phi\in\R$, \ie $c$ is semistable.
\end{definition}

\noindent The slicing will always be part of a lax pre-stability condition $\sigma=(P,Z)$, and we write \emph{$(\sigma,\delta)$-slim} instead, and shorten to \emph{$\delta$-slim} if $\sigma$ is clear from the context.

\begin{definition}[Lax stability condition]
\label{def:lax stability condition}
Let $\sigma=(P,Z)$ be a lax pre-stability condition.
\begin{enumerate}
\item $\sigma$ satisfies \defn{$\delta$-lax support} for some $\delta\geq 0$ if there is $K>0$ such that
\[
|Z(c)| \geq \frac{1}{K} \norm{c}
\]
for all massive, indecomposable $\delta$-slim objects $c\in \CC$. We refer to $K$ as the \defn{support constant} and $\delta$ as the \defn{support width}. 
We also call $\sigma$ a \defn{$\delta$-lax stability condition}.
\item $\sigma$ satisfies \defn{lax support} if it satisfies $\delta$-lax support for some $\delta>0$. 
\item A \defn{lax stability condition} is a lax pre-stability condition with lax support.
\end{enumerate}
\end{definition}

\begin{remarks}
\label{rmk:lax support}
More succinctly, $\sigma$ satisfies $\delta$-lax support if $\exists K>0$ with $|Z(c)| \geq \norm{c}/K$ for all massive, indecomposable $c\in\CC$ with $\phi^+_\sigma(c) - \phi^-_\sigma(c) < 2\delta$. 

Lax support agrees with support for strict pre-stability conditions, as is shown in \cref{lem:lax and strict support} below.
It is clear that massless objects must be excluded from any analogue of the support property for lax pre-stability conditions. 

We consider only \emph{indecomposable} massive objects to avoid support failing simply because there is a massive semistable object $b$ and a non-zero massless semistable object $c$ of the same phase. In that case the sums $b\oplus c^n$ for $n\in \N$ are semistable with fixed mass $m(b\oplus c^n)=m(b)$ but with $\norm{ b \oplus c^n } \to \infty$ as $n\to \infty$. 

A stable object is indecomposable, because any summand is semistable of the same phase.
Moreover, an object of $P(\phi-\delta,\phi+\delta)$ is indecomposable in this subcategory if and only if it is indecomposable in $\CC$ because any summand of a slim object is slim.

We test slim objects, \ie nearby phases, so that lax support becomes an open property on lax pre-stability conditions with the same massless category --- this follows from \cref{cor:propagation}(2) and the observation that if $\sigma$ and $\tau$ have the same massless subcategory then the restriction $\restrict{\NN_\sigma}(\tau)=0$. 
To see that the condition that the lax stability conditions have the same massless subcategory is necessary, see \cref{ex:not well-supported}.
The openness property would not hold if we only tested on semistable objects in the slice; see \cref{subsec:nilrep} for a concrete example.\footnote{\cref{def:lax stability condition} crucially differs from this article's first {\tt arXiv} version where the inequality $|Z(c)| \geq \norm{c} / K$ was only checked on \emph{massive stable} objects $c\in\CC$. The latter notion is not necessarily open on lax pre-stability conditions with the same massless category, as kindly pointed out to us by Arend Bayer. See \cref{subsec:nilrep}.}

If $\sigma$ satisfies $\delta$-lax support then it satisfies $\delta'$-lax support for every $0\leq \delta' \leq \delta$.
\end{remarks}

\begin{remark}
\label{support and local-finiteness}
Unlike support for strict pre-stability conditions, lax support does not imply that the slicing is locally finite. For example $(P,0)$ is a compatible pair of slicing and charge for any slicing $P$, and trivially satisfies lax support because there are no massive objects, regardless of whether $P$ is locally finite. This is why we have to assume our slicings to be locally finite from the outset.

For an explicit example, take $\CC = \Db(\PP^1)$ and $\sigma=(P,Z)$ with charge $Z=0$ and the slicing $P=P_c$ of \cref{ex:slicings}, \ie $P(1)=\coh(\PP^1)$. Then $\sigma$ satisfies the lax support property since every non-zero object is massless. However, it is not locally finite since $P(1-\epsilon,1+\epsilon)=P(1)$ is not length for any $0<\epsilon<1/2$. 
\end{remark}

In addition to the local finiteness issue, there is a further technical question. The slices $P(\phi)$ of a lax stability condition are quasi-abelian but, in contrast to the situation for strict stability conditions, need not be abelian. However, as noted in \cref{subsec:quasi-abelian}, the notion of a `length' quasi-abelian category makes sense. This raises the following natural question.

\begin{question}
Are there examples where the Jordan--H\"older property fails for the slices $P(\phi)$ of a lax stability condition? If there are then, whilst the HN filtrations of objects are unique, their refinements to filtrations with stable factors would not be. Presumably this would have implications for the wall-and-chamber structure.
\end{question}

We conclude this subsection by observing that for strict pre-stability conditions, lax support coincides with the original definition of support.

\begin{lemma}
\label{lem:lax and strict support}
The following properties are equivalent for a strict pre-stability condition $\sigma$ on a Hom-finite triangulated category $\cat{C}$:
\begin{enumerate}[label=(\roman*)]
\item $\sigma$ satisfies support, \ie $\sigma$ is a stability condition;
\item $\sigma$ satisfies $\delta$-lax support for all $0\leq \delta < \tfrac{1}{4}$;
\item $\sigma$ satisfies $\delta$-lax support for some $0\leq \delta$.
\end{enumerate}
\end{lemma}

\begin{proof}
\noindent (i) \timplies (ii): Suppose $\sigma = (P,Z)$ satisfies support. Fix $0\leq \delta < \tfrac{1}{4}$ and let $c\in P(\phi-\delta,\phi+\delta)$ with semistable factors $a_1, \ldots,a_n$. Then there exists $K>0$ such that
\[
|Z(c)| \geq \cos(2\pi\delta) \sum_{i=1}^n |Z(a_i)| \geq \frac{\cos(2\pi\delta)}{K} \sum_{i=1}^n \norm{a_i} \geq \frac{\cos(2\pi\delta)}{K} \norm{c}
\]
by trigonometry, the support property and the triangle inequality. Hence $\sigma$ has $\delta$-lax support.

\noindent (ii) \timplies (iii) is immediate. 

\noindent (iii) $\implies$ (i): Suppose $\sigma$ satisfies $\delta$-lax support for some $\delta\geq 0$. 
Since $\CC$ is Hom-finite, each $c\in P(\phi)$ decomposes as a finite direct sum $c\cong c_1\oplus \cdots \oplus c_n$ of indecomposable objects in $P(\phi)$ by \cref{prop:KRS property}.
These summands are massive, so $\delta$-lax support for $\sigma$ and the triangle inequality yield the following inequality, showing $\sigma$ satisfies support
\[
|Z(c)| = \sum_{i=1}^n |Z(c_i)| \geq \frac{1}{K} \sum_{i=1}^n \norm{c_i} \geq \frac{1}{K}\norm{c}. \qedhere
\]
\end{proof}

\subsection{Semi-norms and support}
\label{subsec:semi-norms}

A lax pre-stability condition $\sigma = (P,Z)$ defines a family of generalised semi-norms on $\Hom{\Lambda}{\C}$: for $\delta\geq0$, let
\[
W \mapsto \norm{W}_{\sigma,\delta} \coloneqq \sup \Big\{ \frac{|W(c)|}{|Z(c)|} : c\in\CC \text{ massive indecomposable $\delta$-slim} \Big\} .
\] 
By convention we set $\sup(\emptyset)=0$ so that $\norm{W}_{\sigma,\delta}=0$ for all $W\in \Hom{\Lambda}{\C}$ when $\sigma$ has no massive objects. The adjective `generalised' refers to the fact that we allow $\norm{W}_{\sigma,\delta}=\infty$ if the supremum does not exist, see \cref{ex:non-supported} below. These are norms when $\sigma$ is strict, but in general are only \emph{semi-}norms because $\norm{W}_{\sigma,\delta}=0$ for a non-zero charge $W$ which vanishes on all \emph{massive} objects. The next result shows that for a strict stability condition all the semi-norms in the family are equivalent to the norm $\norm{\cdot}_\sigma$.

\begin{lemma}
\label{lem:seminorm inequalities}
Let $\sigma$ be a strict pre-stability condition, $W \in \Hom{\Lambda}{\C}$ and $0\leq \delta < \tfrac{1}{4}$. Then
\[
\norm{W}_\sigma \leq \norm{W}_{\sigma,\delta} \leq \norm{W}_\sigma / \cos(2\pi\delta) .
\]
In particular, the semi-norms $\norm{\cdot}_{\sigma,\delta}$ are all equivalent to $\norm{\cdot}_\sigma$.
\end{lemma}

\begin{proof}
First note that $\norm{W}_\sigma = \norm{W}_\sigma^s$ where $\norm{W}_\sigma^s \coloneqq \sup\left \{ {|W(s)|}/{|Z(s)|} : \text{non-zero $\sigma$-stable}\ s\right\}$. This follows since for semistable $c$
\[
 |W(c)| \leq \sum_{s\in S} |W(s)| \leq \norm{W}_\sigma^s \sum_{s\in S} |Z(s)| = \norm{W}_\sigma^s \cdot |Z(c)|
\]
where $S$ is the multi-set of stable factors. Therefore, on the one hand $\norm{W}_\sigma \leq \norm{W}_{\sigma,\delta}$ since each stable object is massive, indecomposable and $\delta$-slim for any $\delta\geq 0$. On the other hand, each $\delta$-slim $c\in P(\phi-\delta,\phi+\delta)$ has a finite set $S$ of indecomposable semistable factors and
\[
|W(c)| \leq \sum_{s\in S}|W(s)| \leq \norm{W}_\sigma \sum_{s\in S} |Z(s)| \leq \frac{\norm{W}_\sigma}{\cos(2\pi\delta)} |Z(c)|
\]
by the triangle inequality and $\delta$-lax support. Hence $\norm{W}_{\sigma,\delta} \leq \norm{W}_\sigma / \cos(2\pi \delta)$. 
\end{proof}

\begin{definition}
A lax pre-stability condition $\sigma$ is \defn{$\delta$-full} if the semi-norm $\norm{\cdot}_{\sigma,\delta}$ is bounded on the unit ball in $\Hom{\Lambda}{\C}$, \ie if there exists $K>0$ such that for all $W\in \Hom{\Lambda}{\C}$
\[
\norm{W}_{\sigma,\delta} \leq K\norm{W} .
\]
\end{definition}

The above bound is the operator norm $\norm{W} = \sup \{ |W(\lambda)| : \lambda\in \Lambda\otimes \R,\ \norm{\lambda}=1\}$. This notion is independent of the norm on $\Lambda\otimes \R$, and reduces to the usual notion \cite{Bridgeland08} when $\sigma$ is strict.

The next result is a simple extension of \cite[Prop.~B.4]{BM11} and \cite[Lem.~11.4]{BMS16}, following \cite[\S 2.1]{KS08}, to the case of lax stability conditions. We will not use the characterisation of lax stability using quadratic forms on $\Lambda\otimes\R$ but the fullness criterion will be handy below.

\begin{proposition}
\label{prop:lax support and full}
The following properties are equivalent for a lax pre-stability condition $\sigma$ and any $\delta \geq 0$:
\begin{enumerate}[label=(\roman*)]
\item $\sigma$ satisfies $\delta$-lax support;
\item $\sigma$ is $\delta$-full, \ie $\norm{-}_{\sigma,\delta}$ is a semi-norm on $\Hom{\Lambda}{\C}$;
\item there exists a quadratic form $\Delta$ on $\Lambda\otimes \R$ such that 
\begin{enumerate}
\item $\Delta(c)\geq 0$ for each massive indecomposable $\delta$-slim $c\in \CC$;
\item $\Delta$ is negative definite on $\ker \cm(\sigma) \subset \Lambda\otimes \R$.
\end{enumerate}
\end{enumerate}
\end{proposition}

\begin{proof}
(i) \timplies (ii), (iii). Assume $\sigma = (P,Z)$ satisfies $\delta$-lax support. Then for any $W\in \Hom{\Lambda}{\C}$
\begin{align*}
\norm{W}_{\sigma,\delta} 
& =      \sup \{ |W(c)| / |Z(c)| : \text{massive indecomposable $\delta$-slim } c\in \CC \} \\
& \leq K \sup \{ |W(c)| / \norm{c} :\text{massive indecomposable $\delta$-slim } c\in \CC \} \\
& \leq K \sup \{ |W(\lambda)| / \norm{\lambda} : 0\neq \lambda\in \Lambda \} = K\norm{W}
\end{align*}
so that $\sigma$ is $\delta$-full. The quadratic form $\Delta(\lambda) \coloneqq K^2|Z(\lambda)|^2 - \norm{\lambda}^2$ on $\Lambda\otimes \R$ satisfies (iii).

(ii) \timplies (i). Assume $\sigma$ is $\delta$-full and, for a contradiction, that $\sigma$ does not satisfy $\delta$-lax support. Then there is a sequence $(c_n)$ of massive indecomposable $\delta$-slim objects with
$m(c_n) = |Z(c_n)| < {\norm{c_n}}/{n}$.
Choosing $W_n\in \Hom{\Lambda}{\C}$ with $\norm{W_n}=1$ and $|W_n(c_n)|=\norm{c_n}$, we get
%
% One way how to make $W_n$: (leave comment in text)
% Choose an orthonormal basis $e_1,\dots,e_r$ of $\Lambda\otimes\R$ such that $\lambda(c_n) = \norm{c_n} e_1$ is a scalar multiple of $e_1$.
% Set $W_n(e_1) \coloneqq 1$ and all other $W(e_i) \coloneqq 0.
% 
\[
\norm{W_n}_{\sigma,\delta} \geq \frac{|W_n(c_n)|}{|Z(c_n)|} > n\frac{|W_n(c_n)|}{\norm{c_n}} = n
\]
so that $\norm{\cdot}_{\sigma,\delta}$ is not bounded on the (compact) unit ball, contradicting that $\sigma$ is $\delta$-full.

(iii) \timplies (i). Finally, suppose that $\Delta$ is a quadratic form with $\Delta(c) \geq 0$ for every massive indecomposable $\delta$-slim object $c\in\CC$ and whose restriction to $\ker Z$ is negative definite. In particular if $\Delta(\lambda) >0$ then $\lambda \not \in \ker Z$ so $|Z(\lambda)|^2>0$. Therefore, because the unit ball is compact, there exists $K>0$ such that
  $\lambda \mapsto K^2|Z(\lambda)|^2 - \Delta(\lambda)$
is a positive definite form on $\Lambda\otimes\R$. If $\norm{\cdot}$ is the induced norm then
$
K^2|Z(c)|^2= \norm{c}^2 +\Delta(c) \geq \norm{c}^2
$
for each massive indecomposable $\delta$-slim object $c\in \CC$. Therefore $\sigma$ satisfies $\delta$-lax support.
\end{proof}

\begin{remark}
\label{rmk:local finiteness not needed}
\Cref{prop:lax support and full} is stated for a lax pre-stability condition $\sigma=(P,Z)$, but in fact holds slightly more generally as the assumption that the slicing $P$ is locally finite is not required. In particular, if $\sigma$ is a strict pre-stability condition then the three properties in the statement are equivalent.
\end{remark}

\begin{example}
\label{ex:not well-supported}
Let $\CC = \Db(\PP^1)$. Its Grothendieck group is $\Lambda = K(\CC) = K(\PP^1) \cong \Z^2$ using the basis $[\mathcal{O}]$ (structure sheaf) and $[\mathcal{O}_x]$ (skyscraper sheaves). The inner product is chosen so that this basis is orthonormal. Consider the lax pre-stability condition $(P_t,0)$ defined by zero charge and the slicing $P_t$ from \cref{ex:slicings}, \ie $P_t(1) = \Clext{\cO_x : x\in\PP^1}$ and $P_t(\tfrac{1}{2}) = \Clext{\cO(n) : n\in\Z}$.
With all objects massless, $(P_t,0)$ trivially satisfies the lax support property and is a lax stability condition.

Let $\sigma=(P_t,W)$ have the same slicing but $W(\mathcal{O}_x)=0$, $W(\mathcal{O})=i$.
This is a lax pre-stability condition because of $\mathcal{O}\in P_t(\tfrac{1}{2})$.
The massless subcategory of $\sigma$ is $\thick{}{\mathcal{O}_x : x\in \PP^1}$ and, for any $\delta\geq 0$, the massive indecomposable $\delta$-slim objects are $\mathcal{O}(n)[l]$ for $n,l\in \Z$. Therefore
\[
\norm{U}_{\sigma,\delta} 
= \sup\big\{ \tfrac{|U(b)|}{|W(b)|} : \text{massive indecomposable $\delta$-slim $b$} \big\} 
= \sup\{ | U (\mathcal{O}(n)) | : n\in \Z \} ,
\]
which is infinite for example when $U(\mathcal{O}_x) = 1$ and $U(\mathcal{O})=0$. Thus $\sigma$ does not satisfy lax support for any $\delta\geq 0$ and is not a lax stability condition.

By scaling the charge function $W$ we can similarly produce a lax pre-stability condition which is arbitrarily close to $(P_t,0)$ in $\Slice{\CC}\times \Hom{\Lambda}{\C}$, but does not satisfy $\delta$-lax support. This shows that $\delta$-lax support is not an open property if unless we restrict to lax pre-stability conditions with a fixed massless subcategory.

\end{example}

\subsection{The massive part of a lax stability condition}

A lax stability condition $\sigma$ on $\CC$ with massless subcategory $\NN$ induces a stability condition $\massive{\NN}(\sigma)$ on the quotient $\CC/\NN$. We refer to this as the `massive part' of the lax stability condition $\sigma$. 

\begin{proposition}
\label{prop:massive stability condition}
Let $\sigma = (P,Z)$ be a $\delta$-lax stability condition for some $\delta<\frac{1}{4}$ on $\CC$ with massless subcategory $\NN$, and let $\massive{\NN}(\sigma) \coloneqq (P_{\CC/\NN}, Z)$. Then
\begin{enumerate}
\item $\massive{\NN}(\sigma)$ is a strict pre-stability condition on $\CC/\NN$. 
\end{enumerate}
If, in addition, $\CC$ is Hom-finite, then 
\begin{enumerate}
\setcounter{enumi}{1}
\item $\norm{W}_{\massive{\NN}(\sigma)} \leq \norm{W}_{\massive{\NN}(\sigma),\delta} \leq \norm{W}_{\sigma,\delta}$ holds for all $W\in \Hom{\Lambda/\Lambda_\NN}{\C}$.
\item $P$ is well adapted to $\NN$, and $\massive{\NN}(\sigma)$ is a stability condition on $\CC/\NN$.
\end{enumerate}

\end{proposition}

\begin{proof} 
\begin{enumerate}[wide, labelwidth=!, labelindent=1em]
\item By \cref{prop:massless thick,,prop:quotient slicing}, $P$ is adapted to $\NN$ and so is compatible with the pair of slicings $(P_\NN,P_{\CC/\NN})$ on the massless subcategory $\NN$ and the quotient $\CC/\NN$. The charge $Z$ lies in the subspace $\Hom{\Lambda/\Lambda_\NN}{\C}$ and $Z(c) = m(c) \exp(i\pi\phi)$ with $m(c)>0$ for $0 \neq c\in P_{\CC/\NN}(\phi)$. Thus $\massive{\NN}(\sigma)$ is a strict pre-stability condition on $\CC/\NN$. 
\item
 By \cref{lem:seminorm inequalities}, $\norm{W}_{\massive{\NN}(\sigma)} \leq \norm{W}_{\massive{\NN}(\sigma),\delta}$. Now suppose $c\in P_{\CC/\NN}(\phi-\delta,\phi+\delta)$ is indecomposable in $\CC/\NN$. Replacing it by an isomorphic object in $\CC/\NN$, we may suppose $c\in P(\phi-\delta,\phi+\delta)$. Since $\CC$ is Hom-finite and $\delta<\frac{1}{4}$ we can decompose $c$ as $c'\oplus c''$ in $P(\phi-\delta,\phi+\delta)$ where $c'$ is a massive indecomposable object and $c''$ is massless; we allow the possibility $c''=0$. Thus $\norm{W}_{\massive{\NN}(\sigma),\delta} \leq \norm{W}_{\sigma,\delta}$ for $W\in \Hom{\Lambda/\Lambda_\NN}{\C}$ because $|W(c')| / |Z(c')| = |W(c)| / |Z(c)|$.
\item It remains to show that $\massive{\NN}(\sigma)$ has the support property, for then $P_{\CC/\NN}$ is locally finite so that $P$ is well-adapted to $\NN$. 
Let $c\in P_{\CC/\NN}(\phi)$. By replacing with an isomorphic object in $\CC/\NN$, we may assume without loss of generality that $c\in P(\phi)$. Using the assumption that $\CC$ is Hom-finite and applying \cref{prop:KRS property}, $c$ decomposes as a finite direct sum of indecomposable objects $\bigoplus_{i=1}^n c_i$. Since $c\in P(\phi)$ it follows that each $c_i\in P(\phi)$ and by replacing $c$ with an isomorphic object in $\CC/\NN$, we may assume without loss of generality that $c_i$ is massive for each $i \in \{1, \dots ,n\}$. Finally, using the fact that $\sigma$ satisfies $\delta$-lax support, we see that
\[
|Z(c)| = \sum_{i=1}^n |Z(c_i)| \geq \frac{1}{K} \sum_{i=1}^n \norm{c_i} \geq \frac{1}{K}\norm{c} \geq \frac{1}{K}\norm{c}_{{\CC/\NN}},
\]
where the final inequality follows from the definition of $\norm{\cdot}_{{\CC/\NN}}$ as the restriction of $\norm{\cdot}$ to the orthogonal complement to $(\Lambda_\NN)_\R \hookrightarrow \Lambda_\R$ with respect to the inner product defining the norm.
\qedhere
\end{enumerate}
\end{proof}

\cref{prop:massive stability condition} gives a positive answer to \cref{q:locally finite} when $P$ is the slicing of a lax stability condition with massless subcategory $\NN$.
It would be interesting to know if \cref{q:locally finite} has a positive answer when $P$ is the slicing of a lax pre-stability condition.

\begin{remark}
\label{rmk:lax quotient}
The above argument applies to a thick subcategory $\MM \subseteq \NN$ of the massless subcategory of $\sigma = (P,Z)$ to which the slicing is well adapted: then $\massive{\MM}(\sigma) \coloneqq (P_{\CC/\MM}, Z)$ is a lax stability condition on $\CC/\MM$ with $\norm{W}_{\massive{\MM}(\sigma),\delta} \leq \norm{W}_{\sigma,\delta}$ for $W\in \Hom{\Lambda/\Lambda_\MM}{\C}$. We need to assume $P$ is well adapted to $\MM$ because lax support for $\massive{\MM}(\sigma)$ does not imply that its slicing is locally finite when there are massless objects.
\end{remark}

The support property satisfied by a lax stability condition is stronger than that for the induced stability condition on the quotient. 

\begin{example}
\label{ex:lax supported versus weak}
Let $\sigma = (P_t,W)$ be the lax pre-stability condition on $\CC = \Db(\PP^1)$ with massless subcategory $\NN=\thick{}{\mathcal{O}_x : x\in \PP^1}$ defined in \cref{ex:not well-supported}. Recall that this is \emph{not} a lax stability condition as lax support fails. The quotient $\CC/\NN \cong \thick{\CC/\NN}{\mathcal{O}}$ is generated by a single object with charge $i$ and phase $\tfrac{1}{2}$ in the stability condition on the quotient $\massive{\NN}(\sigma)$. It follows that $\massive{\NN}(\sigma)$ satisfies the support property so that $\tau$ is a lax pre-stability condition which satisfies support on the quotient $\CC/\NN$. 
\end{example}

Recall that strict pre-stability conditions $\sigma=(P,Z)$ and $\tau=(Q,Z)$ with the same charge and with $d(P,Q)<1$ are equal \cite[Lem.~6.4]{Bridgeland07}. There is an analogue for lax pre-stability conditions; the only difference is that the slicing on the massless objects is not determined by the charge and so we must fix this too.

\begin{lemma}
\label{cor:lax uniqueness}
Suppose $\sigma=(P,Z)$ and $\tau=(Q,Z)$ are lax pre-stability conditions with the same massless subcategory $\NN$ and massless slicing $P_\NN=Q_\NN$. Then $\sigma=\tau$ when $d(P,Q)<1$.
\end{lemma}

\begin{proof}
The induced pre-stability conditions $\massive{\NN}(\sigma)$ and $\massive{\NN}(\tau)$ have the same charge and the distance between their slicings is $d(P_{\CC/\NN},Q_{\CC/\NN}) \leq d(P,Q) <1$. Hence $P_{\CC/\NN}=Q_{\CC/\NN}$ by \cite[Lem.~6.4]{Bridgeland07}. As $P_\NN=Q_\NN$ and the glued slicing is unique by \cref{prop:uniqueness of compatibility}, we get $P=Q$.
\end{proof}

\section{The space of lax stability conditions}
\label{sec:spaces of stability conditions}

\noindent
Let $\CC$ be a Hom-finite $\kk$-linear triangulated category and $\Lambda$ a finite rank lattice with a norm $\norm{\cdot}$ on $\Lambda \otimes \R$ and a surjective homomorphism $v\colon K(\CC) \to \Lambda$. Let $\Stab{\CC}$ be the set of stability conditions whose charges factor through $v$, equipped with the subspace topology from the inclusion $\Stab{\CC} \subset \Slice{\CC}\times \Hom{\Lambda}{\C}$ as in \cref{subsec:charges}; see \cite[\S 6]{Bridgeland07}. Recall that we have defined $\Slice{\CC}$ to be the set of locally finite slicings of $\CC$.

\begin{definition}
\label{def:space of lax stability conditions}
The \defn{space of lax stability conditions} is the subset
\[ \LaxStab{\CC} \subset \Slice{\CC}\times \Hom{\Lambda}{\C} \]
of lax stability conditions, equipped with the subspace topology. For a thick subcategory $\NN \subseteq \CC$, let $\LaxStabN{\CC}{\NN} \subset \LaxStab{\CC}$ denote the subspace where the massless subcategory is $\NN$. For each of these spaces, the \defn{charge map} $\cm$ is the second projection onto $\Hom{\Lambda}{\C}$.
\end{definition}

\subsection{Semi-norm neighbourhoods}
\label{subsec:semi-norm neighbourhoods}
The stability space $\Stab{\CC}$ has a basis of open neighbourhoods of the form
\[
\{ (Q,W) : d(P,Q)<\epsilon \ \text{and}\ \norm{W-Z}_\sigma <\sin(\pi \epsilon)\}
\]
for $(P,Z) \in \Stab{\CC}$ and $\epsilon >0$, see \cite[\S 6]{Bridgeland07}. We define analogous neighbourhoods of $\Slice{\CC}\times \Hom{\Lambda}{\C}$ which will be similarly useful for studying the topology of $\LaxStab{\CC}$.

For $\delta,\epsilon>0$ and a lax pre-stability condition $\sigma =(P,Z)$ define a subset
\[
B_\epsilon^\delta(\sigma) \coloneqq \{ (Q,W) : d(P,Q)<\epsilon \ \text{and}\ \norm{W-Z}_{\sigma,\delta} < \sin(\pi \epsilon) \} \subset \Slice{\CC}\times \Hom{\Lambda}{\C}
\]
using the generalised semi-norms of \cref{subsec:semi-norms}.
Even if $\sigma$ is has lax support, we do not know if $B_\epsilon^\delta(\sigma)$ is a subset of $\LaxStab{\CC}$ for small enough $\epsilon$ and $\delta$. This is why we will always consider $B_\epsilon^\delta(\sigma)$ as subsets of $\Slice{\CC}\times \Hom{\Lambda}{\C}$. Note that $B_\epsilon^\delta(\sigma) \supseteq B_\epsilon^{\delta'}(\sigma)$ for all $\delta \leq \delta'$.

The inequality $\norm{W-Z}_{\sigma,\delta} < \sin(\pi \epsilon)$ in the definition of $B_\epsilon^\delta(\sigma)$ implies that the charge $W(c)$ of massive indecomposable $\delta$-slim object $c$ lies in the bounded region of $\C$ described in the following lemma. This will be used in \cref{lem:thick subcats in nbhd}.

\begin{lemma} 
\label{lem:unpacking semi-norm}
Let $\sigma = (P,Z)$ be a lax pre-stability condition, $W \in \Hom{\Lambda}{\C}$ and $\delta, \epsilon > 0$, $\phi \in \R$. If $\norm{W-Z}_{\sigma,\delta} < \sin(\pi \epsilon)$ then for any $c\in P(\phi-\delta, \phi+\delta)$ massive and indecomposable
\begin{enumerate}
\item $(1-\sin(\pi\epsilon) )|Z(c)| < |W(c)| < (1+\sin(\pi\epsilon) )|Z(c)|$ and
\item $\frac{1}{\pi}\arg W(c) \in (\phi-\delta-\epsilon,\phi+\delta+\epsilon)$.
\end{enumerate}
\end{lemma}

\begin{lemma}
\label{lem:open seminorm neighbourhoods}
If $\sigma$ is a lax stability condition with $\delta$-lax support then $B_\epsilon^\delta(\sigma)$ is an open subset of $\Slice{\CC}\times \Hom{\Lambda}{\C}$ for any $\epsilon > 0$.
\end{lemma}

\begin{proof}
The lax stability condition $\sigma$ is $\delta$-full by \cref{prop:lax support and full}. Hence $\norm{\cdot}_{\sigma,\delta}$ is a semi-norm on $\Hom{\Lambda}{\C}$. Therefore both provisions in the definition of $B_\epsilon^\delta(\sigma)$ are open.
\end{proof}

So for $\delta$-lax stability conditions $\sigma$, the open subsets $B_\epsilon^\delta(\sigma)$ contains all sufficiently small metric balls about $\sigma$, but need not be contained within any such metric ball because $\norm{\cdot}_{\sigma,\delta}$ is only a \emph{semi}-norm. If $\sigma$ does not satisfy $\delta$-lax support then $B_\epsilon^\delta(\sigma)$ need not contain any metric ball about $\sigma$. The provision $\norm{W-Z}_{\sigma,\delta} < \sin(\pi \epsilon)$ is asymmetric in $W$ and $Z$ because $Z$ is the charge of $\sigma$. This asymmetry is illustrated in the example below. 

\begin{example}
\label{ex:seminorm balls asymmetric}
Let $\sigma_g = (P_g, Z_g)$ the strict geometric stability condition on $\CC = \Db(\PP^1)$ with charge $Z_g = -\deg + i \cdot \rank$, and slicing $P_g$ from \cref{ex:slicings}, \ie $P_g(1) = \langle \cO_x : x\in\PP^1 \rangle$ and $P_g(\phi) = \langle \cO(n) \rangle$ for $\phi = \frac{1}{\pi}\arg(-n+i) \in (0,1)$.
And let $\sigma_d = (P_d,Z_d) \coloneqq (P_g,0)$ be the lax stability condition with the same slicing but zero charge.
Then $d(P_g,P_d)=0$ and $\norm{ Z_g-Z_d }_{\sigma_g,\delta} = \norm{Z_g}_{\sigma_g,\delta} = 1$ and $\norm{ Z_d-Z_g }_{\sigma_d,\delta} = \norm{Z_g}_{\sigma_d,\delta} = 0$ because there are no massive $\sigma_d$-stable objects. Thus $\sigma_g \in B_\epsilon^\delta(\sigma_d)$ but $\sigma_d \not \in B_\epsilon^\delta(\sigma_g)$ for any $\delta, \epsilon>0$.
\end{example}

\subsection{Continuity of masses and phases}

The mass and extremal phases of any object vary continuously in $\LaxStab{\CC}$, as does the massless subcategory when we equip the set of thick subcategories, $\Thick{\CC}$, of $\CC$ with an Alexandrov topology. All these maps are locally constant if the charge is fixed.

\begin{proposition}
\label{prop:mass and phase continuity}
For each $0\neq c\in \CC$, the functions $\LaxStab{\CC} \to \R$ given by $\sigma \mapsto m_\sigma(c)$ and $\sigma \mapsto \phi^+_\sigma(c)$ and $\sigma \mapsto \phi^-_\sigma(c)$ are continuous.
\end{proposition}

\begin{proof}
Fix $c \in \CC$. The two functions $\Slice{\CC} \to \R$, $P \mapsto \phi^\pm_P(c)$ are continuous by definition of the metric on the space of slicings; hence $\phi^\pm_\sigma(c)$ are continuous as well.
To show that the mass is continuous consider $\sigma=(P,Z)\in \LaxStab{\CC}$. For sufficiently small $\epsilon >0$ and $\tau =(Q,W)$ with $d(P,Q)<\epsilon$ the HN filtration of $c$ with respect to $\tau$ is the concatenation of the filtrations of the $\sigma$-semistable factors $\{c_i\}$ of $c$. Hence 
$
m_\tau(c) - m_\sigma(c) = \sum_i ( m_\tau(c_i) - m_\sigma(c_i) ).
$
Therefore it suffices to consider the case in which $c$ is $\sigma$-semistable. Assume $c\in P(\phi)$ and let $S\subset Q(\phi-\epsilon,\phi+\epsilon)$ be a multi-set of $\tau$-stable factors of $c$. By the triangle inequality and elementary trigonometry
\[
|W(c)| \leq \sum_{s\in S} |W(s)| \leq \frac{|W(c)|}{\cos(2\pi \epsilon)}
\]
and therefore 
\[
| m_\tau(c) - m_\sigma(c) | = \left| \sum_{s\in S}|W(s)| - |Z(c)| \right| 
\leq \max\left\{ |Z(c)|-|W(c)| , \frac{|W(c)|}{\cos(2\pi \epsilon)} -|Z(c)|\right\}.
\]
Applying the triangle inequality to each term on the right-hand side and the operator norm bounds $|W(c)-Z(c)| \leq \norm{W-Z} \cdot \norm{c}$ and $|Z(c)| \leq \norm{Z} \cdot \norm{c}$ we obtain
\[
| m_\tau(c) - m_\sigma(c)| \leq \max \left\{ \norm{W-Z} , \frac{\norm{W-Z} + (1-\cos(2\pi\epsilon))\norm{Z}}{\cos(2\pi\epsilon)}\right\} \norm{c}.
\]
Requiring $\norm{W-Z}<\epsilon$, in addition to $d(P,Q)<\epsilon$, we see the bound can be made arbitrarily small by reducing $\epsilon$. The result follows.
(This proof doesn't require lax support.)
\end{proof}

\begin{proposition}
\label{prop:local constancy 1}
For each $c\in \CC$, the function $\LaxStab{\CC}$, $\sigma \mapsto m_\sigma(c)$ is locally constant on the fibres of the charge map $\cm \colon\LaxStab{\CC} \to \Hom{\Lambda}{\C}$. Moreover, the set of phases of the massive semistable factors of $c$ is also locally constant on the fibres of $\cm$.
\end{proposition}

\begin{proof}
Fix $c\in \CC$ and $\sigma \in \LaxStab{\CC}$. Then for $\tau$ sufficiently close to $\sigma$ the HN filtration of $c$ with respect to $\tau$ is the concatenation of the filtrations of the $\sigma$-semistable factors of $c$. Suppose $c_i$ is one of these $\sigma$-semistable factors. The charges of the $\tau$-semistable factors of $c_i$ lie in a cone of angle $2\pi\epsilon$ in $\C$, centred on the phase $\phi_i$ of $c_i$. 
Assume that $\cm(\tau) = \cm(\sigma)$. Suppose that two or more $\tau$-semistable factors of $c_i$ are massive. One of these factors must have phase which is strictly greater than $\phi_i$. Taking the $\tau$-semistable factor of $c_i$ of maximal phase and then considering its $\sigma$-semistable factor of maximal phase, we find an object that would destabilise $c_i$ with respect to $\sigma$. Therefore, all but one of the $\tau$-semistable factors of $c_i$ must be massless and the unique massive factor must have the same charge, and in particular the same mass, as $c_i$. It follows that $m_\tau(c) = m_\sigma(c)$, and also that the sets of phases of the massive factors of $c$ with respect to $\sigma$ and $\tau$ are the same.
\end{proof}

\begin{corollary}
\label{cor:local constancy 2}
$\LaxStab{\CC} \to \Thick{\CC}$, $\sigma \mapsto \NN_\sigma$ and $\massive{\NN} \colon \LaxStabN{\CC}{\NN} \to \Stab{\CC/\NN}$ are locally constant on the fibres of the charge map $\cm \colon \LaxStab{\CC} \to \Hom{\Lambda}{\C}$.
\end{corollary}

\begin{proof}
The massless subcategory $\NN_\sigma$ and the semistable objects of the stability condition on the quotient $\massive{\NN_\sigma}(\sigma)$ are locally constant on the fibres of the charge map by \cref{prop:local constancy 1}. By construction the charge of $\massive{\NN_\sigma}(\sigma)$ is constant. 
\end{proof}

\begin{lemma}
\label{lem:thick subcats in nbhd}
Let $\sigma$ be a lax pre-stability condition and $\tau \in B_\epsilon^\delta(\sigma)$ for some $0<\epsilon<\tfrac{1}{6}$ and $\delta<0$. Then $\NN_\tau \subseteq \NN_\sigma$ with equality if and only if $\cm(\tau) \in \Hom{\Lambda/\Lambda_{\NN_\sigma}}{\C}$.
\end{lemma}

Before proving this we state a corollary. We equip $\Thick{\CC}$ with the Alexandrov topology from the opposite of the inclusion partial order, so that $\{ \MM \in \Thick{\CC} \colon \MM \subseteq \NN \}$ is open for any $\NN\in \Thick{\CC}$.
Continuity of massless subcategories follows immediately from the previous lemma, because the subsets $B_\epsilon^\delta(\sigma)$ are open for $\sigma \in \LaxStab{\CC}$.

\begin{corollary}
\label{cor:thick subcats in nbhd}
The map $\LaxStab{\CC} \to \Thick{\CC}, \tau \mapsto \NN_\tau$ is continuous.
\end{corollary}

\begin{proof}[Proof of \cref{lem:thick subcats in nbhd}]
Let $\sigma=(P,Z)$ and $\tau=(Q,W)$. For $\sigma$-stable $c\in \CC$, we have the inequality
$(1-\sin(\pi\epsilon))|Z(c)| \leq |W(c)|$.
This is evident if $c\in \NN_\sigma$, and follows from $\tau \in B_\epsilon^\delta(\sigma)$, \cref{lem:unpacking semi-norm} and the fact that stable objects are $\delta$-slim and indecomposable if $c\not \in \NN_\sigma$.

Now suppose $b \in Q(\phi)$ is $\tau$-semistable. Let $S$ be a (finite) multi-set of $\sigma$-stable factors of $b$. Since $\tau\in B_\epsilon^\delta(\sigma)$ we know that $S\subseteq P(\phi-\epsilon,\phi+\epsilon) \subseteq Q(\phi-2\epsilon,\phi+2\epsilon)$. Hence, using elementary trigonometry and the above inequality we have
\begin{align*}
|W(b)| \geq \cos(2\pi\epsilon) \sum_{s\in S} |W(s)| \geq (1-\sin(\pi\epsilon)) \cos(2\pi\epsilon)\sum_{s\in S} |Z(s)|.
\end{align*}
Therefore $m_\tau(b) \geq (1-\sin(\pi\epsilon)) \cos(2\pi\epsilon) m_\sigma(b)$. In particular, if $b\in \NN_\tau$ then $b\in \NN_\sigma$. 

For the equality statement, use $W\in \Hom{\Lambda/\Lambda_{\NN_\sigma}}{\C}$ if and only if $W(c)=0$ for all $c\in \NN_\sigma$. In particular, if $\NN_\tau = \NN_\sigma$ then $W \in \Hom{\Lambda/\Lambda_{\NN_\sigma}}{\C}$. 
Conversely, if $W\in \Hom{\Lambda/\Lambda_{\NN_\sigma}}{\C}$ and $c\in \NN_\sigma$, let $S$ be the multi-set of $\tau$-semistable factors of $c$. We have
$m_\tau(c) = \sum_{s\in S} |W(s)| = 0$
because $S\subset \NN_\sigma$ by \cref{lem:local persistence of massless factors}. 
% This is where $\delta < 1/6$ is required.
Hence $\NN_\sigma \subseteq \NN_\tau$ and we get equality.
\end{proof}

\subsection{The massless and massive parts of a lax stability condition}
\label{subsec:massive and massless parts}

Consider a lax stability condition with prescribed massless subcategory $\sigma=(P,Z) \in \LaxStabN{\CC}{\NN}$. We have defined its massive part $\massive{\NN}(\sigma)$ in \cref{prop:massive stability condition} which is a strict stability condition on $\CC/\NN$. As a counterpart, we now define
the \defn{restriction}, or \defn{massless part}, $\restrict{\NN}(\sigma) \coloneqq (P_\NN, 0)$. As all objects are massless, lax support is trivially satisfied. The massless and massive parts define maps:
\[ 
\begin{array}{r @{\ } l @{\qquad} l}
 \restrict{\NN} \colon& \LaxStabN{\CC}{\NN} \to \LaxStabN{\NN}{\NN},  & \sigma = (P,Z) \mapsto \restrict{\NN}(\sigma) = (P_\NN,0), \qquad \text{and} \\
 \massive{\NN}  \colon& \LaxStabN{\CC}{\NN} \to \Stab{\CC/\NN},       & \sigma = (P,Z) \mapsto \massive{\NN}(\sigma) = (P_{\CC/\NN},Z). \\
\end{array}
\]
When $\NN=0$ the map $\massive{\NN}$ is the identity and $\LaxStabN{\CC}{0} = \Stab{\CC}$ is the usual space of stability conditions. At the other extreme when $\NN=\CC$ it is the map to a point; in this case $\LaxStabN{\CC}{\CC} \cong \Slice{\CC}$ is homeomorphic to the space of locally finite slicings because when every object is massless the charge is zero and the lax support property is vacuous. 

The construction of $\restrict{\NN}(\sigma)$ extends to an open neighbourhood of $\sigma$ in $\LaxStabN{\CC}{\NN}$.

\begin{lemma}
\label{lem:restriction map}
Given $\sigma=(P,Z)\in\LaxStabN{\CC}{\NN}$ and $0 < \epsilon < \tfrac{1}{6}$, the map
\begin{align*}
\restrict{\NN} \colon \{ \tau=(Q,W) \in \LaxStab{\CC} : d(P,Q)<\epsilon \} &\to \LaxStab{\NN} \\
\tau = (Q,W) &\mapsto \restrict{\NN}(\tau) = (Q\cap\NN, W_\NN)
\end{align*}
is well defined and restricts to a map $\restrict{\NN} \colon B_\epsilon^\delta(\sigma) \cap \LaxStab{\CC} \to \LaxStab{\NN}$ for any $\delta>0$.
\end{lemma}

\begin{proof}
The restriction, $\restrict{\NN}(\tau)$, is a lax pre-stability condition by \cref{lem:local persistence of massless factors}. It satisfies $\delta$-lax support for the same $\delta>0$ and support constant $K>0$ as does $\tau$ because we reduce to checking the definition for massive indecomposable $\delta$-slim objects in $\NN$.
\end{proof}

Massive quotients and massless restrictions are contractions and therefore continuous:

\begin{lemma}
\label{lem:contractions}
Let $\sigma$ and $\tau$ be lax pre-stability conditions on $\CC$ such that their slicings are adapted to a thick subcategory $\NN$ of $\CC$. Then
\[ 
  d( \massive{\NN}(\sigma) , \massive{\NN}(\tau) ) \leq d(\sigma,\tau) 
\qquad \text{and} \qquad
  d( \restrict{\NN}(\sigma), \restrict{\NN}(\tau) ) \leq d(\sigma,\tau). 
\] 
\end{lemma}

\begin{proof}
Writing $\sigma = (P,Z)$ and $\tau = (Q,W)$, \cref{cor:slicing inequalities} gives the corresponding inequalities for the slicings: $d(P_\NN,Q_\NN) \leq d(P,Q)$ and $d(P_{\CC/\NN},Q_{\CC/\NN}) \leq d(P,Q)$.

Write $U_\NN \coloneqq U|_{\Lambda_\NN} \in \Hom{\Lambda_\NN}{\C}$ for the restriction of any $U \in \Hom{\Lambda}{\C}$. We consider $U_\NN \in \Hom{\Lambda}{\C}$, as always using the embedding $\Hom{\Lambda_\NN}{\C} \inj \Hom{\Lambda}{\C}$ arising from the inner product on $\Lambda\otimes \R$, see \cref{rmk:charge space splitting}. Write, in this proof only, $U'_\NN \coloneqq U-U_\NN$ for the component in the orthogonal complement $\Hom{\Lambda/\Lambda_\NN}{\C}$ to $\Hom{\Lambda_\NN}{\C}$. Because $U = U_\NN + U'_\NN$ is an orthogonal decomposition: $\norm{U}^2 = \norm{U_\NN}^2 + \norm{U'_\NN}^2$. Applied to $U \coloneqq Z-W$, this shows that $\norm{Z_\NN-W_\NN} \leq \norm{Z-W}$ and $\norm{Z'_\NN-W'_\NN} \leq \norm{Z-W}$. The claimed inequalities follow because $\restrict{\NN}(\sigma) = (P_\NN,Z_\NN)$ and $\massive{\NN}(\sigma) = (P_{\CC/\NN}, Z'_\NN)$, and likewise for $\tau$.
\end{proof}

The construction of the massive part extends to the closure of $\LaxStabN{\CC}{\NN}$.

\begin{proposition}
\label{prop:massive part}
The map $\massive{\NN}\colon \LaxStabN{\CC}{\NN} \to \Stab{\CC/\NN}$ extends to a continuous map
\[
\massive{\NN} \colon \overline{\LaxStabN{\CC}{\NN}} \to \LaxStab{\CC/\NN}
\]
on the closure in $\Slice{\CC} \times \Hom{\Lambda}{\C}$. Moreover, $\massive{\NN}(\sigma) \in \Stab{\CC/\NN} \iff \sigma \in \LaxStabN{\CC}{\NN}$.
\end{proposition}

\begin{proof}
Recall we assume that $\CC$ is Hom-finite.
Therefore by \cref{prop:massive stability condition} the assignment $\sigma =(P,Z) \mapsto (P_{\CC/\NN},Z) = \massive{\NN}(\sigma)$ defines a map $\LaxStabN{\CC}{\NN} \to \Stab{\CC/\NN}$. It is continuous by \cref{lem:contractions}. The same argument shows the extension to $\overline{\LaxStabN{\CC}{\NN}}$ is continuous where defined.

Now suppose $\sigma \in \overline{\LaxStabN{\CC}{\NN}}$. Then $Z(c)=0$ for all $c\in \NN$ and $P$ restricts to $\NN$ by \cref{lem:local persistence of massless factors}. It follows that the massless subcategory $\NN_\sigma \supseteq \NN$. We claim that the slicing $P$ is well adapted to $\NN$. If this is the case then $\massive{\NN}(\sigma)$ is a well-defined lax stability condition on $\CC/\NN$ by \cref{rmk:lax quotient}. Regarding the final statement: if $\massive{\NN}(\sigma) \in \Stab{\CC/\NN}$ then the massless subcategory of $\massive{\NN}(\sigma)$ is $0 \subseteq \CC/\NN$, hence $\NN_\sigma = \NN$, \ie $\sigma \in \LaxStabN{\CC}{\NN}$.

It remains to prove the claim. First we show $P$ is adapted to $\NN$. Let $0\to a \to b\to c\to 0$ be a short exact sequence in the abelian category $P(I)$ with $b\in \NN\cap P(I)$. As $I$ is an half-open interval, there is a subinterval $J \subsetneq I$ with $a, b, c \in P(J)$. For sufficiently close $\tau = (Q,W)$ in $\LaxStabN{\CC}{\NN}$, we have $P(I) \subseteq Q(J)$. Hence there is a strict length one interval $I'$ such that $0\to a \to b\to c\to 0$ is a short exact sequence in $Q(I')$. Therefore $a, c\in \NN$ because $Q$ is adapted to $\NN$ by \cref{prop:massless thick}. Since $\NN\cap P(I)$ is clearly extension-closed it is therefore a Serre subcategory of $P(I)$. 

Now we show $P$ is well adapted to $\NN$.
By \cref{prop:quotient slicing} there is a slicing $P_{\CC/\NN}$ on the quotient. Since $d(P_{\CC/\NN}, Q_{\CC/\NN})\leq d(P,Q)$ and $\sigma \in \overline {\LaxStabN{\CC}{\NN}}$ this slicing is the limit of slicings appearing in $\Stab{\CC/\NN}$. These are locally finite, indeed the extension closures of semistable objects with phases in any interval of length strictly less than one are length categories. It follows that $P_{\CC/\NN}$ is locally finite. 
\end{proof}

\subsection{Support propagation}
\label{subsec:propagation}

Support propagates in the stability space: nearby deformations of a stability condition are also stability conditions, and not just pre-stability conditions. The key result of this section is the analogue for lax stability conditions. Here we require an extra assumption because the lax support property does not control the charges of massless objects.

We begin by relating the semi-norms associated to nearby lax pre-stability conditions. These are equivalent for nearby pre-stability conditions by \cite[Lem.~6.2]{Bridgeland07}. Our situation is more complicated because the massless subcategories may differ and we have a family of not necessarily equivalent semi-norms for each lax pre-stability condition.

\begin{lemma}
\label{lem:semi-norm bound}
Let $\sigma=(P,Z)$, $\tau=(Q,W)$ be two lax pre-stability conditions and $0<\epsilon<\delta < \tfrac{1}{6}$ such that $\tau \in B_\epsilon^\delta(\sigma)$.
Let $U\in\Hom{\Lambda}{\C}$ and denote by $U_{\NN_\sigma}$ its restriction to $\Lambda_{\NN_\sigma}$. Then
\begin{align*}
    (1-\sin\pi\epsilon) \, \norm{U}_{\tau,\delta-\epsilon} &\leq
       \max \big\{ \norm{U}_{\sigma,\delta} , (1-\sin\pi\epsilon) \, \norm{U_{\NN_\sigma}}_{\restrict{\NN_\sigma}(\tau),\delta-\epsilon} \big\} , \\
    \norm{U}_{\sigma,\delta} &\leq (1+\sin \pi\epsilon) \, \norm{U}_{\tau,\delta+\epsilon} .
\end{align*}
\end{lemma}

\begin{corollary}
\label{cor:semi-norm bound}
Let $\sigma$, $\tau$ and $0< \epsilon < \delta < \tfrac{1}{6}$ and $U\in\Hom{\Lambda}{\C}$ be as in the lemma.

If $\NN_\sigma = \NN_\tau$ or $U\in \Hom{\Lambda/\Lambda_{\NN_\sigma}}{\C}$ then
\[
  (1-\sin \pi\epsilon) \, \norm{U}_{\tau,\delta-\epsilon} \leq \norm{U}_{\sigma,\delta} \leq (1+\sin\pi\epsilon) \, \norm{U}_{\tau,\delta+\epsilon} .
\]
\end{corollary}

\begin{proof}
By \cref{lem:unpacking semi-norm}, $\tau\in B_\epsilon^\delta(\sigma)$ amounts to
\[
Q(\phi-\delta+\epsilon,\phi+\delta-\epsilon) \subseteq P(\phi-\delta,\phi+\delta) \subseteq Q(\phi-\delta-\epsilon,\phi+\delta+\epsilon)
\]
for each $\phi\in \R$ and inequalities $(1-\sin\pi\epsilon)|Z(c)| < |W(c)| < (1+\sin\pi\epsilon) |Z(c)|$ for each massive indecomposable $(\sigma,\delta)$-slim object $c$. Moreover, by \cref{lem:thick subcats in nbhd} 
the massless subcategory of $\tau$ is contained in that of $\sigma$, \ie $\NN_\tau \subseteq \NN_\sigma$. Finally, recall that $c$ is indecomposable in $P(I)$ for some interval $I$ if and only if it is indecomposable in $\CC$, and analogously for the slicing $Q$. Therefore,
\begin{align*}
\norm{U}_{\sigma,\delta}
 & = \sup \left\{ \tfrac{|U(c)|}{|Z(c)|} \colon c \not \in \NN_\sigma \text{ indecomposable $(\sigma,\delta)$-slim} \right\}\\
 & \leq (1+\sin \pi\epsilon) \, \sup \left\{ \tfrac{|U(c)|}{|W(c)|} \colon c \notin \NN_\sigma \text{ indecomposable $(\sigma, \delta)$-slim} \right\} \\
 & \leq (1+\sin \pi\epsilon) \, \sup \left\{ \tfrac{|U(c)|}{|W(c)|} \colon c \notin \NN_\tau \text{ indecomposable $(\tau, \delta+\epsilon)$-slim} \right\} \\
 & =    (1+\sin \pi\epsilon) \, \norm{U}_{\tau, \delta+\epsilon} .
\end{align*}
For the third inequality we split the supremum over $c\not \in \NN_\tau$ into suprema over $c\not \in \NN_\sigma$ and $c\in \NN_\sigma \setminus \NN_\tau$ to estimate 
\begin{align*}
\norm{U}_{\tau,\delta-\epsilon}
 & = \sup \left\{ \tfrac{|U(c)|}{|W(c)|} \colon c \not \in \NN_\tau \text{ indecomposable $(\tau,\delta-\epsilon)$-slim} \right\} \\
 & = \max\Big\{
     \sup \left\{ \tfrac{|U(c)|}{|W(c)|} \colon c \not \in \NN_\sigma \text{ indecomposable $(\tau,\delta-\epsilon)$-slim} \right\} , \\
 &   \qquad \qquad \sup \left\{ \tfrac{|U(c)|}{|W(c)|} \colon c \in \NN_\sigma \setminus \NN_\tau \text{ indecomposable $(\tau,\delta-\epsilon)$-slim} \right\}  \Big\} \\
 & \leq \max \left\{
     \tfrac{1}{1-\sin\pi\epsilon} \, \sup \left\{ \tfrac{|U(c)|}{|Z(c)|} \colon c \notin \NN_\sigma \text{ indecomposable $(\sigma,\delta)$-slim} \right\} ,
     \norm{U_{\NN_\sigma}}_{\restrict{\NN_\sigma}(\tau),\delta-\epsilon} \right\} \\
 & = \max\left\{ \tfrac{1}{1-\sin\pi\epsilon}\norm{U}_{\sigma,\delta} \, , \, \norm{U_{\NN_\sigma}}_{\restrict{\NN_\sigma}(\tau),\delta-\epsilon} \right\}.
\end{align*}
Regarding the corollary, each of $\NN_\sigma = \NN_\tau$ and $U\in \Hom{\Lambda/\Lambda_{\NN_\sigma}}{\C}$ implies $\norm{U_{\NN_\sigma}}_{\restrict{\NN_\sigma}(\tau),\delta-\epsilon} = 0$, simplifying the first inequality of the lemma.
\end{proof}

\begin{corollary}
\label{cor:propagation}
Let $0 < \epsilon < \delta < \tfrac{1}{6}$ and $\sigma$, $\tau$ two lax pre-stability conditions with $\tau \in B_\epsilon^\delta(\sigma)$. Then the following hold:
\begin{enumerate}
\item If $\tau$ has $(\delta+\epsilon)$-lax support then $\sigma$ has $\delta$-lax support. 
\item If $\sigma$ has $\delta$-lax support and $\restrict{\NN_\sigma}(\tau)$ has $(\delta-\epsilon)$-lax support then $\tau$ has $(\delta-\epsilon)$-lax support.
\end{enumerate}
\end{corollary}

\begin{proof}
Feed the inequalities of \cref{lem:semi-norm bound} into the characterisation of lax support via fullness from \cref{prop:lax support and full}. The restriction $\restrict{\NN_\sigma}(\tau)$ is well defined by \cref{lem:restriction map}.
\end{proof}

\begin{remarks}
The corollary gives a new proof of support propagation for strict stability conditions $\sigma$, using our characterisation of support via indecomposable $\delta$-slim objects. To see this, take $\tau$ with massless subcategory $\NN_\sigma$ or, equivalently, $\cm(\tau) \in \Hom{\Lambda/\Lambda_{\NN_\sigma}}{\C}$. Then the demand that $\restrict{\NN_\sigma}(\tau)$ satisfies lax support is trivial, as there are no $\tau$-massive objects in $\NN_\sigma$. Thus $\tau$ has $(\delta-\epsilon)$-lax support when $\sigma$ has $\delta$-lax support. 
\end{remarks}

\subsection{Deforming lax stability conditions}
\label{subsec:deformations}
\label{subsec:tangential deformation}
\label{subsec:normal deformation}
\label{subsec:fibrewise deformation}

\noindent
The technical heart of the theory of stability conditions is \cref{thm:deformation theorem} which governs their deformation. We cannot expect such a simple result for lax stability conditions, but it turns out that it is still possible to deform them in a reasonable way. The heuristic is that the massive and massless parts of a lax stability condition deform independently.

For a lax stability condition $\sigma = (P,Z) \in \LaxStabN{\CC}{\NN}$, the charge space decomposes as 
\[ 
\Hom{\Lambda}{\C} \cong \Hom{\Lambda_\NN}{\C} \oplus \Hom{\Lambda/\Lambda_\NN}{\C},
\]
 with $\cm(\sigma)=Z\in\Hom{\Lambda/\Lambda_\NN}{\C}$. Here, as elsewhere, we consider $\Hom{\Lambda_\NN}{\C}$ as a subspace of $\Hom{\Lambda}{\C}$ using the splitting arising from the inner product on $\Lambda\otimes \R$ --- see \cref{rmk:charge space splitting}. It is geometrically appealing to distinguish three (not mutually exclusive) cases of deformation:

\begin{enumerate}
\item 
A \defn{tangential deformation} of $\sigma$ is given by varying the charge in $\Hom{\Lambda/\Lambda_\NN}{\C}$. Such a deformation fixes the massless subcategory and hence stays inside $\LaxStabN{\CC}{\NN}$.
\item
A \defn{normal deformation} of $\sigma$ is given by varying the charge in $\Hom{\Lambda_\NN}{\C}$. Such a deformation moves out of $\LaxStabN{\CC}{\NN}$ into $\LaxStabN{\CC}{\MM}$ for some thick subcategory $\MM$ of $\NN$. We think of this as deforming in a normal direction to the stratum.
\item 
A \defn{fibrewise deformation} of $\sigma$ takes place when the charge is fixed, \ie only the slicing is deformed. In the strict setting, the charge map $\Stab{\CC}\to\Hom{\Lambda}{\C}$ has discrete fibres, so there are no non-trivial fibrewise deformations. However for the lax stability condition $\sigma$ we can potentially vary the slicing on $\NN$ in a continuous way, as this is not controlled by the charge.
\end{enumerate}
Tangential deformations are treated in \cref{prop:deformations in boundary}, normal deformations to $\Stab{\CC} $ in \cref{thm:deformation to non-lax}, and more general normal deformations to $\LaxStab{\CC}$ together with fibrewise deformations in \cref{prop:deformation off stratum}.

First, we consider tangential deformations where the charge is varied in $\Hom{\Lambda/\Lambda_\NN}{\C}$ and the massless slicing remains fixed.

\begin{proposition}
\label{prop:deformations in boundary}
Let $0 < \epsilon < \delta < \tfrac{1}{8}$, $\sigma = (P,Z) \in \LaxStabN{\CC}{\NN}$ with $\delta$-lax support and $W\in \Hom{\Lambda/\Lambda_\NN}{\C}$. If $\norm{W-Z}_{\sigma,\delta}<\sin(\pi\epsilon)$ then there is a unique lax stability condition $\tau=(Q,W) \in \LaxStabN{\CC}{\NN}$ with $d(P,Q)<\epsilon$ and massless slicing $Q_\NN=P_\NN$.
\end{proposition}

\begin{proof}
Recall we assume $\CC$ is Hom-finite. By \cref{prop:massive stability condition} the massive part $\massive{\NN}(\sigma) = (P_{\CC/\NN},Z) \in \Stab{\CC/\NN}$ is a strict stability condition on the quotient. Moreover
$\norm{W-Z}_\sigma \leq \norm{W-Z}_{\massive{\NN}(\sigma), \delta} \leq \norm{W-Z}_{\sigma,\delta}<\sin(\pi\epsilon)$
where the first inequality is \cref{lem:seminorm inequalities} and the second one \cref{prop:massive stability condition} again.
Therefore, \cref{thm:deformation theorem} provides a unique stability condition $(Q_{\CC/\NN},W) \in \Stab{\CC/\NN}$ with $d(P_{\CC/\NN}, Q_{\CC/\NN}) < \epsilon$. 

By \cref{prop:glueing slicings}, the slicings $P_\NN$ and $Q_{\CC/\NN}$ glue to a locally finite slicing $Q$ on $\CC$. 
Since $Q(\phi) \subseteq Q_{\CC/\NN}(\phi)$ for all $\phi\in \R$, the slicing $Q$ is compatible with the charge $W$. Thus $\tau=(Q,W)$ is a lax pre-stability condition with massless subcategory $\NN$, restricted slicing $Q_\NN=P_\NN$, and, by \cref{lem:slicing distance}, $d(P,Q)<\epsilon$. Uniqueness follows from \cref{cor:lax uniqueness}. 

Finally, $\tau$ satisfies lax support by \cref{cor:propagation} and so is a lax stability condition.
\end{proof}

Next we prove a variant of \cref{thm:deformation theorem} which shows that a lax stability condition with massless subcategory $\NN$ can be deformed in the normal direction with respect to a suitably small stability condition on $\NN$ to obtain a stability condition on $\CC$.

\begin{proposition}
\label{thm:deformation to non-lax}
Let $0 < \epsilon < \delta < \tfrac{1}{8}$, $\sigma=(P,Z) \in \LaxStabN{\CC}{\NN}$ with $\delta$-lax support and $\tau_\NN=(Q_\NN,W_\NN) \in \Stab{\NN}$.
If $d(P_\NN,Q_\NN)<\epsilon$ and $\norm{W_\NN}_{\sigma,\delta}<\sin(\pi\epsilon)$ then there is a unique stability condition $\tau =(Q,W) \in B_\epsilon^\delta(\sigma) \cap \Stab{\CC}$ with charge $W=Z+W_\NN$ and restriction $\restrict{\NN}(\tau)=\tau_\NN$.

Moreover, the stability condition $\tau$ depends continuously on $\sigma$ and $\tau_\NN$.
\end{proposition}

\begin{proof}
We begin by constructing $\tau=(Q,W) $ as a pre-stability condition, \ie we construct a slicing $Q$ which is compatible with $W$, using the method of \cite[\S7]{Bridgeland07}.

Recall that $P(s,t)$ is a \defn{thin subcategory} if $0<t-s<1-2\epsilon$ where $0<\epsilon<\tfrac{1}{8}$ and that 
 this implies $P(s,t)$ is quasi-abelian by \cite[Lem.~4.3]{Bridgeland07}. The charge $W=Z+W_\NN$ defines a skewed stability function \cite[Def.~4.4]{Bridgeland07} on any thin subcategory $P(s,t)$ by the obvious composite $K(P(s,t)) \to K(\CC) \to \Lambda \to \C$, \ie this group homomorphism takes every non-zero object into a rotated copy of the strict half-plane $\U \cup \R_{<0}$. 

To see why, suppose that $c\in P(\phi)$ for some $\phi\in (s,t)$. Let $A$ be a finite multi-set of stable factors of $c$ in the quasi-abelian length category $P(\phi)$. 
If, on the one hand, $a\in A$ is a massless stable object in $\NN$ then $a\in Q_\NN(\phi-\epsilon,\phi+\epsilon)$ and so $W(a) = (Z+W_\NN)(a) = W_\NN(a)$ is non-zero and therefore one can assign the phase $\frac{1}{\pi}\arg W_{\NN}(a) \in (s-\epsilon,t+\epsilon)$ to $a$.
On the other hand, if $a\not \in \NN$ is a massive stable factor then, by the definition of $\norm{\cdot}_{\sigma,\delta}$,
\[
|W(a)-Z(a)| < \sin(\pi\epsilon) |Z(a)|
\]
because $\norm{W-Z}_{\sigma,\delta} = \norm{W_\NN}_{\sigma,\delta} <\sin(\pi\epsilon)$. Therefore $W(a)\neq 0$ and differs in phase from $Z(a)$ by less than $\epsilon$, and again the phase of $a$ with respect to $W$ lies in $(s - \epsilon, t +\epsilon)$. 
We conclude that $W(c) = \sum_{a\in A} W(a) \neq 0$ with phase $\frac{1}{\pi} \arg W(c) \in (s-\epsilon,t+\epsilon)$.

The remainder of the proof follows that of \cref{thm:deformation theorem} in \cite[\S7]{Bridgeland07} {\it verbatim}. This is possible because, after the above initial step of showing that $W$ defines a skewed stability function on each thin subcategory, the charge $Z$ and the masses of objects with respect to $\sigma$ play no role in the proof, one only uses the locally finite slicing $P$. Therefore the same argument goes through and we can construct a unique pre-stability condition $\tau=(Q,W)$ with $d(P,Q)<\epsilon$. 

By \cref{lem:local persistence of massless factors} the slicing $Q$ restricts to a slicing $Q\cap\NN$ on $\NN$, and, by \cref{cor:slicing inequalities}, we have $d(P_\NN,Q\cap\NN)<\epsilon$. Since $\sigma$ has massless subcategory $\NN$ we know $Z\in \Hom{\Lambda/\Lambda_\NN}{\C}$ so that $W|_{\Lambda_\NN}=W_\NN$. Therefore, $\tau=(Q,W)$ restricts to a pre-stability condition $\restrict{\NN}(\tau)$ on $\NN$ with charge $W_\NN$ and slicing within distance $2\epsilon$ of $Q_\NN$. By \cref{cor:lax uniqueness} it follows that $\restrict{\NN}(\tau)= \tau_\NN$ as claimed.

As $\tau$ satisfies lax support by \cref{cor:propagation}, it is a stability condition by \cref{lem:lax and strict support}. 

The charge $W = Z+W_\NN$ obviously depends continuously on $\sigma$ and $\tau_\NN$. Since the charge uniquely determines the slicing in a neighbourhood of $\tau$, continuity for the slicing follows. 
\end{proof}

\begin{example}
\label{ex:norm example}
Let $\CC = \Db(\PP^1)$ and $\Lambda=K(\PP^1) \cong \Z^2$ with basis $[\mathcal{O}], [\mathcal{O}_x]$. The inner product is chosen so that this basis is orthonormal.
Let $\sigma=(P,Z)$ be the lax stability condition with charge $Z(\mathcal{O})=0$, $Z(\mathcal{O}_x)=-1$ and slicing $P=P_b$ from \cref{ex:slicings}, \ie $P(1) = \Clext{ \cO_x, \cO(n), \cO(-n)[1] : x\in\PP^1, n\in\N_{>0}}$ and $P(\tfrac{1}{2}) = \Clext{\cO}$.
The massless subcategory $\NN=\thick{}{\cO}$ and, for any $\delta\geq 0$, the massive indecomposable $\delta$-slim objects are, up to shifts, the indecomposable torsion sheaves, and the line bundles $\mathcal{O}(n)$ for $n\neq 0$.
% This can be seen by considering the cone spanned by the indecomposable massive objects in $K(\PP^1)$ and observing that these objects correspond to the primitive elements in this cone.
This lax stability condition can be deformed to a strict one using the previous result. Let $\tau_\NN = (Q_\NN,W_\NN)$ where $Q_\NN=P_\NN$ is the restricted slicing with $Q_\NN(\tfrac{1}{2}) = \clext{\cO}$ and $W_\NN(\cO) = ri$ for some $r>0$. Considered as a charge in $\Hom{\Lambda}{\C}$ via the orthogonal splitting we also have $W_\NN(\cO_x)=0$, so that all torsion sheaves are in the kernel of $W_\NN$. Therefore
\[
\norm{W_\NN}_\sigma = \sup\left\{ \frac{|W_\NN(c)|}{|Z(c)|} : \text{massive indecomposable $\delta$-slim $c$} \right\} = \sup\left\{ \frac{r}{|n|} : n\neq 0\right\} = r
\]
and the assumptions of \cref{thm:deformation to non-lax} are satisfied. The deformed stability condition $\tau=(Q,W)$ has charge $W(\mathcal{O}_x)=-1$, $W(\mathcal{O}(n))=-n+ri$ and heart $Q(0,1]= \coh(\PP^1)$. Note that $d(P,Q)=\arctan(r)$ so that the slicing converges to $P$ as $r\to 0$. 
\end{example}

\begin{example}
\label{ex:non-supported}
In contrast there are lax pre-stability conditions which cannot be deformed to stability conditions. Again on $\CC = \Db(\PP^1)$, let $\sigma=(P_t,W)$ be the lax pre-stability condition defined by the charge $W(\mathcal{O}_x)=0$, $W(\mathcal{O})=i$ and slicing $P_t$ from \cref{ex:slicings}, \ie $P_t(\tfrac{1}{2}) = \Clext{ \cO(n) : n\in\Z }$ and $P_t(1) = \Clext{ \cO_x : x\in\PP^1 }$. Recall from \cref{ex:not well-supported} that $\sigma$ is not a lax stability condition.

The massless subcategory $\NN=\thick{}{\mathcal{O}_x : x\in \PP^1}$ and, for any $\delta\geq 0$, the massive indecomposable $\delta$-slim objects are, up to shifts, the line bundles $\mathcal{O}(n)$ for $n\in \Z$. Let $\sigma_\NN = ((P_t)_\NN,W_\NN)$ where $W_\NN(\mathcal{O}_x) = w$ for some $0\neq w\in\U$ and $(P_t)_\NN$ is a compatible slicing. Considering $W_\NN$ as a charge in $\Hom{\Lambda}{\C}$ via the orthogonal splitting we also have $W_\NN(\mathcal{O})=0$. Therefore
\[
\norm{W_\NN}_{\sigma,\delta} = \sup\big\{ \tfrac{|W_\NN(c)|}{|W(c)|} : \text{ massive indecomposable $\delta$-slim $c$} \big\} = \sup\{ |nw| : n\in \Z\} = \infty.
\]
Thus the assumptions of \cref{thm:deformation to non-lax} are not satisfied. 
\end{example}

\Cref{thm:deformation to non-lax} states that a lax stability condition $\sigma$ deforms uniquely with respect to a suitably small strict stability condition $\tau_\NN=(Q_\NN,W_\NN)$ on the massless subcategory $\NN \coloneqq \NN_\sigma$ to a strict stability condition. We now allow $\tau_\NN$ to be a suitably close lax stability condition and the result will be a lax stability condition with the same massless subcategory $\MM\subseteq \NN$ as $\tau_\NN$.

The deformations described in this result are in general neither purely normal nor fibrewise, but a mixture of the two: the charge $Z$ is changed by $W_\NN \in \Hom{\Lambda_\NN}{\C}$, so for $\MM\neq\NN$, a component of the deformation occurs in the normal direction, but for $\MM\neq 0$ the slicing on $\NN$ is not fully determined by $W_\NN$, and the choice of $Q_\NN$ determines the fibrewise deformation. When $\MM=\NN$ the deformation is purely fibrewise since the charge, indeed the associated stability condition on the quotient, is fixed and only the massless slicing is deformed.

\begin{theorem}
\label{prop:deformation off stratum}
Let $0 < \epsilon < \delta < \tfrac{1}{8}$ and $\sigma = (P,Z) \in \LaxStabN{\CC}{\NN}$ with $\delta$-lax support and $\tau_\NN = (Q_\NN,W_\NN) \in \LaxStabN{\NN}{\MM}$. Suppose $P_\NN$ is well adapted to $\MM$, that $d(P_\NN,Q_\NN) < \epsilon$ and $\norm{W_\NN}_{\sigma,\delta} < \sin(\pi\epsilon)$. Then there is a unique $\tau = (Q,W) \in B_\epsilon^\delta(\sigma) \cap \LaxStabN{\CC}{\MM}$ with $W = Z+W_\NN$ and restriction $\restrict{\NN}(\tau) = \tau_\NN$.

Moreover, the lax stability condition $\tau$ depends continuously on $\sigma$ and $\tau_\NN$.
\end{theorem}

\begin{proof}
For $\MM=0$ this is \cref{thm:deformation to non-lax}. When $\MM\neq 0$ the strategy is to reduce to this case by taking the quotient by $\MM$ and then lifting back up to $\CC$ using \cref{prop:glueing slicings}. 

Recall we assume $\CC$ is Hom-finite. Therefore $P$ is well adapted to $\NN$ by \cref{prop:massive stability condition}. Since $P_\NN$ is well adapted to $\MM$ it follows that $P$ is adapted to $\MM$ too. The induced slicing $P_{\CC/\MM}$ is compatible with the pair $(P_{\NN/\MM}, P_{\CC/\NN})$. Since these are both locally finite, the slicing $P$ is well adapted to $\MM$ by \cref{cor:compatible slicings}. 
Therefore, as explained in \cref{rmk:lax quotient}, $\sigma$ induces a lax stability condition $\massive{\MM}(\sigma) = (P_{\CC/\MM}, Z)$ on $\CC/\MM$. Moreover, $\massive{\MM}(\sigma)$ and $\massive{\MM}(\tau_\NN)$ satisfy the assumptions of \cref{thm:deformation to non-lax}, applied to the thick subcategory $\NN/\MM$ of $\CC/\MM$. This is because $d(P_{\NN/\MM},Q_{\NN/\MM}) \leq d(P_\NN,Q_\NN) < \epsilon$ and 
\[
\norm{W_\NN}_{\massive{\MM}(\sigma),\delta} \leq \norm{W_\NN}_{\sigma,\delta} < \sin(\pi\epsilon)
\]
by \cref{prop:massive stability condition}. Here we consider $W_\NN \in \Hom{\Lambda_\NN/\Lambda_\MM}{\C} $ as an element of $\Hom{\Lambda/\Lambda_\MM}{\C}$ in the natural way: the given splitting $\Hom{\Lambda_\NN}{\C} \inj \Hom{\Lambda}{\C}$ restricts to a splitting $\Hom{\Lambda_\NN/\Lambda_\MM}{\C} \inj \Hom{\Lambda/\Lambda_\MM}{\C}$. Therefore, by applying (the proof of) \cref{thm:deformation to non-lax} to $\massive{\MM}(\sigma)$ and $\massive{\MM}(\tau_\NN)$, there is a unique lax stability condition $(Q_{\CC/\MM}, W) \in \Stab{\CC/\MM} \cap B_\epsilon^\delta( \massive{\MM}(\sigma) )$, with charge $W = Z+W_\NN$ and restriction $\massive{\MM}(\tau_\NN)$ on $\NN/\MM$. (Note that we cannot assume $\CC/\MM$ is Hom-finite so this lax stability condition is not necessarily a stability condition even though it has no massless objects because we cannot apply \cref{lem:lax and strict support}.)

We now lift $(Q_{\CC/\MM}, W)\in \Stab{\CC/\MM}$ to a lax stability condition $(Q,W) \in \LaxStabN{\CC}{\MM}$. Since $d(P_\MM,Q_\MM) \leq d(P_\NN,Q_\NN) < \epsilon$, \cref{prop:glueing slicings} enables us to glue $Q_\MM$ and $Q_{\CC/\MM}$ to a locally finite slicing $Q$, with $d(P,Q)<\epsilon$ by \cref{lem:slicing distance}. By \cref{lem:local persistence of massless factors} this slicing $Q$ restricts to $\NN$. It follows from the construction that the restriction is the slicing glued from $Q_\MM$ and $Q_{\NN/\MM}$, which by uniqueness is $Q_\NN$. The slicing $Q$ is compatible with the charge $W$, so $\tau = (Q,W)$ is a lax pre-stability condition. By construction it is in $B_\epsilon^\delta(\sigma)$, has charge $W=Z+W_\NN$, restriction $\restrict{\NN}(\tau)=\tau_\NN$ and massless subcategory $\MM$. \Cref{cor:lax uniqueness} implies that $\tau$ is unique with these properties. Finally, $\tau$ satisfies lax support by \cref{cor:propagation} and so is a lax stability condition.

The charge $W = Z+W_\NN$ clearly depends continuously on $\sigma$ and $\tau_\NN$; continuity for the slicing follows from the fact that the slicing is glued from $Q_{\CC/\MM}$ and $Q_\MM$ which depend continuously on $\sigma$ and $\tau_\NN$ by \cref{thm:deformation to non-lax} and \cref{lem:contractions} respectively.
\end{proof}

\begin{remark}
We apply this result extensively in two special cases: first when $P_\NN=Q_\NN$ and second when $\MM=\NN$. In both these situations the technical condition that $P_\NN$ is well adapted to $\MM$ is automatic, in the first by applying \cref{prop:massive stability condition} to $\tau_\NN$ to show the lax stability condition $Q_\NN$ is well adapted to its massless subcategory $\MM$, and in the second trivially because $\NN/\MM=0$.
\end{remark}

\subsection{Group actions}
\label{subsec:group actions}

Let $\Aaut{\Lambda}{\CC}$ be the subgroup of auto-equivalences $\alpha \colon \CC \to \CC$ that are compatible with the structure morphism $v\colon K(\CC) \to \Lambda$, \ie the induced automorphism $K(\alpha) \colon K(\CC) \to K(\CC)$ descends (necessarily uniquely) to an isomorphism $[\alpha] \colon \Lambda \to \Lambda$ with $v\circ K(\alpha) = [\alpha]\circ v$. Then $\Aaut{\Lambda}{\CC}$ acts continuously on the left of $\Slice{\CC} \times \Hom{\Lambda}{\C}$ via
\[
(P,Z) \mapsto \left(\alpha \circ P, Z \circ [\alpha]^{-1} \right).
\]
There is also a continuous right action by the universal cover $G$ of the orientation-preserving component $\GL$. An element $g\in G$ corresponds to a pair $(T_g,\theta_g)$ where $T_g \in \GL$ is the projection of $g$ under the covering map and $\theta_g\colon\R \to \R$ is an increasing map with $\theta_g(t+1)=\theta_g(t)+1$ which induces the same map as $T_g$ on the circle $\R/2\Z = (\R^2-\{0\}) / \R_{>0}$. The element acts by
\begin{equation}
\label{eqn:G action}
(P,Z) \mapsto \left( P\circ \theta_g, T_g^{-1} \circ Z\right)
\end{equation}
where we think of the central charge as taking values in $\R^2$. This action preserves the semistable and stable objects and the HN filtrations of all objects. The subgroup consisting of pairs with $T$ conformal is isomorphic to $\C$ with $w\in \C$ acting via
\[
(P,Z) \mapsto \big(P( \phi + \mathrm{Re}\, w), \exp(-i\pi w)Z )\big)
\]
\ie by rotating the phases and rescaling the masses of semistable objects. The $\C$-action is free provided $\CC\neq 0$. A special case is mass dilation: $(P,Z) \cdot i\log(t)/\pi = (P,tZ)$ for $t\in\R_{\geq0}$.

The second projection $\Slice{\CC} \times \Hom{\Lambda}{\C}\to \Hom{\Lambda}{\C}$ is equivariant with respect to these actions and the evident actions on $\Hom{\Lambda}{\C}$. 
 
The actions by auto-equivalences and $\C$ preserve the semi-norms $\norm{\cdot}_{\sigma,\delta}$ for $\sigma \in\LaxStab{\CC}$ in the sense that for any $\alpha\in\Aaut{\Lambda}{\CC}$, $w\in \C$ and $U\in \Hom{\Lambda}{\C}$, we have
\begin{equation}
\label{eqn:semi-norm invariance}
\norm{ \exp(-i\pi w) \ U \circ [\alpha]^{-1}}_{\alpha\cdot \sigma\cdot w, \delta} = \norm{U}_{\sigma,\delta} .
\end{equation}
They also preserve the semi-norm neighbourhoods: $\alpha \cdot B_\epsilon^\delta(\sigma) \cdot w = B_\epsilon^\delta(\alpha \cdot \sigma\cdot w)$. By contrast, the $G$-action does not preserve the semi-norms because it distorts phase ranges.

The above actions restrict to smooth actions on the manifold $\Stab{\CC}$; the next result is the analogue for $\LaxStab{\CC}$ which is (a priori) just a topological space.

\begin{lemma}
\label{lem:actions on dstab}
The actions of $\Aaut{\Lambda}{\CC}$ and of $G$ on $\Stab{\CC}$ extend uniquely to continuous actions on $\LaxStab{\CC}$ so that the charge map is equivariant. Elements of $G$ preserve $\LaxStabN{\CC}{\NN}$ and each $\alpha \in \Aaut{\Lambda}{\CC}$ maps $\LaxStabN{\CC}{\NN}$ to $\LaxStabN{\CC}{\alpha(\NN)}$. 

The map $\massive{\NN} \colon \LaxStabN{\CC}{\NN} \to \Stab{\CC/\NN}$ is $G$-equivariant and for $\alpha \in \Aaut{\Lambda}{\CC}$, the square
\[
\begin{tikzcd}
\LaxStabN{\CC}{\NN} \ar{r}{\alpha}\ar{d}{\massive{\NN}} & \LaxStabN{\CC}{\alpha(\NN)} \ar{d}{\massive{\alpha(\NN)}}\\
\Stab{\CC/\NN} \ar{r}{\alpha} & \Stab{\CC/\alpha(\NN)}
\end{tikzcd}
\]
commutes. The map $\massive{\NN}$ is equivariant for the subgroup of $\Aaut{\Lambda}{\CC}$ preserving $\NN$.
\end{lemma}

\begin{proof}
The actions of $\Aaut{\Lambda}{\CC}$ and $G$ preserve the subspace of lax stability conditions. The equivariance of the charge map and the properties of $\massive{\NN}$ are easy to verify.
\end{proof}

\begin{remark}
\label{rem:projective stability space}
Assuming $\CC\neq0$, the right action of $\C$ on $\Stab{\CC}$ is smooth, free and proper.
The \defn{projective stability space} of $\CC$ is the complex manifold $\PStab{\CC} \coloneqq \Stab{\CC} / \C$.
The charge map $\cm \colon \Stab{\CC} \to \Hom{\Lambda}{\C}$ descends to a holomorphic map $\PStab{\CC} \to \PP( \Hom{\Lambda}{\C} )$ which by abuse of notation we also denote $\cm$.

The $\C$-action is simply transitive on the fibres of $\Stab{\CC} \to \PStab{\CC}$, so the quotient map is a smooth principal $\C$-bundle.
Since $\C$ is contractible its classifying space is homotopy equivalent to a point and the universal bundle is trivial. Therefore all smooth principal $\C$-bundles are trivial, in particular there is a diffeomorphism
\[ \Stab{\CC} \cong \PStab{\CC} \times \C . \]
\end{remark}

\section{The lax closure of the stability space}
\label{sec:topology of dstab}

\noindent
As before, $\CC$ is a Hom-finite $\kk$-linear triangulated category with a fixed surjective homomorphism $v\colon K(\CC) \to \Lambda$ to a free lattice of finite rank. We enlarge the space $\Stab{\CC}$ of stability conditions by adding lax stability conditions in its boundary. \Cref{thm:StabL structure summary} describes the geometry of this lax closure  and this theorem collects all results of the section.

\begin{definition}
\label{def:StabL}
The \defn{lax closure} of $\Stab{\CC}$ is the subspace
\[
\StabL{\CC} \coloneqq \LaxStab{\CC} \cap \overline{\Stab{\CC}} 
\]
of $\Slice{\CC}\times \Hom{\Lambda}{\C}$. The \defn{stratum} of a thick subcategory $\NN$ of $\CC$ is the subspace 
\[
\StabLN{\CC}{\NN} \coloneqq \LaxStabN{\CC}{\NN} \cap \overline{\Stab{\CC}} \subset \StabL{\CC} .
\]
\end{definition}

\subsection{Lax stability conditions in the boundary}
\label{subsec:lax sc in boundary}

We discuss when a lax stability condition is in the lax closure $\StabL{\CC}$. This is not always the case: the group of strictly increasing homeomorphisms $f\colon \R \to \R$ with $f(\phi+1)=f(\phi)+1$ acts on $\Slice{\CC}$ by $P \mapsto P \circ f$. This action provides points in $\LaxStab{\CC}$ which are not in $\StabL{\CC}$. For a specific example, let $P_g$ be the geometric slicing of $\CC = \Db(\PP^1)$ from \cref{ex:slicings}. Applying a non-linear $f$ to $P_g$ yields a slicing $(P',0) \in \Slice{\CC} = \LaxStabN{\CC}{\CC}$. However, $P'$ is incompatible with non-zero charges and so $(P',0) \notin \StabL{\CC}$.
The next result gives an inductive criterion for when a lax stability condition is in the lax closure in terms of the restriction to its massless subcategory. 

\begin{proposition}
\label{prop:dstab criterion}
The following are equivalent for $\sigma=(P,Z) \in \LaxStabN{\CC}{\NN}$:
\begin{enumerate}[label=(\roman*)]
  \item $\sigma\in \StabLN{\CC}{\NN}$;
  \item $P_\NN \in \overline{ \{ Q \colon (Q,W)\in \Stab{\NN} \}}$;
  \item $\restrict{\NN}(\sigma) \in \StabLN{\NN}{\NN}$. 
\end{enumerate}
\end{proposition}

\begin{proof}
(i) $\implies$ (ii). Suppose $\sigma \in \StabLN{\CC}{\NN}$. Then there is a sequence $(\sigma_n)$ of stability conditions in the open neighbourhood $B_\epsilon^\delta(\sigma)$ with $\sigma = \lim_{n\to \infty} \sigma_n$. The continuity of $\restrict{\NN}$, \cref{lem:contractions}, then implies $P_\NN \in \overline{ \{ Q \colon (Q,W)\in \Stab{\NN} \}}$, where the closure is taken in $\Slice{\NN}$.

(ii) $\implies$ (iii). Now suppose $P_\NN =\lim_{n\to \infty} Q_n$ where $(Q_n,W_n)\in \Stab{\NN}$ for $n\in \N$. Then $P_\NN$ is locally finite because it is the limit of locally finite slicings for which the extension closures of the semistable objects with phases in any interval of length strictly less than one are length categories. Moreover, $(Q_n, W_n / n\norm{W_n}) \in \Stab{\NN}$ and converges to $\restrict{\NN}(\sigma) = (P_\NN,0)$ in $\Slice{\NN}\times \Hom{\Lambda_\NN}{\C}$ as $n\to \infty$. So $(P_\NN,0) \in \overline{\Stab{\NN}}$ and, since lax support is automatic when all objects are massless, $(P_\NN,0) = \restrict{\NN}(\sigma) \in \StabLN{\NN}{\NN}$.

(iii) $\implies$ (i). Finally, suppose $\restrict{\NN}(\sigma)=(P_\NN,0) \in \StabLN{\NN}{\NN}$. We can choose a sequence of stability conditions $(Q_n,W_n) \in \Stab{\NN}$ converging to $(P_\NN,0)$, \ie $d(P_\NN,Q_n) \to 0$ and $W_n\to 0$ in the operator norm on $\Hom{\Lambda_\NN}{\C}$. This implies $W_n\to 0$ in the operator norm on $\Hom{\Lambda}{\C}$, using the standard splitting of \cref{rmk:charge space splitting}. As $\sigma$ satisfies $\delta$-lax support for some $\delta>0$, we get $\norm{W_n}_{\sigma,\delta}\to 0$ by \cref{prop:lax support and full}. By \cref{thm:deformation to non-lax} this sequence lifts uniquely to a sequence of stability conditions $(P_n, Z+W_n) \in \Stab{\CC}$ converging to $(P,Z)=\sigma$. Hence $\sigma\in \StabL{\CC}$.
\end{proof}

\begin{example}
Suppose $\CC = \Db(\kk A_2)$, where $A_2$ is the quiver $1 \too 2$. Define a sequence of stability conditions $\tau_n = (Q_n, W_n)$ via $W_n(S_1) = -\tfrac{1}{n}$ and $W_n(S_2) = i/2^n$, where $S_1$ and $S_2$ are the simple modules at $1$ and $2$ in the standard heart, respectively. Note that the indecomposable projective module $P_1$ in $\mod{\kk A_2}$ is also $\tau_n$-semistable.
The limit slicing $Q = \lim_{n\to \infty} Q_n$ is given by $Q(\tfrac{1}{2}) = \langle S_2\rangle$ and $Q(1) = \langle S_1 , P_1\rangle$. This is not a slicing for any (pre-)stability condition on $\Db(\kk A_2)$ because the slice $Q(1)$ is not abelian. Nevertheless, $(Q,0) \in \StabL{\CC}$ because lax support is automatic when all objects are massless. Thus, it is essential to take the closure of the slicings in \cref{prop:dstab criterion}.
\end{example}

\subsection{Local structure of the lax closure}
\label{subsec:local structure thm}

In this section we prove the structure theorem, \cref{thm:StabL structure}, which describes a neighbourhood of the stratum $\StabLN{\CC}{\NN}$ in $\StabL{\CC}$. Strata are locally closed because the map $\StabL{\CC} \to \Thick{\CC}$ assigning the massless subcategory to a lax stability condition is continuous by \cref{cor:thick subcats in nbhd}. 

\begin{lemma}
\label{lem:rho and mu restrict}
The continuous maps $\restrict{\NN}$ and $\massive{\NN}$ defined in \cref{subsec:massive and massless parts} restrict to maps
\[
\restrict{\NN} \colon \StabLN{\CC}{\NN} \to \StabL{\NN} 
  \text{ and }
\massive{\NN} \colon \overline{\StabLN{\CC}{\NN}} \to \StabL{\CC/\NN}
\]
and the latter restricts further to $\massive{\NN} \colon \StabLN{\CC}{\NN} \to \Stab{\CC/\NN}$.
Moreover, for $\sigma\in\StabLN{\CC}{\NN}$, $\epsilon \in (0, \tfrac{1}{6})$ and $\delta>0$, the restriction map extends to
$\restrict{\NN} \colon B_\epsilon^\delta(\sigma) \cap \StabL{\CC} \to \StabL{\NN}$.
\end{lemma}

\begin{proof}
We have to show that both maps preserve the property of being in the closure of the set of strict stability conditions.
For $\restrict{\NN}$ this follows from \cref{prop:dstab criterion}. By \cref{prop:massive part}
  $\massive{\NN} \colon \overline{\LaxStabN{\CC}{\NN}} \to \LaxStab{\CC/\NN}$
is continuous. Therefore the image of $\overline{\StabLN{\CC}{\NN}}$ is in the subspace $\StabL{\CC/\NN}$.
The last claim follows from \cref{lem:restriction map} and continuity.
\end{proof}

We recast the deformation result \cref{prop:deformation off stratum} in geometric form to show that the stratum $\StabLN{\CC}{\NN}$ has a `decoupling neighbourhood' in which the massive and massless parts of the stability condition deform independently. We can explicitly define open subsets on which this decoupling happens. 

$\StabLN{\CC}{\NN} \times_{\Slice{\NN}} \StabL{\NN}$ is the usual fibre product of topological spaces; it is the set of pairs $(\sigma, \tau_\NN) \in \StabLN{\CC}{\NN} \times \StabL{\NN} $ whose slicings agree on $\NN$. It contains the subset of pairs $\big( (P,Z), (P_\NN,0) \big)$ with vanishing charge in the second factor that can be written as
  $\StabLN{\CC}{\NN} \times_{\Slice{\NN}} \StabLN{\NN}{\NN} \cong \StabLN{\CC}{\NN} \times_{\Slice{\NN}} \Slice{\NN} \cong \StabLN{\CC}{\NN}$.

\begin{definition}
\label{def:open nbhd U} 
We define a subset $U(\CC,\NN) \subseteq \StabLN{\CC}{\NN} \times_{\Slice{\NN}} \StabL{\NN}$ by
\[
  U(\CC,\NN) \coloneqq \{ (\sigma,\tau_\NN) \mid \exists\, 0<\epsilon<\delta<\tfrac{1}{16} : \text{$\sigma$ has $\delta$-lax support and } \norm{\cm(\tau_\NN)}_{\sigma,\delta} <\sin(\pi\epsilon) \}.
\]
It follows from the definition of the norm $\|\cdot\|_{\sigma,\delta}$ that $U(\CC,\NN)$ is preserved by the action of the universal cover $G$ of $GL_2^+(\R)$ on $\StabLN{\CC}{\NN} \times_{\Slice{\NN}} \StabL{\NN}$. We also define the \defn{decoupling neighbourhood} $V(\CC,\NN)$ to be the set of all $\tau=(Q,W) \in \StabL{\CC}$ for which there exist $0< \epsilon<\delta<\tfrac{1}{16}$ and $\sigma=(P,Z) \in \StabLN{\CC}{\NN}$ with $\delta$-lax support such that
\begin{enumerate}
\item $\tau \in B_\epsilon^\delta(\sigma)$, so that the restriction $\restrict{\NN}(\tau)=(Q_\NN,W_\NN)$ exists by \cref{lem:restriction map},
\item the charge $W=Z+W_\NN$,
\item the restricted slicing $Q_\NN=P_\NN$.
\end{enumerate}
\end{definition}

We show in \cref{prop:deformation map} that $U(\CC,\NN)$ and $V(\CC,\NN)$ are homeomorphic. In particular this homeomorphism is $G$-equivariant so that $V(\CC,\NN)$ is preserved by the $G$ action on $\StabL{\CC}$. First we show both are open and that \cref{prop:deformation off stratum} applies to all pairs $(\sigma,\tau_\NN)\in U(\CC,\NN)$.

\begin{lemma}
\label{lem:U is open}
\label{prop:V is nbhd}
\begin{enumerate}
  \item 
  The subset $U(\CC,\NN)$ is open in $\StabLN{\CC}{\NN} \times_{\Slice{\NN}} \StabL{\NN}$ and contains
  $\StabLN{\CC}{\NN} \times_{\Slice{\NN}} \StabLN{\NN}{\NN}$.
  \item 
  If $(\sigma, \tau_\NN) \in U(\CC,\NN)$ with $\sigma = (P,Z)$ then the restricted slicing $P_\NN$ is well adapted to the massless subcategory of $\tau_\NN$.
  \item The subset $V(\CC,\NN) \subset \StabL{\CC}$ is an open neighbourhood of $\StabLN{\CC}{\NN}$.
\end{enumerate}
\end{lemma}

\begin{proof}
\begin{enumerate}[wide, labelwidth=!, labelindent=1em]
\item
We have $U(\CC,\NN) \supset \StabLN{\CC}{\NN} \times_{\Slice{\NN}} \StabLN{\NN}{\NN}$ because $\tau_\NN = (P_\NN,0)$ has no massive objects whatsoever, so that the norm inequality holds trivially.

Since the fibre product has the subspace topology from the product $\StabLN{\CC}{\NN} \times \StabL{\NN}$, it suffices to show that the set of pairs $(\sigma,\tau_\NN)$ in this product for which there exist $0 < \epsilon < \delta < \tfrac{1}{16}$ such that $\sigma$ has $\delta$-lax support and $\norm{W_\NN}_{\sigma,\delta} < \sin(\pi\epsilon)$ is open, where $W_\NN=\cm(\tau_\NN)$. Let $(\sigma, \tau_\NN)$ be such a pair. Choose $\epsilon<\epsilon'<\delta'<\delta$ such that 
\begin{equation}
\label{eqn:delta' choice}
C(\delta',\epsilon')\coloneqq (1-\sin(\pi(\delta-\delta'))\sin(\pi\epsilon')-\sin(\pi\epsilon) >0.
\end{equation}
Then any $(\sigma',\tau_\NN')$ in the product with $\sigma'\in B_{\delta-\delta',\delta}(\sigma)$ and 
$\norm{W_\NN'-W_\NN}_{\sigma,\delta} < C(\delta',\epsilon')$
satisfies the required condition because $\sigma'$ has $\delta'$-lax support by \cref{cor:propagation} and
\[
\norm{W_\NN'}_{\sigma',\delta'} 
\leq \frac{\norm{W_\NN'}_{\sigma,\delta}}{1-\sin(\pi(\delta-\delta'))}
\leq \frac{\sin(\pi\epsilon)+\norm{W_\NN'-W_\NN}_{\sigma,\delta}}{1-\sin(\pi(\delta-\delta'))}
<\sin(\pi\epsilon')
\]
by \cref{cor:semi-norm bound}, the triangle inequality and \cref{eqn:delta' choice}. This is clearly an open condition on $(\sigma',\tau_\NN')$ so $U(\CC,\NN)$ is open as claimed.
\item
 The slicing $Q_\NN$ of $\tau_\NN$ is well adapted to its massless subcategory by \cref{prop:massive stability condition}, hence the same holds for $P_\NN = Q_\NN$.
\item
Let $\tau=(Q,W)\in V(\CC,\NN)$. By definition of $V(\CC,\NN)$, we can choose $0<\epsilon<\delta$ and $\sigma=(P,Z)\in \StabLN{\CC}{\NN}$ with $\delta$-lax support, $\tau\in B_\epsilon^\delta(\sigma)$, charge $Z=W-W_\NN$ and $P_\NN = Q_\NN$ where $\restrict{\NN}(\tau)=(Q_\NN,W_\NN)$. Choose $\eta>0$ sufficiently small that there exists $\epsilon +\eta < \epsilon' < \delta' \coloneqq \delta-\eta$ with $\sin(\pi\epsilon') > \sin(\pi\epsilon) / (1-\sin(\pi\eta))$. 

We claim there is $\eta' \in (0, \epsilon'-\epsilon-\eta)$ such that for any $\tau' = (Q',W') \in \StabL{\CC}$ with
\begin{equation}
\label{eqn:tau' condition 1}
\norm{ (W'-W_\NN') - (W-W_\NN)}_{\sigma,\delta} < \sin(\pi\eta') \quad \text{and} \quad d(Q',Q)<\eta'
\end{equation}
we can find $\sigma' = (P',Z') \in B_\eta^\delta(\sigma) \cap \StabLN{\CC}{\NN}$ with charge $Z'=W'-W_\NN'$ and restricted slicing $P_\NN' = Q_\NN'$. 
To construct $\sigma'$, first, inequality~\eqref{eqn:tau' condition 1} allows us to apply \cref{prop:deformations in boundary} to deform the charge of $\sigma$ to $Z' = W'-W_\NN'$, keeping the massless slicing fixed, to obtain a lax stability condition $\sigma'' = (P,Z')$ with $(\delta-\eta')$-lax support by \cref{cor:propagation}.
Second, applying \cref{prop:deformation off stratum}, we use $(Q'_\NN,0)$ to deform $\sigma''$ to a unique $\sigma' = (P', Z') \in B_{\eta'}^{\delta-\eta'}(\sigma'') \cap \LaxStabN{\CC}{\NN}$. 
As $B_\eta^\delta(\sigma)$ is open we can choose $\eta'>0$ sufficiently small that the resulting lax stability condition $\sigma'$ lies in $B_\eta^\delta(\sigma) \cap \LaxStabN{\CC}{\NN}$. Finally, since $\restrict{\NN}(\sigma') = (Q_\NN',0) \in \StabLN{\NN}{\NN}$ by \cref{prop:dstab criterion} we conclude $\sigma' = (P',Z') \in B_\eta^\delta(\sigma) \cap \StabLN{\CC}{\NN}$ by \cref{prop:dstab criterion} again.

If, in addition to inequality~\eqref{eqn:tau' condition 1}, $\tau'$ is such that
\begin{equation}
\label{eqn:tau' condition 2}
\norm{W_\NN'-W_\NN}_{\sigma,\delta} < (1-\sin(\pi\eta))\sin(\pi\epsilon') - \sin(\pi\epsilon)
\end{equation}
then $\tau' \in V(\CC,\NN)$. This is because $\sigma'$ has $\delta'$-lax support by \cref{cor:propagation}, the slicing distance $d(P',Q') \leq d(P',P) + d(P,Q)+d(Q,Q') \leq \eta+\epsilon+\eta' < \epsilon'$, and we have the norm estimate
\[
\norm{W_\NN'}_{\sigma',\delta'} 
\leq \frac{\norm{W_\NN'}_{\sigma,\delta} }{1-\sin(\pi\eta)}
\leq \frac{\norm{W_\NN}_{\sigma,\delta} + \norm{W_\NN'-W_\NN}_{\sigma,\delta} }{1-\sin(\pi\eta)}
   < \frac{\sin(\pi\epsilon) + \norm{W_\NN'-W_\NN}_{\sigma,\delta} }{1-\sin(\pi\eta)}
\leq \sin(\pi\epsilon')
\]
by \cref{cor:semi-norm bound}. The requirements \eqref{eqn:tau' condition 1} and \eqref{eqn:tau' condition 2} are open, hence so is the subset $V(\CC,\NN)$.
\qedhere
\end{enumerate}
\end{proof}

\begin{lemma}
\label{lem:stratum retraction}
There is a continuous retraction
\[ \Phi_\NN \colon V(\CC,\NN) \to \StabLN{\CC}{\NN} \]
such that $ \cm(\tau) = \cm(\Phi_\NN(\tau)) +\cm(\restrict{\NN}(\tau))$ and the slicings of $\restrict{\NN}(\tau)$ and $\restrict{\NN}(\Phi_\NN(\tau))$ coincide. The germ of $\Phi_\NN$ along $\StabLN{\CC}{\NN} \subset V(\CC,\NN)$ is uniquely determined by these properties.
\end{lemma}

\begin{proof}
For each $\tau\in V(\CC,\NN)$ there is a unique choice of $\sigma\in \StabLN{\CC}{\NN}$ satisfying the properties in the definition of $V(\CC,\NN)$. This is because the massless subcategory of $\sigma$ is fixed, the charge is determined by the charge additivity and the massless slicing by the coincidence of restricted slicings, so that the condition $\tau \in B_\epsilon^\delta(\sigma)$ where $\epsilon < \tfrac{1}{16}$ uniquely determines $\sigma$ by 
\Cref{cor:lax uniqueness}. Setting $\Phi_\NN(\tau)=\sigma$ defines a retraction because the uniqueness implies $\Phi_\NN(\sigma)=\sigma$ for all $\sigma\in \StabLN{\CC}{\NN}$. By construction, $\Phi_\NN$ has the required charge additivity and slicing properties. For continuity, recall that $\tau=(Q,W) \in B_\epsilon^\delta(\Phi_\NN(\tau))$ and suppose $\tau' =(Q',W') \in B_\epsilon^\delta(\Phi_\NN(\tau))$ too. Then, since $\tau' \in B_\epsilon^\delta(\Phi_\NN(\tau'))$, the respective slicings of $\Phi_\NN(\tau)$ and $\Phi_\NN(\tau')$ satisfy $d(P,P') \leq d(Q,Q')+2\epsilon$ where $P$ and $P'$. Since the charge component of $\Phi_\NN$ is given by linear projection, the map $\Phi_\NN$ is continuous. 

If $\Phi$ is any continuous retraction with the properties of the lemma then the charge and massless slicing of $\Phi(\tau)$ are uniquely determined. As $\Phi$ is a continuous retraction, \cref{cor:lax uniqueness} implies that $\Phi$ and $\Phi_\NN$ coincide on an open neighbourhood of $\StabLN{\CC}{\NN}$, \ie the germ of $\Phi_\NN$ along that subset is uniquely determined.
\end{proof}

The next two results show first that the decoupling neighbourhood $V(\CC,\NN)$ is homeomorphic to $U(\CC,\NN)$, and second that it is an open neighbourhood of the stratum $\StabLN{\CC}{\NN}$ in $\StabL{\CC}$.

\begin{proposition}
\label{prop:deformation map}
There is a homeomorphism
\[ \deform{\NN} \colon U(\CC,\NN) \to V(\CC,\NN) \]
with inverse $\Phi_\NN\times \restrict{\NN} \colon V(\CC,\NN) \to U(\CC,\NN)$ and the properties
$\cm(\deform{\NN}(\sigma,\tau_\NN)) = \cm(\sigma)+\cm(\tau_\NN)$ and $\restrict{\NN}(\deform{\NN}(\sigma,\tau_\NN)) = \tau_\NN$ for all $(\sigma,\tau_\NN) \in U(\CC,\NN)$.

The germ of $\deform{\NN}$ along $\StabLN{\CC}{\NN}\times_{\Slice{\NN}} \StabLN{\NN}{\NN} \subset U(\CC,\NN)$ is uniquely determined by continuity and these properties. Moreover, $\deform{\NN}$ is equivariant with respect to the $G$ actions on $\StabLN{\CC}{\NN}\times_{\Slice{\NN}} \StabL{\NN} $ and $\StabL{\CC}$.
\end{proposition}

\begin{proof}
\Cref{prop:deformation off stratum} applies to any pair $(\sigma,\tau_\NN) \in U(\CC,\NN)$. We can define $\deform{\NN}(\sigma,\tau_\NN) $ to be the unique lax stability condition $\tau \in B_\epsilon^\delta(\sigma)$ with charge $\cm(\tau) = \cm(\sigma)+\cm(\tau_\NN)$ and restriction $\restrict{\NN}(\tau) = \tau_\NN$ constructed therein. The resulting map $\deform{\NN} \colon U(\CC,\NN) \to \LaxStab{\CC}$ is continuous by the final statement of \cref{prop:deformation off stratum}. Its $G$-equivariance follows from the above properti

We need to show $\deform{\NN}(\sigma,\tau_\NN) \in \StabL{\CC}$, \ie can be approximated by strict stability conditions. 
By construction, $\deform{\NN}(\sigma,\tau_\NN) \in \LaxStabN{\CC}{\MM}$, where $\MM\subseteq \NN$ is the massless subcategory of $\tau_\NN$, and by \cref{prop:dstab criterion}
\[
\restrict{\MM}(\deform{\NN}(\sigma,\tau_\NN)) = \restrict{\MM} \circ \restrict{\NN}(\deform{\NN}(\sigma,\tau_\NN)) = \restrict{\MM}(\tau_\NN) \in \StabLN{\MM}{\MM} .
\]
We deduce $\deform{\NN}(\sigma,\tau_\NN) \in \StabLN{\CC}{\MM} \subset \StabL{\CC}$ from a second application of \cref{prop:dstab criterion}.

By construction, $\restrict{\NN}(\deform{\NN}(\sigma,\tau_\NN)) = \tau_\NN$. Also $\Phi_\NN(\deform{\NN}(\sigma,\tau_\NN)) =\sigma$ because $\deform{\NN}(\sigma,\tau_\NN) \in B_\epsilon^\delta(\sigma)$ has charge $Z+W_\NN$ and restricted slicing $P_\NN=Q_\NN$ so \cref{cor:lax uniqueness} applies, where $\sigma=(P,Z)$ and $\tau_\NN=(Q_\NN,W_\NN)$. Hence $\Phi_\NN\times \restrict{\NN}$ is left inverse to $\deform{\NN}$.
Conversely, the definitions imply $(\Phi_\NN(\tau),\restrict{\NN}(\tau)) \in U(\CC,\NN)$ for any $\tau\in V(\CC,\NN)$. Furthermore, $\deform{\NN}(\Phi_\NN(\tau),\restrict{\NN}(\tau))=\tau$ by \cref{cor:lax uniqueness} because both sides share charge, massless subcategory and massless slicing, and their slicings are within $\epsilon<\tfrac{1}{16}$ of that of $\Phi_\NN(\tau)$. Thus $\deform{\NN}$ is continuous with inverse $\Phi_\NN\times\restrict{\NN}$, which is continuous by \cref{lem:stratum retraction,,lem:contractions}.

If $\deform{}$ is any map with the three listed properties then the charge, massless subcategory and massless slicing of $\deform{}(\sigma,\tau_\NN)$ are uniquely determined. The continuity of $\deform{}$ together with the fact that $\deform{}(\sigma, \restrict{\NN}(\sigma))=\sigma$ then uniquely determine the slicing, so that $\deform{} = \deform{\NN}$ on some neighbourhood of $\StabLN{\CC}{\NN}\times_{\Slice{\NN}} \StabLN{\NN}{\NN}$ by \cref{cor:lax uniqueness}. Hence the germ of $\deform{\NN}$ along this subset is unique.
\end{proof}

The set $U(\CC,\NN)$ from \cref{def:open nbhd U} is a subset of the product $\StabLN{\CC}{\NN} \times \StabL{\NN}$, so that the map $\massive{\NN}\times\id \colon U(\CC,\NN) \to \Stab{\CC/\NN} \times \StabL{\NN}$ below is well defined.

\begin{theorem}
\label{thm:StabL structure}
Let $\NN\subseteq \CC$ be a thick subcategory. Then there is a commutative diagram 
\[
\begin{tikzcd}[column sep=large]
  \StabLN{\CC}{\NN} \ar[equals]{r} \ar[hook]{d}
& \StabLN{\CC}{\NN} \ar{d}{\id \times \restrict{\NN}} \ar{r}{\massive{\NN}\times \restrict{\NN}}
& \Stab{\CC/\NN} \times \StabLN{\NN}{\NN} \ar[hook]{d}
\\
 V(\CC,\NN)
& U(\CC,\NN)        \ar{l}[swap]{\deform{\NN}} \ar{r}{\massive{\NN}\times\id}% was: \StabLN{\CC}{\NN} \times_{\Slice{\NN}} \StabL{\NN} 
& \Stab{\CC/\NN}\times \StabL{\NN}
\end{tikzcd}
\]
in which all maps are injective and $G$-equivariant, and the horizontal maps are open embeddings. The maps in the bottom row are stratum-preserving where $V(\CC,\NN)$ inherits its stratification from $\StabL{\CC}$, and $U(\CC,\NN)$ and $\Stab{\CC/\NN}\times\StabL{\NN}$ are stratified respectively by $\StabLN{\CC}{\NN}\times_{\Slice{\NN}}\StabLN{\NN}{\MM}$ and $\Stab{\CC/\NN}\times\StabLN{\NN}{\MM}$ for thick $\MM \subseteq \NN$. 
\end{theorem}

\begin{proof}
The maps are defined and continuous by \cref{lem:rho and mu restrict} and \cref{prop:deformation map}, and $\restrict{\NN}$, $\massive{\NN}$ and $\deform{\NN}$ are all $G$-equivariant. The commutativity of the left hand square is the previously noted identity $\deform{\NN}(\sigma,\restrict{\NN}(\sigma))=\sigma$; that of the right hand square is evident. The left and right hand vertical maps are inclusions, and the middle vertical map is evidently injective with left inverse the first projection. It remains to show that the horizontal maps are open and injective. 

The product $\massive{\NN}\times\restrict{\NN}$ is injective: given $\sigma = (P,Z) \in \StabLN{\CC}{\NN}$, the slicings $P_\NN$ and $P_{\CC/\NN}$ determine $P$ uniquely by \cref{cor:uniqueness of compatibility}, \cref{prop:quotient slicing,,prop:massless thick} and the charge $Z$ is determined by its factorisation through $\Lambda/\Lambda_\NN$. 
To see that it is open, note that we can deform the charge of $\sigma$ to any nearby charge in $\Hom{\Lambda/\Lambda_\NN}{\C}$ whilst keeping the massless slicing $P_\NN$ constant using \cref{prop:deformations in boundary}, and independently we can deform the massless slicing $P_\NN$ to any nearby massless slicing whilst keeping the charge constant using \cref{prop:deformation off stratum}. Combining these two shows that $\massive{\NN}\times\restrict{\NN}$ is an open map.

The map $\massive{\NN}\times \id$ is injective and open for essentially the same reasons; the slicing of $\sigma$ in $\StabLN{\CC}{\NN}$ can be reconstructed uniquely from the slicing of $\massive{\NN}(\sigma)$ and the slicing on $\NN$ by \cref{prop:uniqueness of compatibility,prop:quotient slicing,,prop:massless thick}, and we can deform the charge of $\sigma$ in $\Hom{\Lambda/\Lambda_\NN}{\C}$ whilst keeping the massless slicing constant.

The map $\deform{\NN}$ is open and injective by \cref{prop:deformation map,,prop:V is nbhd}. It is immediate from the definition that $\massive{\NN}\times \id$ is stratum-preserving; for $\deform{\NN}$ it follows because $\deform{\NN}(\sigma,\tau)$ has the same massless category as $\tau$.
\end{proof}

\begin{definition}
\label{def:deformation retract}
The \defn{stratum deformation retraction} is the map
\[
  \Phi_{\NN,t} \coloneqq \deform{\NN} \circ ( \Phi_\NN \times t\cdot \restrict{\NN} ) \colon V(\CC,\NN) \to V(\CC,\NN)
\]
where $t\in[0,1]$ and $t\cdot (Q_\NN,W_\NN) = (Q_\NN,tW_\NN)$ is mass dilation in $\StabL{\NN}$. For $t>0$ this is given by acting by $i\log(t)/\pi \in\C^*$ on the right as described in \cref{subsec:group actions}.
\end{definition}

\begin{remark}
\label{rmk:retraction identities}
From the definition $\Phi_{\NN,1}=\id$. The commutativity of the left hand square in \cref{thm:StabL structure} applied to $\Phi_\NN(\tau) \in \StabLN{\CC}{\NN}$ implies $\Phi_{\NN,0} = \Phi_\NN$.
Moreover, $\Phi_\NN \circ \Phi_{\NN,t} = \Phi_\NN$, and $\restrict{\NN}\circ \Phi_{\NN,t} = t\cdot \restrict{\NN}$ as $\restrict{\NN} \circ \deform{\NN}$ is the second projection by \cref{prop:deformation map}.
\end{remark}

\begin{corollary}
\label{cor:V homeo}
The map $(\massive{\NN}\circ\Phi_\NN)\times\restrict{\NN}\colon V(\CC,\NN) \to \Stab{\CC/\NN}\times \StabL{\NN}$ is a homeomorphism onto an open neighbourhood of the image of $\massive{\NN}\times \restrict{\NN} \colon \StabLN{\CC}{\NN} \inj \Stab{\CC/\NN} \times \StabLN{\NN}{\NN}$. 
\end{corollary}

\begin{proof}
This is immediate from \cref{thm:StabL structure}.
\end{proof}

\begin{remark}
\label{rmk:mass decoupling}
We refer to the previous result as \defn{mass decoupling}: in a neighbourhood of $\StabLN{\CC}{\NN}$ the massive and massless components, $\massive{\NN}\circ \Phi_\NN(\sigma)$ and $\restrict{\NN}(\sigma)$ respectively, of the stability condition $\sigma$ `decouple' and deform independently of one another. 
\end{remark}

\begin{proposition}
\label{prop:deformation retract 1}
The map $V(\CC,\NN) \times [0,1] \to V(\CC,\NN)$, $(\sigma,t) \mapsto \Phi_{\NN,t}(\sigma)$ is an almost-stratum-preserving deformation retraction of $V(\CC,\NN)$ onto $\StabLN{\CC}{\NN}$, \ie a deformation retraction such that if $\tau \in \StabLN{\CC}{\MM}$ for some thick $\MM \subseteq \NN$ then $\Phi_{\NN,t}(\tau)\in \StabLN{\CC}{\MM}$ for $0<t\leq 1$.
\end{proposition}

\begin{proof}
From \cref{rmk:retraction identities}, $\Phi_{\NN,0} = \Phi_\NN$ and $ \Phi_{\NN,1} = \id$. Continuity follows from continuity of $\deform{\NN}$, $\Phi_\NN$ and $\massive{\NN}$. Finally, if $\sigma = (P,Z) \in \StabLN{\CC}{\NN}$ then $\Phi_\NN(\sigma)=\sigma$ and 
\[
\Phi_{\NN,t}(\sigma)= \deform{\NN}(\Phi_\NN(\sigma), t\cdot\restrict{\NN}(\sigma)) = \deform{\NN}(\sigma, \restrict{\NN}(\sigma)) = \sigma
\]
for all $0<t\leq 1$ too, because $\restrict{\NN}(\sigma) = (P_\NN,0)$. Therefore the map is a deformation retraction. 

The massless subcategory of $\tau\in V(\CC,\NN)$ is determined by the restriction $\restrict{\NN}(\tau)$. So the deformation retraction is almost-stratum-preserving because $\restrict{\NN}( \Phi_{\NN,t}(\tau)) = t\cdot \restrict{\NN}(\tau)$ has the same massless subcategory as $\tau$ for $0<t\leq 1$.
\end{proof}

\begin{remark}
\label{rmk:V examples}
When $\NN=0$ the open neighbourhood $V(\CC,\NN)=\Stab{\CC}$ and $\Phi_{\NN,t} = \id$ for all $t\in [0,1]$. At the other extreme, when $\NN=\CC$, the open neighbourhood $V(\CC,\NN) = \StabL{\CC}$ is the entire lax closure and the deformation retraction $\Phi_{\NN,t}(Q,W) = (Q,tW)$ is given by uniform mass dilation, in particular, $\Phi_\NN(Q,W) = (Q,0)$ is the projection onto $\Slice{\CC}$.
\end{remark}

Although we refer to $\StabLN{\CC}{\NN}$ as a stratum of $\StabL{\CC}$, the decomposition into these subsets is {\em not} a stratification in (any) classical sense. In general the `strata' $\StabLN{\CC}{\NN}$ are not manifolds and do not satisfy the frontier condition, \ie the closure of a stratum need not be a union of strata.

\begin{example}
Let $\CC=\Db(\PP^1)$ and consider the pair $\MM = \thick{}{\mathcal{O}} \subset \Db(\PP^1) = \NN$. The intersection $\StabLN{\CC}{\NN} \cap \overline{\StabLN{\CC}{\MM}} \neq \emptyset$ because it contains the pairs $(P,0)$ where the slicing $P$ occurs in a pair $(P,Z)\in \Stab{\CC}$ for which $\mathcal{O}$ is semistable. However, it is not the whole of $\StabLN{\CC}{\NN}$ because it does not contain any slicings for which $\mathcal{O}$ is unstable, such as those in which the only two semistable objects are $\mathcal{O}(1)$ and $\mathcal{O}(2)$ and their shifts.
\end{example}
An important consequence of the frontier condition is that when it holds the strata are partially ordered by containment of their closures. Surprisingly, even though the frontier condition fails, the strata of the lax closure admit a similar partial order.

\begin{lemma}
\label{lem:lax partial order}
Define a binary relation on the strata of $\StabL{\CC}$ by
\[
\StabLN{\CC}{\NN} \leq \StabLN{\CC}{\NN'} \iff \StabLN{\CC}{\NN} \cap \overline{\StabLN{\CC}{\NN'}} \neq \emptyset.
\]
This is a partial order, and moreover $\StabLN{\CC}{\NN} \leq \StabLN{\CC}{\NN'} \iff \StabLN{\NN}{\NN'} \neq \emptyset$.
\end{lemma}

\begin{proof}
The final statement holds because $\StabLN{\CC}{\NN}$ has an open neighbourhood stratum-preserving homeomorphic to an open neighbourhood in $\Stab{\CC/\NN} \times \StabL{\NN}$ by \cref{thm:StabL structure}.

The binary relation is clearly reflexive, and if $\StabLN{\CC}{\NN} \leq \StabLN{\CC}{\NN'}$ then $\NN' \subseteq \NN$ so that it is also antisymmetric. To show that it is transitive we show that if both $\StabLN{\NN}{\NN'}$ and $\StabLN{\NN'}{\NN''}$ are non-empty then so is $\StabLN{\NN}{\NN''}$. To see this note that the non-empty stratum $\StabLN{\NN}{\NN'}$ in $\StabL{\NN}$ has an open neighbourhood homeomorphic, by a homeomorphism which preserves the strata where particular subcategories are massless, to an open neighbourhood in $\Stab{\NN/\NN'} \times \StabL{\NN'}$. In particular, since $\Stab{\NN/\NN'} \times \StabLN{\NN'}{\NN''} \neq \emptyset$ we conclude $\StabLN{\NN}{\NN''} \neq \emptyset$.
\end{proof}

The following result will be useful in showing that the strata of the quotient stability space introduced in the next section do satisfy the frontier condition.

\begin{lemma}
\label{lem:closed union of strata}
For any thick subcategory $\NN$ of $\CC$ the union $\bigcup_{\StabLN{\NN'}{\NN}\neq \emptyset} \StabLN{\CC}{\NN'}$ of all strata below $\StabLN{\CC}{\NN}$ is closed.
\end{lemma}

\begin{proof}
Suppose $\StabLN{\CC}{\MM}$ is \emph{not} below $\StabLN{\CC}{\NN}$. Its open neighbourhood $V(\CC,\MM)$ is contained in the union of strata above $\StabLN{\CC}{\MM}$, and hence not below $\StabLN{\CC}{\NN}$. Therefore the complement of $\bigcup_{\StabLN{\NN'}{\NN} \neq \emptyset} \StabLN{\CC}{\NN'}$ is open.
\end{proof}

For ease of reference we state the following summary. 

\begin{theorem}
\label{thm:StabL structure summary}
The lax closure $\StabL{\CC}$ is a topological space with a decomposition 
\[
\StabL{\CC} = \bigsqcup_{\NN \in \Thick{\CC}} \StabLN{\CC}{\NN}
\]
into (possibly empty) strata $\StabLN{\CC}{\NN}$ indexed by thick subcategories $\NN \subseteq \CC$. There is a continuous map $\cm \colon \StabL{\CC}\to\Hom{\Lambda}{\CC}$ sending $\StabLN{\CC}{\NN}$ to $\Hom{\Lambda/\Lambda_\NN}{\C}$. 
\begin{enumerate}[leftmargin=*]
\item The stratum $\StabLN{\CC}{0} \cong \Stab{\CC}$ is open and dense. 
\item The strata $\StabLN{\CC}{\NN}$ are locally closed subsets of $\StabL{\CC}$.
\item The set of strata is partially ordered by
\[
\StabLN{\CC}{\NN} \leq \StabLN{\CC}{\MM} 
\iff \StabLN{\CC}{\NN} \cap \overline{\StabLN{\CC}{\MM}} \neq \emptyset
\iff \StabLN{\NN}{\MM} \neq \emptyset
\]
and this poset has height bounded by $\rk(\Lambda)$.
\item 
\label{local model V}
The subset $V(\CC,\NN)$ is an open neighbourhood of the stratum $\StabLN{\CC}{\NN}$ and
\[
V(\CC,\NN)\longrightarrow \Stab{\CC/\NN} \times \StabL{\NN}, \quad \sigma \mapsto (\massive{\NN} \circ \Phi_\NN(\sigma),\restrict{\NN}(\sigma))
\]
is a homeomorphism onto an open neighbourhood of the image of $\StabLN{\CC}{\NN}$.
\item There is an almost-stratum-preserving deformation retraction
\[
  V(\CC,\NN)\times [0,1] \to V(\CC,\NN), \quad (\sigma,t) \mapsto \Phi_{\NN,t}(\sigma)
\]
of $V(\CC,\NN)$ onto the stratum $\StabLN{\CC}{\NN}$ with $\Phi_{\NN,0} = \Phi_{\NN}$.
\item The right action on $\Stab{\CC}$ of the universal cover $G$ of $\GL$ extends continuously to an action on $\StabL{\CC}$ preserving the strata. The maps $\restrict{\NN}$, $\massive{\NN}$ and $\Phi_{\NN,t}$ are $G$-equivariant.

The subgroup $\C \subset G$ acts freely except on the deepest stratum $\StabLN{\CC}{\CC}$; the subgroup $\R \subset \C$ of phase translations still acts freely on the deepest stratum but the subgroup $i\R \subset \C$ of mass dilations stabilises each point.
\item The left action on $\Stab{\CC}$ of $\Aaut{\Lambda}{\CC}$ extends continuously to $\StabL{\CC}$ with $\alpha\in \Aaut{\Lambda}{\CC}$ mapping $\StabLN{\CC}{\NN}$ to $\StabLN{\CC}{\alpha(\NN)}$.
\end{enumerate}
\end{theorem}

\begin{proof}
The proofs can be found as follows:
\begin{enumerate}[wide, labelwidth=!, labelindent=1em]
\item Immediate from the definition $\StabL{\CC} \coloneqq \overline{\Stab{\CC}} \cap \LaxStab{\CC}$ and \cref{cor:thick subcats in nbhd}.
\item Follows from \cref{cor:thick subcats in nbhd}.
\item For existence see \cref{lem:lax partial order}; for the bound on the height see \cref{cor:lax poset bded}.
\item See \cref{thm:StabL structure}.
\item See \cref{prop:deformation retract 1}.
\item Follows from \cref{lem:actions on dstab}.
\item Follows from \cref{lem:actions on dstab}. \qedhere
\end{enumerate}
\end{proof}

\subsection{Codimension one strata}
\label{subsec:codim one lax}

We investigate the case in which the massless subcategory is non-zero but `as small as possible'. More precisely, let $\NN$ be a thick subcategory of $\CC$ for which the saturation $\Lambda_\NN$ of the image of $K(\NN) \to K(\CC) \to \Lambda$ is a rank one lattice. We assume that $\StabLN{\CC}{\NN}$ is non-empty. 

By \cref{thm:StabL structure summary} this stratum has an open neighbourhood $V(\CC,\NN)$ homeomorphic to an open subset of $\Stab{\CC/\NN} \times \StabL{\NN}$ via $\sigma \mapsto ( \Phi_\NN(\sigma) , \restrict{\NN}(\sigma) )$. We begin by describing $\StabL{\NN}$. Observe that it, hence also $\Stab{\NN}$, is non-empty. The charge 
\[
\cm \colon \Stab{\NN} \to \Hom{\Lambda_\NN}{\C}
\]
is the universal cover of the set $\Hom{\Lambda_\NN}{\C}^*$ of non-zero charges. The stability space $\Stab{\NN}$ is a single orbit of the $\C$ action; the semistable objects are the same in each stability condition and lie in a single slice, which is a length heart $\cat{H}$ in $\NN$. The slicing is determined by a single phase in $\R$, and every phase occurs, so $\StabLN{\NN}{\NN} \cong \R$ by \cref{prop:dstab criterion}. 

To describe the topology of $\StabL{\NN}$ choose a generator $\lambda \in \Lambda_\NN$ for the free summand and identify $\Hom{\Lambda_\NN}{\C} \cong \C$ via $Z \mapsto Z(\lambda)$. Then there is a commutative diagram
\[
\begin{tikzcd}
\R \cong \StabLN{\NN}{\NN} \ar{d} \ar[hook]{r} &
\StabL{\NN} \ar{d} \ar[hookleftarrow]{r} & 
 \Stab{\NN} \cong \C \ar{d} \\
S^1 \ar{d}  \ar[hook]{r} & 
\mathrm{Bl}_\R(\C,0) \ar{d} \ar[hookleftarrow]{r}
& 
\C^* \ar[equals]{d} \\
0 \ar[hook]{r} & \C \ar[hookleftarrow]{r} & \C^*
\end{tikzcd}
 \]
in which $\mathrm{Bl}_\R(\C,0)$ denotes the real oriented blowup of $\C$ at the origin; the left hand maps are, in ascending order, the inclusions of the centre of the blowup, the exceptional divisor and its universal cover; the right hand maps are the inclusions of the complements; the upper row of maps are the universal covers, and the lower row the blowdown map and its restriction to exceptional divisor and complement. The composite of the middle two vertical maps is the charge map for the lax closure.

The open neighbourhood $V(\CC,\NN)$ is homeomorphic to the subset of $\left( (P_{\CC/\NN},Z_{\CC/\NN}), (P_\NN,W_\NN) \right)$ in $\Stab{\CC/\NN} \times \StabL{\NN}$ for which there exists a locally finite slicing $P$ compatible with $(P_{\CC/\NN},P_\NN)$ such that $(P,Z_{\CC/\NN})$ satisfies lax support. This is open by \cref{prop:glueing slicings} and \cref{cor:propagation}. Note that $(P,Z_{\CC/\NN})$ is in the closure of $\Stab{\CC}$ by \cref{prop:dstab criterion} and that all objects in the heart $\cat{H}$ of $\NN$ are semistable of the same phase for each lax stability condition in $V(\CC,\NN)$; this common massless phase is the only information in the slicing $P_\NN$. 

The map $\restrict{\NN} \colon V(\CC,\NN) \to \StabL{\NN}$ is $G$-equivariant. Since $\Stab{\NN} \subset \StabL{\NN}$ consists of a single $G$-orbit it follows that $\restrict{\NN}$ is surjective and that
\[
V(\CC,\NN) \cap \Stab{\CC} = \restrict{\NN}^{-1}(\Stab{\NN}) \cong \Stab{\NN} \times \restrict{\NN}^{-1}(\sigma)
\]
for any $\sigma\in \Stab{\NN}$.

For some choices of stability condition in $\Stab{\CC/\NN}$ the subset of possible massless phases is empty, \ie there is no lax stability condition having the chosen massive part -- see \cref{ex:lax supported versus weak}. For others it can be a proper subset of the reals --- see \eg \cref{subsec:nilrep} --- and for yet others it can be the entire real line -- see \eg \cref{subsec:A2}.

\begin{corollary}
The union of strata of codimension less than one in $\StabL{\CC}$ is a manifold with boundary $\bigcup_{\rk(\Lambda_\NN)=1} \StabLN{\CC}{\NN}$ and interior $\StabLN{\CC}{0} = \Stab{\CC}$.
\end{corollary}

\begin{proof}
By the above $\StabL{\NN}$ is a manifold with boundary $\StabLN{\NN}{\NN}$. The local model then shows that $V(\CC,\NN)$ is a manifold with boundary $\StabLN{\CC}{\NN}$. The result follows.
\end{proof}

\section{The quotient stability space}
\label{sec:space qstab}

\noindent
As in the previous section, $\CC$ is a Hom-finite $\kk$-linear triangulated category with a fixed surjective homomorphism $v\colon K(\CC) \to \Lambda$ to a free lattice of finite rank. Define an equivalence relation on the points of $\StabL{\CC}$ by $\sigma \sim \tau$ if they have the same massless subcategory, $\NN_\sigma=\NN_\tau=\NN$ say, and $\massive{\NN}(\sigma)=\massive{\NN}(\tau)$ in $\Stab{\CC/\NN}$. In particular, equivalent lax stability conditions have the same charge. 

\begin{definition}
\label{def:StabQ}

The \defn{quotient stability space} $\StabQ{\CC} \coloneqq \StabL{\CC} / {\sim}$ equipped with the quotient topology. We denote the quotient map by $q$ so that there is a commuting diagram
\[ 
\begin{tikzcd}
\StabL{\CC} \ar{rr}{q} \ar{dr}[swap]{\cm}&& \StabQ{\CC} \ar{dl}{\cm} \\
& \Hom{\Lambda}{\C}
\end{tikzcd} \]
where we use the same notation for the charge map on the quotient. 
\end{definition}

The theme of this section is that $\StabQ{\CC}$ is a stratified space, with $\Stab{\CC}$ as its unique open stratum. The first step is to specify a decomposition. Let $\StabQN{\CC}{\NN} \coloneqq \StabLN{\CC}{\NN}/{\sim}$ be the subspace with massless subcategory $\NN$. Evidently, as a set,
\[
\StabQ{\CC} = \bigsqcup_{\NN \in \Thick{\CC}}\StabQN{\CC}{\NN} 
\]
is the disjoint union of these subsets, which we refer to as its \defn{strata}. We do not assume the strata are connected. The maximal and minimal dimensional strata are easy to identify. If the massless subcategory $\NN=0$ then $\massive{\NN}(\sigma)=\sigma$ so there are homeomorphisms 
\[
\Stab{\CC} \cong \StabLN{\CC}{0} \cong \StabQN{\CC}{0} .
\]
Thus the space of stability conditions embeds continuously in $\StabQ{\CC}$ as the unique open, dense stratum. At the other extreme if the massless subcategory $\NN=\CC$ then $\massive{\NN}(\sigma)$ is the unique stability condition in $\Stab{0}$. Hence $\StabQ{\CC}$ has a unique $0$-dimensional stratum. 

\subsection{Local structure of the quotient stability space}

We begin by describing how masses and phases vary in $\StabQ{\CC}$. For $q(\sigma)\in \StabQN{\CC}{\NN}$ and $c\in \CC$ we define
\[ m_{q(\sigma)}(c) \coloneqq m_\sigma(c), \qquad \phi^-_{q(\sigma)}(c) \coloneqq \phi^-_{\massive{\NN}(\sigma)}(c), \qquad \phi^+_{q(\sigma)}(c) \coloneqq \phi^+_{\massive{\NN}(\sigma)}(c) \]
to be the mass and the minimal/maximal phases of the semistable factors of $c$ in $\CC/\NN$; the latter two assume that $c \in \CC$ is $\sigma$-massive.
Note that $m_{q(\sigma)}(c)$ is the mass $m_{\massive{\NN}(\sigma)}(c)$ of $c$ in the corresponding stability condition on $\CC/\NN$. In concrete terms \smash{$\phi^-_{q(\sigma)}(c)$} is the minimal phase of a massive factor of the HN filtration of $c \in \CC$ with respect to $\sigma$, and analogously for $\phi^+_{q(\sigma)}(c)$.

\begin{lemma}
For each $c\in \CC$ the mass $m_\bullet(c) \colon \StabQ{\CC} \to \R_{\geq 0}$ is continuous. The phases $\phi_\bullet^\pm(c) \colon m_\bullet(c)^{-1}(\R_{>0}) \to \R$ are well defined on the open subset where $c$ is massive and $\phi_\bullet^-(c)$ and $\phi_\bullet^+(c)$ are respectively upper and lower semi-continuous.
\end{lemma}

\begin{proof}
The mass of $c$ is continuous on $\StabL{\CC}$ and constant on equivalence classes by definition. Therefore masses are continuous on $\StabQ{\CC}$ too.
We show that $\phi_\bullet^-(c)$ is upper semi-continuous; the proof that $\phi_\bullet^+(c)$ is lower semi-continuous is similar.
Let $\epsilon > 0$ and choose $\delta > 0$ such that $\delta$ is smaller than the minimum of $\epsilon$ and half the smallest phase gap between the $\sigma$-semistable factors of $c$. 
Let $c_0$ be the massive $\sigma$-semistable factor of $c$ of smallest phase, $\phi_{q(\sigma)}^-(c)$. 
Then $c_0$ is $\tau$-massive for $\tau \in B_\epsilon^\delta(\sigma)$ by \cref{lem:thick subcats in nbhd} and therefore $c$ has a massive $\tau$-semistable factor of phase at most $\phi_{q(\sigma)}^-(c) + \delta \leq \phi_{q(\sigma)}^-(c) + \epsilon$. In particular, $\phi_{q(\tau)}^-(c) \leq \phi_{q(\sigma)}^-(c) + \epsilon$.
\end{proof}

Second, we describe the equivalence relation locally, in a neighbourhood of $\StabLN{\CC}{\NN}$.

\begin{lemma}
\label{lem:local equivalence}
For $\sigma,\tau \in V(\CC,\NN) \cap \StabLN{\CC}{\MM}$ the following are equivalent:
\begin{enumerate}
\item $\massive{\MM}(\sigma)= \massive{\MM}(\tau)$,
\item $\massive{\MM}\circ \restrict{\NN}(\sigma) = \massive{\MM} \circ \restrict{\NN}(\tau)$ and $\massive{\MM}\circ \Phi_\NN(\sigma) = \massive{\MM}\circ \Phi_\NN(\tau)$,
\item $\massive{\MM}\circ \restrict{\NN}(\sigma) = \massive{\MM}\circ \restrict{\NN}(\tau)$ and $\massive{\NN}\circ \Phi_\NN(\sigma) = \massive{\NN}\circ \Phi_\NN(\tau)$.
\end{enumerate}
\end{lemma}
\begin{proof}
Suppose $\sigma \in V(\CC,\NN)$. Then $\sigma = \deform{\NN}(\Phi_\NN(\sigma),\restrict{\NN}(\sigma))$. Moreover, if $\sigma$ has massless subcategory $\MM$ then it follows from the proof of \cref{prop:deformation off stratum} that
\begin{equation} \label{eq: M deformation}
\massive{\MM}(\sigma) = \massive{\MM}\circ \deform{\NN}(\Phi_\NN(\sigma),\restrict{\NN}(\sigma)) = \deform{\NN/\MM}(\massive{\MM}\circ \Phi_\NN(\sigma),\massive{\MM}\circ \restrict{\NN}(\sigma))
\end{equation}
is constructed by deforming $\massive{\MM}\circ \Phi_\NN(\sigma)$ in the direction of $\massive{\MM}\circ \restrict{\NN}(\sigma)$. The equivalence of the first two items follows from the injectivity of $\deform{\NN/\MM}$.

The second statement implies the third since $\Phi_\NN(\sigma) \in \StabLN{\CC}{\NN} \cap \overline{ \StabLN{\CC}{\MM} }$ whenever $\sigma\in V(\CC,\NN)\cap \StabLN{\CC}{\MM}$ and
\[
\begin{tikzcd}
\StabLN{\CC}{\NN} \cap \overline{ \StabLN{\CC}{\MM} } \ar{d}[swap]{\massive{\MM}} 
 \ar{dr}{\massive{\NN}} & \\
\StabLN{\CC/\MM}{\NN/\MM} \ar{r}[swap]{\massive{\NN/\MM}}
& \Stab{\CC/\NN}
\end{tikzcd}
\]
commutes. Conversely, the third implies the second because if $\massive{\NN}\circ \Phi_\NN(\sigma) = \massive{\NN}\circ \Phi_\NN(\tau)$ then $\massive{\MM}\circ \Phi_\NN(\sigma)$ and $\massive{\MM}\circ \Phi_\NN(\tau)$ have the same charge, and the slicing of $\massive{\MM}\circ \Phi_\NN(\sigma)$ can be uniquely reconstructed from those of $\massive{\NN}\circ \Phi_\NN(\sigma)$ and $\massive{\MM}\circ \restrict{\NN}(\sigma)$, and analogously with $\tau$ in place of $\sigma$.
\end{proof}

For brevity let $W(\CC,\NN) \coloneqq V(\CC,\NN)/{\sim}$ be the image of the open neighbourhood $V(\CC,\NN)$ of $\StabLN{\CC}{\NN}$ in the quotient. 

\begin{remark} \label{rem:quotient-decoupling}
In general we do not know if $W(\CC,\NN)$ is open in $\StabQ{\CC}$, \ie if
 \[
q^{-1}(W(\CC,\NN)) = \{\tau \in \StabL{\CC} \mid \tau \sim \sigma \in V(\CC,\NN)\}
\]
is open in $\StabL{\CC}$. 
Moreover, it is unclear if $V(\CC,\NN)$ is a union of equivalence classes, and therefore if the local product structure of \cref{thm:StabL structure summary}(4) descends to $\StabQ{\CC}$. 
However, it is a stratum-wise local model for $\StabQ{\CC}$ near $\StabQN{\CC}{\NN}$ in the sense that it intersects each stratum whose closure intersects $\StabQN{\CC}{\NN}$ in a non-empty open subset --- see item (\ref{stratum-wise local model}) below.
\end{remark}

\begin{theorem}
\label{thm:stabq structure}
The quotient stability space $\StabQ{\CC}$ is a topological space with a decomposition into (possibly empty) strata $\StabQN{\CC}{\NN}$ indexed by thick subcategories $\NN \subseteq \CC$:
\[
\StabQ{\CC} = \bigsqcup_{\NN \in \Thick{\CC}} \StabQN{\CC}{\NN} .
\]
There is a continuous map $\cm \colon \StabQ{\CC}\to \Hom{\Lambda}{\C}$ sending $\StabQN{\CC}{\NN}$ to $\Hom{\Lambda/\Lambda_\NN}{\C}$. This decomposition is a stratified space in the following sense:
\begin{enumerate}
\item The stratum $\StabQN{\CC}{0} \cong \Stab{\CC}$ is open and dense.
\item The strata are locally closed.
\item 
\label{discrete fibres}
The fibres of $\cm \colon \StabQ{\CC} \to \Hom{\Lambda}{\C}$ are discrete.
\item 
There is a commutative diagram of local homeomorphisms:
\begin{equation}
\label{eqn:stabq strata}
\begin{tikzcd}
   \StabQN{\CC}{\NN}  \ar[hookrightarrow]{rr}{\massive{\NN}} \ar{dr}[swap]{\cm}&&  \Stab{\CC/\NN} \ar{dl}{\cm}\\
 &  \Hom{\Lambda/\Lambda_\NN}{\C}.
\end{tikzcd}
\end{equation}
The image of $\massive{\NN}$ is open and $\StabQN{\CC}{\NN}$ is a complex manifold of dimension $\rk(\Lambda/\Lambda_\NN)$.
\item 
\label{frontier condition}
The decomposition satisfies the frontier condition: for any thick subcategory $\MM\subseteq \CC$,
\[
\overline{\StabQN{\CC}{\MM}} = \bigcup_{\StabQN{\NN}{\MM} \neq \emptyset} \StabQN{\CC}{\NN} .
\]
\item \label{bded height} The poset of strata, ordered by inclusion of their closures, has height bounded by $\rk(\Lambda)$.
\item 
\label{stratum-wise local model}
The intersection $W(\CC,\NN) \cap \StabQN{\CC}{\MM}$ is open in $\StabQN{\CC}{\MM}$ for thick $\MM\subset \CC$, and 
\[
\StabQN{\CC}{\NN} \subset \overline{\StabQN{\CC}{\MM}} \iff W(\CC,\NN) \cap \StabQN{\CC}{\MM} \neq \emptyset.
\]
\item There is an almost-stratum-preserving deformation retraction
\[
  W(\CC,\NN)\times [0,1] \to W(\CC,\NN), \quad (q(\sigma),t) \mapsto q(\Phi_{\NN,t}(\sigma))
\]
of $W(\CC,\NN)$ onto the stratum $\StabQN{\CC}{\NN}$. In particular $W(\CC,\CC) \cong \StabQ{\CC}$ deformation retracts onto $\StabQN{\CC}{\CC} \cong \{*\}$ so that $\StabQ{\CC}$ is contractible.
\item The right action of the universal cover $G$ of $\GL$ on $\Stab{\CC}$ extends continuously to an action on $\StabQ{\CC}$ preserving the strata. This action is equivariant with respect to the open embedding $\StabQN{\CC}{\NN} \inj \Stab{\CC/\NN}$. The subgroup $\C\subset G$ acts freely except on the deepest stratum $\StabQN{\CC}{\CC}$, which it fixes.
\item The left action of $\Aaut{\Lambda}{\CC}$ on $\Stab{\CC}$ extends continuously to $\StabQ{\CC}$ with an auto-equivalence $\alpha\in \Aaut{\Lambda}{\CC}$ mapping $\StabQN{\CC}{\NN}$ to $\StabQN{\CC}{\alpha(\NN)}$.
\end{enumerate}
\end{theorem}

\begin{proof}

\begin{enumerate}[wide, labelwidth=!, labelindent=1em]
\item 
We observed above that $\StabQN{\CC}{0} \cong \Stab{\CC}$ is the unique open stratum of $\StabQ{\CC}$. It is dense because $\Stab{\CC}$ is dense in $\StabL{\CC}$, or by part \eqref{frontier condition}.
\item 
The strata are locally closed because the continuous map $\StabL{\CC} \to \Thick{\CC}$ assigning the massless subcategory to a lax stability condition factors continuously through the quotient map $\StabL{\CC} \to \StabQ{\CC}$. 
\item 
To show that $\cm \colon \StabQ{\CC} \to \Hom{\Lambda}{\C}$ has discrete fibres we show that the equivalence relation identifies lax stability conditions in the same connected component of the fibre of $\StabL{\CC}\to \Hom{\Lambda}{\C}$. Suppose $\sigma=(P,Z)$ is a $\delta$-lax stability condition with massless subcategory $\NN$. Let $\tau=(Q,Z) \in B_\epsilon^\delta(\sigma)$ for sufficiently small $0 < \epsilon<1$. Then $\tau$ has massless subcategory $\NN_\tau \subseteq \NN$ and since $\restrict{\NN}(\tau) = (Q_\NN,0)$, because $Z|_{\Lambda_\NN} = 0$, in fact $\NN_\tau=\NN$. Moreover $\massive{\NN}(\sigma), \massive{\NN}(\tau) \in \Stab{\CC/\NN}$ have the same charge and $d(P_{\CC/\NN}, Q_{\CC/\NN})\leq d(P,Q)< \epsilon <1$, so $\massive{\NN}(\sigma)=\massive{\NN}(\tau)$ by \cite[Lem.~6.4]{Bridgeland07}. Hence $\sigma\sim\tau$.
\item
By \cref{thm:StabL structure}, $q(\sigma) \mapsto \massive{\NN}(\sigma)$ is a homeomorphism from $\StabQN{\CC}{\NN}$ onto an open subset of $\Stab{\CC/\NN}$. In particular $\StabQN{\CC}{\NN}$ can be given the structure of a complex manifold of dimension $\rk(\Lambda/\Lambda_\NN)$. This homeomorphism fits into the diagram (\ref{eqn:stabq strata}). Since the right hand map is a local homeomorphism, so is the left hand. 
\item 
From $\StabQN{\NN}{\MM} \neq \emptyset \iff \StabLN{\NN}{\MM} \neq \emptyset$ we get
\[
q^{-1}\bigg( \bigcup_{\StabQN{\NN}{\MM} \neq \emptyset} \StabQN{\CC}{\NN} \bigg) = \bigcup_{\StabLN{\NN}{\MM} \neq \emptyset} \StabLN{\CC}{\NN}.
\]
This is closed by \cref{lem:closed union of strata}. Indeed it is the minimal closed union of equivalence classes containing $\StabLN{\CC}{\MM}$ because it follows from \cref{thm:StabL structure summary} (\ref{local model V}) that 
\[
\StabLN{\CC}{\NN} = q^{-1}(q(\StabLN{\CC}{\NN} \cap \overline{\StabLN{\CC}{\MM}})).
\]
 Thus,
 $\overline{\StabQN{\CC}{\MM}} = \bigcup_{\StabQN{\NN}{\MM} \neq \emptyset} \StabQN{\CC}{\NN}$ as claimed. In particular, the decomposition satisfies the frontier condition because the closure of any stratum is a union of strata.
\item
Suppose $\StabQN{\CC}{\NN'} \subset \overline{\StabQN{\CC}{\NN}}$. Then as $\cm \colon \StabQ{\CC} \to \Hom{\Lambda}{\C}$ has discrete fibres and restricts to a local homeomorphism $\StabQN{\CC}{\NN} \to \Hom{\Lambda/\Lambda_\NN}{\C}$ we conclude that $\Hom{\Lambda/\Lambda_{\NN'}}{\C} \subset \Hom{\Lambda/\Lambda_{\NN}}{\C}$, and therefore that $\rk(\Lambda/\Lambda_{\NN'}) < \rk(\Lambda/\Lambda_{\NN})$. The result follows.
\item
For each thick $\MM\subset \CC$ there is a commuting diagram
\[
\begin{tikzcd}
\StabLN{\CC}{\MM} \ar{r}{q} \ar{d}[swap]{\massive{\MM}} & \StabQN{\CC}{\MM} \ar{d}{\cm}\\
\Stab{\CC/\MM} \ar{r}[swap]{\cm} & \Hom{\Lambda/\Lambda_{\MM}}{\C} 
\end{tikzcd}
\]
in which the left hand map is open by \cref{thm:StabL structure}, the bottom map is a local homeomorphism by \cref{thm:deformation theorem}, and the right hand map is a local homeomorphism by (\ref{discrete fibres}). Since $V(\CC,\NN)$ is open it follows that 
\[
W(\CC,\NN) \cap \StabQN{\CC}{\MM} = q\left( V(\CC,\NN) \cap \StabLN{\CC}{\MM} \right)
\]
is also open. Moreover, by (\ref{frontier condition}) and \Cref{thm:StabL structure summary} 
\begin{align*}
\StabQN{\CC}{\NN} \subset \overline{\StabQN{\CC}{\MM} } 
& \iff \StabQN{\NN}{\MM} \neq \emptyset\\
& \iff \StabLN{\NN}{\MM} \neq \emptyset\\
& \iff V(\CC,\NN) \cap \StabLN{\CC}{\MM} \neq \emptyset\\
&\iff W(\CC,\NN) \cap \StabQN{\CC}{\MM} \neq \emptyset.
\end{align*}
\item 
\Cref{prop:deformation retract 1} implies that $q(\sigma) \mapsto q(\Phi_{\NN,t}(\sigma))$ is a deformation retraction of $W(\CC,\NN)$ onto $\StabQN{\CC}{\NN}$ provided that this map is well defined. To show it is, suppose $\sigma,\tau \in V(\CC,\NN)$ with $\sigma \sim \tau$. Let $\MM$ be the common massless subcategory. Then the third clause of \cref{lem:local equivalence} yields $\Phi_\NN(\sigma)\sim \Phi_\NN(\tau)$. Moreover, $\Phi_{\NN,t}(\sigma)$ and $\Phi_{\NN,t}(\tau)$ have common massless subcategory $\MM$ for $0<t\leq 1$. Since
\[
\massive{\MM}\circ \Phi_{\NN,t}(\sigma) = \massive{\MM} \circ \deform{\NN}(\Phi_\NN(\sigma) , t \cdot \restrict{\NN}(\sigma)) = \deform{\NN/\MM}( \massive{\MM}\circ \Phi_\NN(\sigma) , \massive{\MM}\circ t \cdot \restrict{\NN}(\sigma)),
\] 
where the second equality is given by equation \eqref{eq: M deformation} in the proof of \cref{lem:local equivalence}. Proceeding analogously for $\tau$, we see that the second clause of \cref{lem:local equivalence} implies $\Phi_{\NN,t}(\sigma) \sim \Phi_{\NN,t}(\tau)$. Hence, the map is well defined. 

That $W(\CC,\CC)=\StabQ{\CC}$ follows from \cref{rmk:V examples}, and $\StabQN{\CC}{\CC}=\{*\}$ is the unique class of lax stability conditions with zero charge. The resulting deformation retraction of $\StabQ{\CC}$ to a point is nothing but uniform mass dilation.
\item
 The action of $G$ on $\StabL{\CC}$ descends to an action on $\StabQ{\CC}$ because it preserves the strata of $\StabL{\CC}$ and is equivariant with respect to $\massive{\NN} \colon \StabLN{\CC}{\NN} \to \Stab{\CC/\NN}$ by \cref{lem:actions on dstab}. It follows that it is equivariant with respect to the open embedding $\StabQN{\CC}{\NN} \inj \Stab{\CC/\NN}$. Since $\C$ acts freely on $\Stab{\CC/\NN}$ when $\rk(\Lambda/\Lambda_\NN)>0$ we conclude it acts freely on $\StabQ{\CC} \setminus \StabQN{\CC}{\CC}$.
\item
Similarly, the action of $\Aaut{\Lambda}{\CC}$ descends from $\StabL{\CC}$ to $\StabQ{\CC}$ by \cref{lem:actions on dstab} because $\massive{\NN}$ is equivariant. 
\qedhere
\end{enumerate}
\end{proof}

\begin{corollary}
\label{cor:lax poset bded}
The poset of strata in $\StabL{\CC}$, see \cref{lem:lax partial order}, has height bounded by $\rk(\Lambda)$. 
\end{corollary}
\begin{proof}
This follows immediately from (\ref{bded height}) above because $\StabLN{\CC}{\NN} \mapsto \StabQN{\CC}{\NN}$ is an isomorphism from the poset of strata of $\StabL{\CC}$ to that of $\StabQ{\CC}$.
\end{proof}

\subsection{Codimension one strata}
\label{sec:codim1}

Suppose, as in \cref{subsec:codim one lax} that $\NN$ is a thick subcategory of $\CC$ for which the saturation $\Lambda_\NN$ of the image of $K(\NN) \to K(\CC) \to \Lambda$ is a rank one lattice with generator $\lambda \in \Lambda_\NN$. Recall that $\StabLN{\CC}{\NN} \neq \emptyset \iff \Stab{\NN}\neq \emptyset \iff \NN$ has an algebraic heart, $\HH$ say. When this is the case, $\StabL{\NN}$ is the universal cover of the real oriented blowup of $\Hom{\Lambda_\NN}{\C}$ at the origin. The charge map is the composite of the covering map and blowdown, and the slicing is determined by $\HH=P(\phi)$ where $Z(\HH)\subset \R_{\geq 0}e^{\pi i \phi}$.

By definition $\StabQ{\NN} = \StabL{\NN}/\sim$ is the quotient by the equivalence relation identifying points in the cover of the exceptional divisor. More explicitly, we have a commutative diagram
\begin{center}
\begin{tikzcd}
\StabQ{\NN} \ar{d} \ar[hookleftarrow]{r}& 
\Stab{\NN}\cup \Stab{0} \ar{d} \ar{rr} &&
\C \cup \{-\infty\} \ar{d}{\exp} \\
\Hom{\Lambda_\NN}{\C} \ar[hookleftarrow]{r}& 
\Hom{\Lambda_\NN}{\C}^* \cup \Hom{0}{\C} \ar{rr}{Z \mapsto Z(\lambda)} && 
\C^*\cup\{0\}
\end{tikzcd}
\end{center}
in which the vertical arrows are the charge projections and the horizontal arrows are homeomorphisms. A neighbourhood $U$ of $-\infty$ is open precisely if $U\cap \C$ is open and for each $y\in \R$ there is some $a\in \R$ such that $\{x+iy \mid x<a\} \subset U$.

The subset $W(\CC,\NN)$ is an open neighbourhood of the stratum $\StabQ{\CC,\NN}$ because
\[
q^{-1}(W(\CC,\NN)) = \StabL{\CC,\NN} \cup ( \Stab{\CC} \cap V(\CC,\NN)) = V(\CC,\NN)
\]
is open. The stratum $\StabQ{\CC,\NN}$ corresponds to the open subset of $\Stab{\CC/\NN}$ consisting of stability conditions which lift to $\StabL{\CC,\NN}$. The fibre of $q$ is the set of such lifts; it is homeomorphic to the open subset of compatible massless phases of objects in the heart $\HH$. Open neighbourhoods of such a point $\sigma$ correspond under $q$ to open neighbourhoods of the fibre $q^{-1}(\sigma)$ in $\StabL{\CC}$. It follows that $q\colon \StabL{\CC} \to \StabQ{\CC}$ is not an open map when $\StabL{\CC} \neq \Stab{\CC}$. To see this note that each point of $q^{-1}(\sigma)$ has an open neighbourhood $U$ not containing the entire fibre $q^{-1}(\sigma)$, and therefore such that $q(U)$ is not open.

\begin{remark}
\label{rmk:W(C,N) embedding}
Recall from \cref{thm:StabL structure summary} that $V(\CC,\NN) \hookrightarrow \Stab{\CC/\NN}\times \StabL{\NN}$ is a homeomorphism onto its image, which is an open subset. This descends to a continuous embedding 
\[
W(\CC,\NN) \hookrightarrow \Stab{\CC/\NN}\times \StabQ{\NN}
\]
between the quotients by \cref{lem:local equivalence}. However, even in the codimension one setting this needn't be a homeomorphism onto its image, and the image need not be open. If there is $\sigma \in \StabQ{\CC,\NN}$ for which $q^{-1}(\sigma)$ is (identified with) a proper non-empty subset of $\R$ then the topology on $W(\CC,\NN)$ is finer and the image of $W(\CC,\NN)$ itself is not open. See \cref{subsec:nilrep} for a concrete example in which this occurs. 
\end{remark}

\section{Representation finite examples}
\label{sec:finite type components}

\noindent
Throughout this section, we impose the following finiteness conditions on $\CC$:
\begin{enumerate}
\item $K(\CC) \cong \Z^n$ and $v=\id \colon K(\CC) \to \Lambda$.
\item The hearts $P(\phi,\phi+1]$ have finitely many indecomposable objects for all $(P,Z) \in \Stab{\CC}$ and $\phi\in \R$.
\item The set of such hearts is closed under Happel--Reiten--Smalø tilts.
\end{enumerate}
This is a very strong assumption. An immediate consequence is that the heart $P(\phi,\phi+1]$ has finitely many stable objects and therefore the phase distribution $\{ \phi \in \R \mid P(\phi)\neq 0\}$ is finite. This implies that $P(\phi,\phi+1]=P(\phi+\epsilon,\phi+1]$ for some $\epsilon>0$ and so is algebraic (see, \eg \cite{QW18}), \ie a length category with finitely many simple objects.

The assumption holds for the bounded derived categories of path and Ginzburg algebras of ADE Dynkin quivers, as well as for silting-discrete algebras such as derived-discrete algebras; see \cite{Keller12, Aihara13, Vossieck01}, respectively, for the definitions, and \cite{QW18, PSZ18, BPP17}, respectively, for verification of the assumption in each case.

For simplicity of notation we assume that $\Stab{\CC}$ is connected; otherwise consider a component. Our first result is that in this section's setup, the lax closure is just the closure of the stability manifold. This is not true in general; see \cref{subsec:nilrep}.

\begin{proposition}
\label{prop:boundary support}
The lax closure $\StabL{\CC} = \overline{\Stab{\CC}}$, where the closure is taken in the product $\Slice{\CC} \times \Hom{K(\CC)}{\C}$. 
\end{proposition}

\begin{proof}
Suppose $\sigma = (P,Z) \in \overline{\Stab{\CC}}$ does not satisfy lax support. Then for any $\delta>0$ we can find a sequence $(c_n)$ of massive indecomposable $\delta$-slim objects with $m_\sigma(c_n) / \norm{c_n} \to 0$. As the phase distribution of $\sigma$ is finite, we can choose $\delta>0$ sufficiently small that all $\delta$-slim objects are semistable. In particular, all $c_n$ are semistable. Shifting if necessary, we may assume $c_n \in P(\phi_n)$ for some $\phi_n \in (0,1]$. As there are only finitely many non-zero slices in $P(0,1]$, we can pass to a subsequence of massive indecomposable objects in $P(\phi)$ for some $\phi \in (0,1]$. By assumption there are only finitely many of these, so we can pass to a constant subsequence $(c_n)$. This contradicts the assumption that $c_n$ is massive, so $\sigma$ satisfies lax support after all.
\end{proof}

\begin{proposition}
\label{cor:boundary algebraic}
Let $(P,Z) \in \StabL{\CC}$. Then for any $\phi \in \R$ the full subcategory $P(\phi,\phi+1]$ is the heart of a stability condition in $\Stab{\CC}$. Moreover, $P$ has finite phase distribution.
\end{proposition}

\begin{proof}
Rotating phases by the $\C$ action it suffices to prove the first part for $\phi=0$. Choose $(Q,W)$ in $\Stab{\CC}$ with $d(P,Q) < \tfrac{1}{2}$ so that the hearts $P(0,1]$ and $Q(0,1]$ are both contained in $Q(-\tfrac{1}{2},\tfrac{3}{2}] = Q(\tfrac{1}{2},\tfrac{3}{2}] * Q(-\tfrac{1}{2},\tfrac{1}{2}]$. Then $P(0,1]$ is obtained by first tilting from $Q(0,1]$ to $Q(-\tfrac{1}{2},\tfrac{1}{2}]$ and then tilting from that to $P(0,1]$. Since the set of hearts of stability conditions in the component $\Stab{\CC}$ is closed under tilting, $P(0,1]$ is the heart of some stability condition in $\Stab{\CC}$. Therefore $P(0,1]$ is algebraic. A similar argument shows that $P[0,1)$ is algebraic.

Since $P(\phi,\phi+1]$ and $P[\phi-1,\phi)$ are algebraic there exists $\epsilon>0$, depending on $\phi$, for which $P(\phi, \phi+\epsilon) = 0 = P(\phi-\epsilon,\phi)$. The same holds for any $\phi\in \R$ because we can translate phases using the $\C$ action. Therefore $P$ has finite phase distribution.
\end{proof}

\begin{corollary}
\label{cor:finite type massless}
For a thick subcategory $\NN \subseteq \CC$, the following conditions are equivalent:
\begin{enumerate}[label=(\roman*)]
\item $\NN$ is the massless subcategory of some $\sigma \in \StabL{\CC}$;
\item $\NN$ is the triangulated hull of a finite set of simple objects in the heart of some $\tau \in \Stab{\CC}$.
\end{enumerate}
\end{corollary}

\begin{proof}
Suppose $\NN$ is generated by a finite subset $I$ of simple objects in the heart $Q(0,1]$ of some $\tau = (Q,W)\in \Stab{\CC}$. Deforming the charge, we may assume $W(s)=-1$ if $s\notin I$ and $W(s)=-r$ if $s\in I$, for some $r\in \R_{>0}$. Clearly the charges converge as $r\to 0$. Moreover, since $Q(0,1]= Q(1)$ for all $r$ the slicings also converge as $r\to 0$. In the limit as $r\to 0$ we obtain a lax pre-stability condition $\sigma$ with massless subcategory $\NN$ and slicing $Q$. Since $\sigma\in \overline{\Stab{\CC}}$ it is in $\StabL{\CC}$ by \cref{prop:boundary support}.
 
Now suppose $\sigma = (P,Z) \in \StabL{\CC}$. By \cref{lem:massless simples} the massless subcategory $\NN$ is the triangulated closure of the set of stable massless objects in $P(0,1]$. Since $P(0,1]$ is algebraic each stable massless object has a composition series whose factors are massless simple objects, each of which is of course also stable. Therefore $\NN=\triang{\CC}{S}$ is the triangulated hull of the set $S$ of massless simple objects in $P(0,1]$. The result follows because $P(0,1]$ is the heart of some stability condition in $\Stab{\CC}$ by \cref{cor:boundary algebraic}.
\end{proof}

\begin{remark}
\label{rmk:non-admissible}
\Cref{cor:finite type massless} tells us that massless subcategories need not be admissible, \ie the inclusion functor $\NN \hookrightarrow \CC$ need not admit a left or right adjoint. An example is the derived-discrete algebra $A = \Lambda(2,1,0)$; see \cite{BGS04} for notation and \cite[\S 8]{KY14} for a detailed description of $\Db(A)$. One of the two simple modules $S$ is a $2$-spherical object. Therefore, its thick hull $\NN \coloneqq \thick{}{S}$ occurs as a massless subcategory. On the other hand, the duality property $\Hom{S}{-} = \Hom{-}{S[2]}^*$ implies that an adjoint of the inclusion $\NN \inj \Db(A)$ would lead to a splitting $\Db(A) \cong \NN \oplus \MM$ but $\Db(A)$ is indecomposable.
\end{remark}

\section{Closures of $G$-orbits}
\label{sec:orbit closures}

\noindent
Recall from \cref{subsec:group actions} that the universal cover $G$ of $\GL$ acts on stability spaces and our partial compactifications. In this section we describe the closures of free $G$-orbits in $\Stab{\CC}$, in $\StabL{\CC}$ and in $\StabQ{\CC}$ by explicitly constructing strict and lax stability conditions in the boundary. This provides examples of lax and quotient stability conditions, and is also the key technical ingredient in describing two-dimensional stability spaces in the following section.

The phase diagrams $\Phi_\sigma = \{ \phi + \Z : P(\phi) \neq 0 \} \subset \R/\Z$ of stability conditions $\sigma = (P,Z)$ in the same orbit are related by orientation-preserving diffeomorphisms of the circle $\R/\Z$. The structure of the closure of the orbit in $\Stab{\CC}$, in $\StabL{\CC}^*$, and in $\StabQ{\CC}^*$ can be described in terms of the phase diagram. Here $\StabL{\CC}^*$ and $\StabQ{\CC}^*$ are the spaces obtained by removing the strata where all objects are massless from $\StabL{\CC}$ and $\StabQ{\CC}$ respectively. These are more convenient to consider because the action of $\C$ on them is free. We describe the closure of $(\sigma\cdot G)/\C$ in the respective quotients $\PStab{\CC}$, $\PStabL{\CC}$ and $\PStabQ{\CC}$.

Fix a stability condition $\sigma=(P,Z)$. If there is only one point in the phase diagram $\Phi_\sigma$ then there is a non-trivial stabiliser and $(\sigma\cdot G)/\C$ is a point. In this case the orbit is closed in $\PStab{\CC}$ and also in both partial compactifications. Henceforth, we assume that $\Phi_\sigma$ consists of at least two points. In particular the image of the charge is the whole of $\C$ so that the $G$-orbit through $\sigma$ is free and $(\sigma\cdot G)/\C$ is biholomorphic to the Poincaré disk $\pd$ because $G \cong \C \times \U$.
In the following, we take $\partial \pd = \overline{\pd} \setminus \pd$.

\begin{construction} \label{orbit identification}
We set up an explicit identification $\pd \isom (\sigma\cdot G)/\C$, $w \mapsto \C\cdot \sigma_w$, where the latter is the $\C$-orbit of $\sigma_w \coloneqq (P_w,Z_w)$ which is defined as follows. We fix an explicit identification of the Poincaré disk and the strict upper half-plane by choosing $f \colon \pd \to \U$, $f(w) = i(1+w)/(1-w)$,
which extends to a homeomorphism $f\colon \overline{\pd} \to \U\cup\R\cup\{\infty\}$.

Next, we use the map to define a family of real transformations of the complex plane
\[ 
M \colon \overline{\pd} \to \End{\R}{\C}, \quad w \mapsto
   \begin{cases}
     M_w(1)=1, \quad M_w(i) = f(w) & \text{if } w\neq 1; \\
     M_1(1)=0, \quad M_1(i) =-1
   \end{cases}
\]
and in turn define a one-parameter family of charges by $Z_w \coloneqq M_w \circ Z$.
If $w\in \pd$ then $M_w(i) \in \U$, hence there is a unique slicing $P_w$ that is compatible with $Z_w$ and satisfies $P_w(0,1]= P(0,1]$.
This gives $\sigma_w = (P_w, Z_w)$.

The reason for this particular choice is that for $w\in\partial \pd=S^1$
\begin{align*}
      P(\phi) \subset \ker Z_w 
&\iff M_w(e^{i\pi \phi})=0\\
&\iff M_w(1) \cos(\pi\phi) + M_w(i) \sin(\pi\phi)=0\\
&\iff \left( w=1 \text{ and } \phi=0 \right)\ \text{or}\ \left( w=e^{2\pi i \phi} \text{ and } \phi\neq 0 \right)\\
&\iff w=e^{2\pi i \phi}
\end{align*}
so that points on $\partial \pd$ correspond to charges for which semistable objects of a certain phase have vanishing mass in a natural way.
For $w =e^{2\pi i \phi}\in \partial \pd$, $f(w) = - \cos(\pi \phi)/\sin(\pi \phi)$, so for $c \in P(\phi')$ the trigonometric angle sum formula yields
\begin{equation}
\label{eqn:angle sum charge equation}
Z_w(c) = \begin{cases} 
 |Z(c)| \dfrac{\sin\pi(\phi-\phi')}{\sin\pi\phi} & \text{ if } w \neq 1; \\
-|Z(c)| \sin \pi \phi'                           & \text{ if } w = 1.
\end{cases}
\end{equation}
\end{construction}

If for $w = e^{2\pi i \phi}$, $P(\phi)=0$, \ie \cref{orbit identification} yields a charge $Z_w$ with trivial kernel, then there is a unique choice of slicing $P_w$ which is compatible with $Z_w$ and with $P(0,1] \subset P_w[0,1]$, namely 
\[
P_w(1)=P(\phi,\phi+1] = P[\phi,\phi+1) 
\]
and all slices with phase in $(0,1)$ are zero. One can verify that $\sigma_w=(P_w,Z_w)$ is a pre-stability condition. 
Below we give a criterion for $\sigma_w$ to be a stability condition.

\begin{lemma}
\label{lem:non-lax support}
Suppose $w=e^{2\pi i \phi} \in \partial \pd$ and $P(\phi)=0$. Then $\sigma_w=(P_w,Z_w)$ is a stability condition if and only if $\phi$ is not an accumulation point of the phase diagram $\Phi_\sigma$. 
\end{lemma}

\begin{proof}
Suppose $\phi\neq 0$ so that $M_w(1)=1$. 
If $\phi$ is an accumulation point of $\Phi_\sigma$ and $\epsilon>0$ then using the first equation in \eqref{eqn:angle sum charge equation} shows that one can choose $\phi' \in \Phi_\sigma$ sufficiently close to $\phi$ that 
 \[
 |Z_w(c)| \leq \epsilon |Z(c)| \leq \epsilon \norm{c} \norm{Z}
 \]
 for any $c\in P(\phi')$. Hence the support property fails since $c$ remains semistable for $\sigma_w$. 
 
Conversely, if $\phi$ is not an accumulation point then there is $L>0$ such that
 \[
 |Z_w(c)| \geq L |Z(c)| \geq KL \norm{c}
 \]
 for all $c\in P(\phi')$ where $\phi'\in (0,1]$, and $K$ is a support constant for $\sigma$. Therefore the same inequality holds for all $c\in P_w(1)=P(\phi,\phi+1]$ so that $\sigma_w$ satisfies the support property. 
 
 The case $\phi=0$ is similar but uses the second equation in \eqref{eqn:angle sum charge equation}.
\end{proof}

If instead $P(\phi)\neq 0$, then there is a one-parameter family of slicings $P_w^\psi$ for $w = e^{2\pi i \phi}$ compatible with $Z_w$ constructed above and with $P(0,1] \subset P_w[0,1]$. 
They differ by the choice of phase for the massless objects in $P(\phi)$. Namely, for each $\psi\in [0,1]$ there is a unique such slicing $P_w^\psi$ with
\[
P_w^\psi(1) \supset P(\phi,\phi+1) \quad \text{and}\quad P_w^\psi(\psi)\supset P(\phi)
\]
and all other slices with phase in $(0,1]$ zero, where the inclusions are equalities for $\psi \neq 1$.
The pair $\sigma^\psi_w=(P^\psi_w,Z_w)$ is a lax pre-stability condition with massless category $\triang{\CC}{P(\phi)}$. Below, we give criteria for when it is in $\StabL{\CC}$.

\begin{proposition}
\label{prop:lax closure}
Suppose $\sigma=(P,Z) \in \Stab{\CC}$ has $P(\phi)\neq 0$ for some $0<\phi\leq 1$. Let $w=e^{2\pi i \phi} \in \partial \pd$. Then the lax pre-stability condition $\sigma_w^\psi$ is in $\overline{\Stab{\CC}}$ and satisfies lax support if and only if there is some $\epsilon>0$ such that one of the following cases applies:
\begin{enumerate}
\item \label{lax closure 1} $\psi\in(0,1)$ and $P(\phi-\epsilon,\phi)=0=P(\phi,\phi+\epsilon)$;
\item \label{lax closure 2} $\psi=1$, $P(\phi-\epsilon,\phi)=0$ and no indecomposable in $P[\phi,\phi+1)$ has phase in $(\phi,\phi+\epsilon)$;
\item \label{lax closure 3} $\psi=0$, $P(\phi,\phi+\epsilon)=0$ and no indecomposable in $P(\phi-1,\phi]$ has phase in $(\phi-\epsilon,\phi)$.
\end{enumerate}
\end{proposition}

The second two provisions in the proposition are stronger than the first: if no indecomposable in $P[\phi,\phi+1)$ has phase in $(\phi,\phi+\epsilon)$ then $P(\phi,\phi+\epsilon)=0$, and similarly in the other case.

\begin{proof}
First we check that $\sigma_w^\psi \in \overline{\Stab{\CC}}$.
For any $w\in \partial \pd$ one can construct a sequence $(w_n)$ in $\pd$ converging to $w\in \partial \pd$ and such that the limiting phase of objects in $P(\phi)$ is $\psi$. The charges $Z_{w_n}$ converge to $Z_w$ and the slicings $P_{w_n}$ converge to $P_w^\psi$ because each of the provisions implies $\phi$ is isolated in the phase diagram $\Phi_\sigma$. Hence $\sigma_w^\psi$ is in the closure of the orbit $\sigma\cdot G$ in $\Slice{\CC}\times\Hom{\Lambda}{\C}$. 

Suppose $0< \psi <1$ and $w\neq 1$. We show that in this case the lax support property holds for $\sigma_w^\psi$ if and only if provision \eqref{lax closure 1} holds.
By \eqref{eqn:angle sum charge equation}, for any indecomposable object $c \in P_w(1)=P(\phi,\phi+1)$, we have
\[
\frac{|Z_w(c)|}{\norm{c}} = \frac{|Z_w(c)|}{|Z(c)|} \frac{|Z(c)|}{\norm{c}} = \frac{|\sin \pi(\phi(c)-\phi)|}{\sin\pi\phi} \frac{|Z(c)|}{\norm{c}}. 
\]
Therefore, $\delta$-lax support holds for $0<\delta<1-\psi$ if and only if the right hand side is bounded away from $0$. As $1/\sin(\pi\phi)$ is fixed, this is equivalent to $|\sin\pi(\phi(c)-\phi)|$ and $|Z(c)| / ||c||$ being bounded away from $0$; the latter is by the support property for $\sigma$ and holds in any case.
The former is equivalent to $P(\phi,\phi+\epsilon) = 0 = P(\phi+1-\epsilon,\phi+1)$ for some $\epsilon>0$. When this is the case $P(\phi,\phi+1)=P(\phi+\epsilon, \phi+1-\epsilon)$, \ie $\delta$-lax support is equivalent to provision \eqref{lax closure 1}. 
The proof when $w=1$ is similar but uses instead $| Z_w(c)| / |Z(c) | = \sin\pi\phi(c)$. 

When $\psi = 1$ (resp., $\psi = 1$), the proof that the lax support property holds if and only if provision \eqref{lax closure 2} (resp., \eqref{lax closure 3}) holds is essentially the same, but now lax support must be checked for indecomposable $c$ in $P[\phi,\phi+1) \setminus P(\phi)$ (resp., $P(\phi,\phi+1]\setminus P(\phi+1)$), yielding the stronger provisions involving indecomposable objects rather than semistable objects
\end{proof}

\begin{definition}
We define subsets of the closure of the Poincaré disk by 
\begin{align*}
 \pd_\sigma &= \pd \cup \{e^{2\pi i\phi} : \phi \ \text{is not in the closure of $\Phi_\sigma$}\}\\
\text{and}\ \pd^Q_\sigma &= \pd\cup\{ e^{2\pi i \phi} : \phi\ \text{is not an accumulation point of $\Phi_\sigma$} \}.
\end{align*}
We give $\pd_\sigma$ the subspace topology from $\C$; the topology on $\pd^Q_\sigma$ is a refinement of the subspace topology, described below. By construction, $\pd_\sigma \subset \pd^Q_\sigma$ with complement $\pd^Q_\sigma \setminus \pd_\sigma$ the discrete subset of the boundary of the disk consisting of the isolated phases in $\Phi_\sigma$. 

We also define a space $\pd^L_\sigma$ as follows. Take the real oriented blow up of $\pd^Q_\sigma$ at each point in $\pd^Q_\sigma \setminus \pd_\sigma$. Since these points lie on the boundary, this adds a semicircular exceptional divisor over each point in the centre of the blowup with a natural orientation inherited from $\C$. Now delete the end, respectively start, point of the exceptional divisor over $w=e^{2\pi i \phi}$ whenever item (2), respectively (3), of \cref{prop:lax closure} fails. The resulting space $\pd^L_\sigma$ is given the subspace topology from the blow up, and $\pd^Q_\sigma$ is given the quotient topology from this via the blowdown map $\pi \colon \pd^L_\sigma \to \pd^Q_\sigma$. By construction $\pi$ and $\pd_\sigma \inj \pd^Q_\sigma$ are continuous. See \cref{fig:orbit closures} for an illustration.
\begin{figure}
\begin{center}
\begin{tabular}{ccccc}
\begin{tikzpicture}
%draw Poincare disk and bdy cpts in stab
\draw[ultra thick, red, fill=red!10] (130:\ra) arc (130:50:\ra);
%mark missing pt
 \fill[white] (90:\ra)  circle (2pt);
 \node at (90:1.2*\ra) {$w$};

\end{tikzpicture} &&
\begin{tikzpicture}
%fill Poincare disk
\fill[red!10] (130:\ra) arc (130:50:\ra);
 
\begin{scope}
%clip
\clip (130:\ra) arc (130:50:\ra);
%draw bdy cpt in stabl
\draw[ultra thick, blue, fill=white] (90:\ra) circle (10pt);
\end{scope}
% draw bdy cpts in stab
\draw[ultra thick, red] (130:\ra) arc (130:98:\ra);
\draw[ultra thick, red] (82:\ra) arc (82:50:\ra);
 \fill[white] (98:\ra)  circle (2pt);
 \fill[blue] (82:\ra)  circle (2pt);

\end{tikzpicture} &&
\begin{tikzpicture}
%fill Poincare disk
\fill[red!10] (130:\ra) arc (130:50:\ra);
 
% draw bdy cpts in stab
\draw[ultra thick, red] (130:\ra) arc (130:95:\ra);
\draw[ultra thick, red] (90:\ra) arc (90:50:\ra);
 
 % draw bdy pt in stabq
 \fill[blue] (90:\ra)  circle (3pt);

\end{tikzpicture} \\
$\pd_\sigma$ && $\pd_\sigma^L$ && $\pd_\sigma^Q$
\end{tabular}
\end{center}
\caption{
Local pictures of the orbit closures near a point $w=e^{2\pi i \phi}$ with $P(\phi)\neq 0$ and $P(\phi-\epsilon,\phi)=0=P(\phi,\phi+\epsilon)$ for some $\epsilon>0$. The closure $\pd_\sigma$ in $\PStab{\CC}$ is obtained by adding the red boundary arcs for the phase gaps; the closure $\pd_\sigma^L$ in $\PStabL{\CC}$ is obtained by adding the blue exceptional divisor --- to illustrate the two cases this is shown with an endpoint in the closure of the righthand arc, but not in the closure of the lefthand one; finally the closure $\pd_\sigma^Q$ in $\PStabQ{\CC}$ is obtained by collapsing this exceptional divisor to a point --- in the quotient topology this is in the closure of the righthand arc but not of the lefthand one.
}
\label{fig:orbit closures}
\end{figure}
\end{definition}

\begin{proposition}
\label{prop:orbit closure bijection}
The closures of the orbit $\sigma\cdot G$ in the stability space, in its lax closure and in the quotient stability space, are described by the commutative diagram 
\[
\begin{tikzcd}
 \pd_\sigma \ar[hook]{r} \ar{d} & \pd^L_\sigma \ar{r}{\pi} \ar{d}& \pd^Q_\sigma\ar{d}\\
(\Stab{\CC} \cap \overline{\sigma\cdot G})/\C \ar[hook]{r} &
( \StabL{\CC}^* \cap \overline{\sigma\cdot G})/\C \ar{r}{q} &
( \StabQ{\CC}^* \cap \overline{\sigma\cdot G})/\C
 \end{tikzcd}
\]
in which the vertical maps are homeomorphisms extending $\pd \to (\sigma\cdot G)/\C$, $w \mapsto \sigma_w = (P_w,Z_w)$.
\end{proposition}

\begin{proof}
The left hand homeomorphism is constructed in \cref{lem:non-lax support}, and the middle one in \cref{prop:lax closure}. The right hand homeomorphism is induced from the middle one via $q$ and the blowdown $\pi$, noting that both spaces are given the respective quotient topology. 
\end{proof}

In brief, the closure of the orbit of $\sigma=(P,Z)$ 
\begin{itemize}
\item in $\PStab{\CC}$, is obtained by adding a boundary arc $\{ e^{2\pi i \theta} \mid \phi <\theta<\phi'\}$ to the disk for each phase gap $P(\phi,\phi')=0$; stability conditions in this arc have length heart $P(\phi,\phi+1]$; 
\item in $\PStabL{\CC}$, is obtained by adding an additional arc (either with or without endpoints) over $e^{2\pi i \phi}$ for each isolated occupied phase $\phi$; lax stability conditions in this arc have length heart either $P(\phi,\phi+1]$ or $P[\phi,\phi+1)$ and massless category $\triang{\CC}{P(\phi)}$; 
\item in $\PStabQ{\CC}$, is obtained by collapsing these latter arcs to points, \ie adding a point at $e^{2\pi i \phi}$ for each isolated occupied phase $\phi$; the quotient stability condition at this point has length heart $P[\phi,\phi+1)/P(\phi)$ because the quotient of a length category by a Serre subcategory is again length.
\end{itemize}

\begin{corollary}
\label{cor:orbit closed}
The following are equivalent:
\begin{enumerate}
\item the orbit $\sigma\cdot G$ is closed in $\Stab{\CC}$;
\item the orbit $\sigma\cdot G$ is closed in $\StabL{\CC}^*$;
\item the orbit $\sigma\cdot G$ is closed in $\StabQ{\CC}^*$;
\item the phase diagram $\Phi_\sigma$ is dense.
\end{enumerate}
\end{corollary}

\begin{proof}
The orbit $\sigma\cdot G$ is closed in each of the three spaces if, and only if, respectively $\pd_\sigma=\pd$, $\pd_\sigma^L=\pd$ and $\pd_\sigma^Q=\pd$. Each provision is equivalent to the phase diagram $\Phi_\sigma$ being dense.
\end{proof}
\begin{remark}
When the phase diagram is dense, the Bridgeland metric induces the standard hyperbolic metric on the quotient $(\sigma\cdot G)/\C \cong \pd$ up to a factor \cite[Prop.~4.1]{Woolf12}. Since the hyperbolic metric is complete, the orbit is closed.
\end{remark}

\section{Two-dimensional stability spaces}
\label{sec:2d case}

\noindent
We illustrate our results in the case in which the stability space is a two-dimensional complex manifold. In this context there is a very close relationship between the boundary strata we add and the wall-and-chamber structure of the stability space.

Throughout this section, we assume that $\Lambda \cong \Z^2$ is a rank two lattice and $v = \id\colon K(\CC) \to \Lambda$. This implies that any length heart in $\CC$ contains exactly two simple objects up to isomorphism. If they exist, the two non-isomorphic simple objects are usually denoted $s$ and $t$.

\subsection{Walls and chambers}

Since $\rk(\Lambda)=2$ the projective stability space $\PStab{\CC} = \Stab{\CC} / \C$ is a non-compact Riemann surface such that $\Stab{\CC} \cong \PStab{\CC} \times \C$ as smooth manifolds; see \cref{rem:projective stability space}. Moreover, $\PP( \Hom{\Lambda}{\C} ) \cong \C\PP^1$ with equator $\PP( \Hom{\Lambda}{\R} ) \cong \R\PP^1$ consisting of the charges with rank one image. The charge map descends to a map $\cm \colon \PStab{\CC} \to \C\PP^1$.

The action of $\C$ preserves the set of semistable objects, so the wall-and-chamber structure of the stability space descends to a partition of $\PStab{\CC}$ into open chambers and real codimension one walls between them. 

The action of the universal cover $G$ of $\GL$ also preserves the sets of semistable objects, so each chamber and each wall is a union of orbits. The image of a free $G$-orbit in $\PStab{\CC}$ is a copy of $G/\C\cong \U$ biholomorphic to its image under the charge map. This image is either the Southern or Northern hemisphere in $\C\PP^1$. We refer to these free orbits as \defn{cells}. 

The image of a non-free orbit in $\PStab{\CC}$ is a point. The heart of a stability condition in such an orbit consists of semistable objects of a single phase and so is length. By our assumption it has two isomorphism classes of simple objects, $s$ and $t$. There is a one-parameter family consisting of images of orbits through stability conditions for which $s$ and $t$ are semistable of the same phase and where the ratio $m(s)/m(t)$ varies in $\R_{>0}$. This describes a real analytic curve in $\PStab{\CC}$ which we refer to as a \defn{cell-wall}. Each cell-wall is isomorphic to its image in $\PP\Hom{\Lambda}{\C}$ which is an arc in the equator from the point $Z(s)=0$ to $Z(t)=0$. The following lemma is immediate. 

\begin{lemma} \label{lem:cell-walls-algebraic-hearts}
Suppose $\CC$ admits an algebraic heart with two simple objects. Then there is a bijection between the set of cell-walls and the set of algebraic hearts (up to shift).
\end{lemma}

If there are non-split extensions between the simple objects $s$ and $t$ then a cell-wall is a genuine wall in the stability space along which these extensions are strictly semistable. All walls are of this form; in particular no two walls intersect. If there are no non-split extensions, for instance if the heart is semisimple, then the cell-wall lies in the same chamber as the two neighbouring cells.

\begin{lemma} \label{lem:linear-chains}
Each chamber of $\PStab{\CC}$ is a linear chain of cells, separated by cell-walls. If it is a doubly-infinite chain then it is biholomorphic to $\C$, otherwise it is biholomorphic to $\U$.
\end{lemma}

\begin{proof}
Let $C$ be a cell. If it has no cell-walls in its closure then it is a chamber and we are done. Let $W$ be a cell-wall in the closure, and let $s$ and $t$ be simple objects in a corresponding algebraic heart $\cat{H}$. Their phases agree on $W$ and we may assume $\phi(t) - 1 < \phi(s)<\phi(t)$ in $C$. 

Suppose $W'$ is another cell-wall in the closure of $C$ which is not an actual wall. Let $s'$ and $t'$ be simple objects in a corresponding algebraic heart $\cat{H}'$. Since $s$ and $t$ are semistable on $W'$, some shifts $s[m]$ and $t[n]$ lie in $\cat{H}'$. Indeed, since we assume $\phi(t)-1<\phi(s)<\phi(t)$ in $C$, we may apply a shift so that $s[1], t\in\cat{H}'$. Since $s[1]$ and $t$ are indecomposable and there are no non-split extensions between $s'$ and $t'$ we deduce, after swapping $s'$ and $t'$ if necessary, that $s[1]$ and $t$ are respectively self-extensions of $s'$ and $t'$. However, the classes of $s[1]$ and $t$ are primitive in $\Lambda=K(\CC)$ so $s'=s[1]$ and $t'=t$. It follows that the closure of $C$ in $\PStab{\CC}$ is precisely $W\cup C \cup W'$.

Arguing inductively we conclude that the chamber is either a linear chain of cells as claimed or a cycle. The possibility that it is a cycle is excluded since the phase difference between $s$ and $t$ is well defined and monotonic as we move along a chain of cells. 

For the final part note that the chamber is an open subset of the universal cover $\C$ of $\PP\Hom{\Lambda}{\C} \setminus \{ Z(s)=0,Z(t)=0\}$. If the chain of cells is doubly-infinite then it is the entirety of $\C$. Otherwise it is a simply-connected open strict subset of $\C$ and is therefore biholomorphic to $\U$ by the Riemann Mapping Theorem.
\end{proof}

\subsection{The lax closure and quotient stability space}
\label{subsec:lax and quotient partial compactifications}

As before, we remove the boundary strata of $\StabL{\CC}$ where all objects are massless and denote the resulting space by $\StabL{\CC}^*$. Therefore $\rank(\Lambda_\NN)= 1$ when the massless subcategory $\NN\neq 0$. The action of $\C$ on $\StabL{\CC}^*$ is free and the quotient $\PStabL{\CC} \coloneqq \StabL{\CC}^*/\C$ is a union of the open subset $\PStab{\CC}$ and boundary components $\StabLN{\CC}{\NN}/\C$ homeomorphic to an open subset of $\R$ for each massless subcategory $\NN$. Passing to the quotient $\PStabQ{\CC} \coloneqq \StabQ{\CC}^*/\C$ the boundary component $\StabLN{\CC}{\NN}/\C$ is collapsed to a point, corresponding to the unique point of $\PStab{\CC/\NN}$. The charge projection maps this to $\Hom{\Lambda/\Lambda_\NN}{\C} \in \PP( \Hom{\Lambda}{\C} ) \cong \C\PP^1$.

The decomposition of $\PStab{\CC}$ into cells and cell-walls is locally finite so that $\PStab{\CC}$ is homeomorphic to the union of the closures of its cells. The same applies to $\PStabL{\CC}$, and therefore also to its quotient $\PStabQ{\CC}$. The closures of the cells were described in the previous section. This leads to the following local description of a neighbourhood of $\StabLN{\CC}{\NN}/\C$ in $\PStabL{\CC}$. 

Let $\Delta_\NN$ be a disk about $\Hom{\Lambda/\Lambda_\NN}{\C}$ in $\PP( \Hom{\Lambda}{\C} )$. Let $\beta \colon \mathrm{Bl}(\Delta_\NN) \to \Delta_\NN$ be the real oriented blowup at the centre $\Hom{\Lambda/\Lambda_\NN}{\C}$ and $\gamma \colon \widetilde{\mathrm{Bl}}(\Delta_\NN)\to \mathrm{Bl}(\Delta_\NN) $ its universal cover:

\begin{center}
\begin{tikzpicture}
\draw[red, dotted, fill=red!10] (0,0) circle (1cm);
\draw[thick, dashed, blue]  (0,0) -- (1,0);
\fill[red] (0,0) circle (0.05cm);
\draw[red, dotted, fill=red!10] (4,0) circle (1cm);
\draw[thick, dashed, blue]  (4,0) -- (5,0);
\draw[red, thick,  fill=white] (4,0) circle (0.3cm);
\fill[red!10] (7,-0.5) -- (11,-0.5) --(11,0.5) -- (7,0.5) -- cycle;
\draw[thick, red] (7,-0.5) -- (11,-0.5);
\draw[dotted, red] (7,0.5) -- (11,0.5);
\draw[thick, dashed, blue] (8,-0.5) -- (8,0.5);
\draw[thick, dashed, blue] (9,-0.5) -- (9,0.5);
\draw[thick, dashed, blue] (10,-0.5) -- (10,0.5);
\draw[->] (2.5,0) -- (1.5,0);
\draw[->] (6.5,0)--(5.5,0);
\node at (2,0.25) {$\beta$};
\node at (6,0.25) {$\gamma$};
\node at (0,-1.5) {$\Delta_\NN$};
\node at (4,-1.5) {$\mathrm{Bl}(\Delta_\NN)$};
\node at (9,-1.5) {$\widetilde{\mathrm{Bl}}(\Delta_\NN)$};
\end{tikzpicture}
\end{center}
The homeomorphisms of \cref{prop:orbit closure bijection} identify an open neighbourhood of $\StabLN{\CC}{\NN}/\C$ with an open subset of $\widetilde{\mathrm{Bl}}(\Delta_\NN)$; the stratum $\StabLN{\CC}{\NN}/\C$ is identified with an open subset of the universal cover of the exceptional divisor, \ie with an open subset of $\R$ recording the allowed massless phases. Note that $\PStabL{\CC}$ is \emph{not} a Riemann surface with boundary because each boundary stratum has a holomorphic function which vanishes along it, namely the charge of (any) object which becomes massless on that stratum. The local model for $\PStabQ{\CC}$ is the quotient space obtained by collapsing the boundary component $\StabLN{\CC}{\NN}/\C$ to a point.

Next we classify the stable massless objects. For simple $s$ in a heart $\HH$ set $s\orth = \{ h\in \HH \mid \Hom{s}{h}=0\}$. When $\HH$ is algebraic, $\HH=\clext{s}*s\orth$ and the tilt $s\orth[1]*\clext{s}$ is another heart\footnote{This heart is the left tilt at the torsion pair $(\clext{s}, s\orth)$ in $\HH$, which coincides with the shift of the right simple tilt at $s$, so the condition in \cref{prop:massless stables = algebraic simples} requires the right simple tilt at $s$ to be algebraic.}.

\begin{proposition}
\label{prop:massless stables = algebraic simples}
For an object $s\in \CC$, the following statements are equivalent:
\begin{enumerate}[label=(\roman*)]
\item $s$ is massless and stable at some point in $\StabL{\CC}^*$;
\item $s$ is simple in an algebraic heart such that the tilt $s\orth[1]*\clext{s}$ is also algebraic. 
\end{enumerate}
\end{proposition}

\begin{proof}
$(i) \implies (ii):$
Suppose $s$ is a massless stable object for some lax stability condition in $\StabL{\CC}^*$. Since $\PStabL{\CC}$ is the union of the closures of its cells, this lax stability condition is in the closure of a free orbit $\sigma\cdot G$ for some $\sigma=(P,Z)\in \Stab{\CC}$. By \cref{prop:lax closure,prop:orbit closure bijection} it has the form $\sigma_w^\psi$ for some $w=e^{2\pi i \phi}$ and $\psi\in[0,1]$. Therefore $s\in P(\phi)$ where $P(\phi-\epsilon,\phi)=0=P(\phi,\phi+\epsilon)$ for some $\epsilon >0$. It follows that $s$ is simple in $P(\phi-1,\phi]$, which is an algebraic heart because $P(\phi,\phi+\epsilon)=0$, see, for example, \cite[Thm.~2.7]{BCPW24}. By assumption this heart has only two simple objects. Since the other is massive it is not in $P(\phi)$ because the orbit is free so the simple objects have distinct phases. We conclude that $P(\phi)=\clext{s}$. Therefore $s\orth = P(\phi-1,\phi)$ so that $s\orth[1]*\clext{s} = P[\phi,\phi+1)$. This is also algebraic because $P(\phi-\epsilon, \phi)=0$.

$(ii) \implies (i):$
Suppose $s$ is simple in an algebraic heart $\HH$ such that $s\orth[1]*\clext{s}$ is algebraic. Denote by $t \in \HH$ the other simple object. The charge $Z(s)=-1$ and $Z(t)=i$ determines a stability condition $\sigma=(P,Z)$ with heart $P(0,1]=\HH$ and $P(1)=\clext{s}$. By assumption $s\orth[1]*\clext{s}=P[1,2)$ is algebraic so that $P(1-\epsilon,1)=0$ for some $\epsilon>0$. 
Similarly, $P(1,1+\epsilon) = 0$ for some $\epsilon > 0 $ because $\HH$ is algebraic; in fact, we have $P(1, 3/2) = 0$.
Then, by \cref{prop:lax closure}, $s$ is massless and stable in the lax stability condition $\sigma_w^\psi$ for $w=e^{2\pi i \phi}$ and any $0< \psi<1$.
\end{proof}

\begin{corollary}
\label{cor:rank 2 boundary points}
The boundary points of $\PStabQ{\CC}$ are in bijection with iso-classes up to shift of simple objects $s$ in algebraic hearts for which $s\orth[1]*\clext{s}$ is also algebraic. The corresponding boundary point has massless category $\NN=\thick{\CC}{s}$ and lies over $\Hom{\Lambda/\Lambda_\NN}{\C}$ in $\PP( \Hom{\Lambda}{\C} )$.
\end{corollary}

\begin{proof}
The boundary points of $\PStabQ{\CC}$ have the form $\PStabQ{\CC,\NN}$ for the possible massless categories $\NN$. These massless categories have the form $\NN = \thick{\CC}{s}$ for a massless stable $s$. The result follows from \cref{prop:massless stables = algebraic simples}.
\end{proof} 

\subsection{The exchange graph} 
The \defn{exchange graph} $\EG{\CC}$ has one vertex for each algebraic heart of a stability condition in $\Stab{\CC}$ and an edge whenever two hearts are related by a simple HRS tilt; see \cite{HRS96,Woolf10}. Each vertex of $\EG{\CC}$ is at most $4$-valent because we can tilt left or right at each of the two simple objects of the corresponding heart, but in some cases the tilted hearts may not themselves be algebraic, see the example in \cref{subsec:nilrep}.

The shift acts freely on $\EG{\CC}$ and the quotient $\EG{\CC}/\Z$ has one vertex for each algebraic heart up to shift, \ie one vertex for each cell-wall. This quotient graph can be embedded into $\PStabQ{\CC}$ by placing a vertex on each cell-wall and embedding edges as smooth curves in the unique cell containing the two cell-walls corresponding to its vertices in its closure.

\subsection{Dense phase case}
\label{subsec:dense phase case}

If $\Stab{\CC}$ contains a stability condition $\sigma$ whose phase diagram $\Phi_\sigma$ is dense in $\R/\Z$ then the $G$-orbit of $\sigma$ is free and by \cref{cor:orbit closed} closed. Since the orbit is also open it is an entire connected component, consisting of a single chamber with no walls. Every stability condition in this component has dense phases and therefore no stability condition in this component has an algebraic heart. Our theory adds no boundary points to this component.

This situation occurs for the space $\Stab{X}$ of numerical stability conditions on a smooth complex projective curve $X$ of genus $g>0$, see \cite{Bridgeland07} for elliptic curves and  \cite{Macri07} for the higher genus cases. It also occurs for the space $\Stab{Q}$ of stability conditions on the bounded derived category of finite-dimensional representations of a $2$-vertex quiver $Q$ with an oriented cycle, see \cite[Rem.~3.33]{DHKK14} for the existence of a dense-phase stability condition.

\subsection{Non-dense phase case}
\label{subsec:non-dense phase case}

Now suppose that there is at least one $\sigma$ in $\Stab{\CC}$ with non-dense phases. By the above, every stability condition in the component of $\sigma$ has non-dense phases, and therefore lies in the $\C$-orbit of a stability condition with length heart. Recall that we assume $\Lambda = K(\CC) \cong \Z^2$ so that any algebraic heart has two iso-classes of simple objects.

Using the description of the closure of the orbit $(\sigma\cdot G)/\C$ in \cref{prop:orbit closure bijection}, we can describe the component of $\sigma$ in $\PStabL{\CC}$ and $\PStabQ{\CC}$. First we determine the chamber containing $\sigma$ and its walls. The latter correspond to non-trivial algebraic hearts whose simple objects are stable in the chamber. Then we find the stable objects in the adjacent chambers and repeat. In our examples there are only finitely many chambers up to the action of $\Aaut{\Lambda}{\CC}$, so this is an effective strategy. 

\subsection{The $A_2$ quiver}
\label{subsec:A2}

Let $Q$ be the $A_2$ quiver with two vertices and one arrow. In this section we consider two stability spaces associated with algebras associated with the $A_2$ quiver, namely the $A_2$ path algebra, and the associated $2$-Calabi--Yau Ginzburg algebra.

\subsubsection{Classic $A_2$}

First we consider the bounded derived category $\Db(A_2)$ of finite-dimensional representations of the classic $A_2$. Its stability space $\Stab{A_2}\cong \C^2$ was first described by Alastair King \cite{elephant}, see \cite{BQS20} for further references.

The standard heart of $\Db(A_2)$ has two exceptional simple objects $s$ and $t$ with one non-split extension $0\to s\to e\to t \to 0$ between them. It is easy to construct a stability condition in which $s$, $e$ and $t$ are the only stable objects up to shift. Its phase diagram has three isolated phases. By \cref{lem:non-lax support} the cell containing this stability condition has three cell-walls. The object $e$ destabilises as we cross the cell-wall along which $s$ and $t$ have the same phase, and we enter a chamber in which only $s$ and $t$ are stable. Since $e$ is the unique indecomposable extension between any shifts of $s$ and $t$, this chamber is a chain of cells indexed by $\N$. Similar considerations apply to the other two walls of the initial cell. Thus $\PStab{A_2}$ has four chambers, one in which $s$, $e$ and $t$ are all stable and three in which pairs of them are stable.

Since each heart in $\Db(A_2)$ has finite representation type \cref{prop:lax closure} implies that the closure of each cell in $\PStabL{A_2}$ is obtained by adding a closed arc at each boundary point where a stable object's mass vanishes. The unions of these arcs form three boundary strata, each diffeomorphic to $\R$, where the objects $s$, $t$ and $e$ respectively are massless. The space $\PStabQ{A_2}$ is obtained by collapsing each of these to a point. The set of boundary strata agrees with the descriptions in \cref{cor:finite type massless} and \cref{prop:massless stables = algebraic simples}.

The category $\Db(A_2)$ is fractional Calabi--Yau; the Serre functor $S$ satisfies $S^3=[1]$. This acts by rotation, preserving the central chamber and cyclically permuting the other three chambers and the three boundary strata. See \cref{fig:A2} for an illustration.

\subsubsection{$2$-Calabi--Yau $A_2$}
\label{subsec:GA2}
Now consider the $2$-Calabi--Yau category $\Db(\Ginzburg A_2)$, where $\Ginzburg A_2$ is the Ginzburg dg algebra of the $A_2$ quiver;
see \eg \cite[\S 7.2]{Keller12} for details.

The stability space $\PStab{\Ginzburg A_2} \cong \pd$ is the universal cover of the thrice punctured Riemann sphere and was first described in \cite{Thomas06}. See \cite{BQS20} for a detailed discussion and further references, and also \cite{KQ15,QW18} for more general discussions of the stability spaces of the Ginzburg dg-algebras associated to Dynkin quivers.

The standard heart of $\Db(\Ginzburg A_2)$ has two $2$-spherical simple objects $s$ and $t$, with one non-split extension $0\to s \to e\to t \to 0$ and $0\to t\to f \to s \to 0$ in each direction. Each $2$-spherical object in $\Db(\Ginzburg A_2)$ generates a twist automorphism. For example, applying the twist $\twist_{s}$ about $s$ to the triangle $s\to e\to t \to s[1]$ yields the triangle $s[-1] \to t \to f \to s$, and then applying $\twist_{t}$ yields the rotation $e[-1] \to t[-1] \to s \to e$ of the original triangle. In particular $e$ and $f$ are also $2$-spherical. The subgroup of $\Aut{\Db(\Ginzburg A_2)}$ generated by $\twist_{s}$ and $\twist_{t}$ is isomorphic to the Artin--Tits braid group $B_3$ of the $A_2$ quiver, \ie the braid group on three strands. The centre $Z(B_3)$ is generated by a single automorphism which acts as the Serre functor $S = [2]$. Let $\mathcal{S}$ be the set of equivalence classes of spherical objects in $\Db(\Ginzburg A_2)$ up to isomorphism and shift. We abuse notation by using the same notation for spherical objects and their classes in $\mathcal{S}$; this is harmless since the twists $\twist_{u}=\twist_{u[1]}$ agree for any spherical object $u$. The quotient $B_3/Z(B_3)\cong \PSL_2(\Z)$ acts on $\mathcal{S}$ and the stabiliser of $u$ is the infinite cyclic subgroup generated by $\twist_{u}$. From the above examples $t$ is in the orbit of $s$ (indeed the action on $\mathcal{S}$ is transitive although we do not need this). 

Now consider the stability space $\Stab{\Ginzburg A_2}$. As in the $\Db(A_2)$ case one can easily construct a stability condition in which, up to shift, the stable objects are the two simple objects $s$ and $t$ of the standard heart and one, $e$ say, of the two extensions between them. The phase diagram has three isolated phases so the corresponding cell in $\PStab{\Ginzburg A_2}$ has three cell-walls. As before $e$ destabilises as we cross the wall where $s$ and $t$ have the same phase, but now the other extension $f$ becomes stable on the far side of the wall. Thus we enter the chamber obtained by applying $\twist_{s}$ to the initial one. Similar considerations apply to the other walls of the initial chamber. Therefore $\PSL_2(\Z)$ acts transitively on the chambers in $\PStab{\Ginzburg A_2}$, each of which is a single cell bounded by three walls. There are three stable $2$-spherical objects in each chamber whose respective masses vanish at the three boundary points of the chamber. The action of $\PSL_2(\Z)$ on chambers is free because no pair of distinct spherical objects, {\it a fortiori} no triple, is fixed. The action on walls is also free and it quickly follows from the examples of twist actions that it is transitive. 

The heart of each stability condition is algebraic, see \eg \cite{QW18}. Therefore by \cref{cor:finite type massless,,prop:massless stables = algebraic simples}, the only massless stable objects are the simple objects of the hearts of stability conditions and there is one boundary stratum in $\PStabL{\Ginzburg A_2}$ for each (up to shift and isomorphism). Moreover, since each heart has finite representation type, \cref{prop:lax closure} implies that there is no restriction on the massless phase so that each boundary stratum is a copy of $\R$. The space $\PStabQ{\Ginzburg A_2}$ is obtained by collapsing each of these strata to a point.

The twist $\twist_{s}$ acts on $\PStabQ{\Ginzburg A_2}$ by a hyperbolic isometry fixing the boundary point at which $s$ is massless. Therefore, $\twist_{s}$ acts either by an ideal rotation about that point or a translation, since, locally at the fixed point the action universally covers the action on $\PP\Hom{\Lambda}{\C}$. This is given by the matrix
$
\begin{psmallmatrix} -1 & 0 \\ \phantom{-}1 & 1 \end{psmallmatrix}
$
with respect to the basis $\{ [s], [t]\}$ of $\Lambda=K(\Db(\Ginzburg A_2))$. Since the eigenvalues are $\pm 1$ the twist acts by an ideal rotation. Up to isometry, the three boundary points where $s$, $e$ and $t$ are massless can be chosen arbitrarily on $\partial \pd$ and this fixes the remaining boundary points of $\PStabQ{{\Ginzburg A_2}}$ uniquely. They form a dense subset of $\partial \pd$. See \cref{fig:A2} for an illustration.

\begin{figure}[ht]
\begin{center}
\begin{tabular}{cc}
\begin{tikzpicture}
%draw Poincare disk
\draw[thick, dashed, fill=red!10] (0,0) circle (\ra);

\begin{scope}
%clip
\clip (0,0) circle (\ra);
%fill chambers
\foreach \theta in {0,120,240}  \hgfill{\theta}{\theta+120}{red!20};

%draw walls
\foreach \theta in {0,120,240} \hgline{red}{\theta}{\theta+120};

%draw boundary components
\foreach \theta in {0,120,240} \draw[thick, fill=white] (\theta:\ra)  arc (\theta:\theta+360:0.15*\ra);
\end{scope}

%mark and label boundary points
\foreach \theta in {0,120,240} \fill[white] (\theta:\ra)  circle (1pt);
\node at (0:1.15*\ra) {$e$}; 
\node at (120:1.15*\ra) {$s$}; 
\node at (240:1.15*\ra) {$t$}; 
\end{tikzpicture} & \begin{tikzpicture}
%draw Poincare disk
\draw[thick, dashed, fill=red!10] (0,0) circle (\ra);

\begin{scope}
%clip
\clip (0,0) circle (\ra);
%fill chambers
\foreach \theta in {0,120,240} \hgfill{\theta}{\theta+120}{red!20};
%draw walls
\foreach \theta in {0,120,240} \hgline{red}{\theta}{\theta+120};
\end{scope}

%mark and label boundary points
\foreach \theta in {0,120,240} \fill (\theta:\ra)  circle (2pt);
\node at (0:\ra+0.35) {$e$}; 
\node at (120:\ra+0.35) {$s$}; 
\node at (240:\ra+0.35) {$t$}; 
\end{tikzpicture} \\
$\PStabL{A_2}$ & $\PStabQ{A_2}$ \\
&\\
\begin{tikzpicture}
%draw Poincare disk
\draw[thick, dashed, fill=blue!05] (0,0) circle (\ra);

\begin{scope}
%clip
\clip (0,0) circle (\ra);
%fill regions
\foreach \theta in {0,120,240} \hgfill{\theta}{\theta+120}{blue!10};
\foreach \theta in {0,60,...,300} \hgfill{\theta}{\theta+60}{blue!15};
\foreach \theta in {0,30,...,330} \hgfill{\theta}{\theta+30}{blue!20};
\foreach \theta in {0,15,...,345} \hgfill{\theta}{\theta+15}{blue!30};

%draw walls

\foreach \theta in {0,120,240} \hgline{blue}{\theta}{\theta+120};
\foreach \theta in {0,60,...,300} \hgline{blue}{\theta}{\theta+60};
\foreach \theta in {0,30,...,330} \hgline{blue}{\theta}{\theta+30};
\foreach \theta in {0,15,...,345} \hgline{blue}{\theta}{\theta+15};

% draw boundary components

\foreach \theta in {0,120,240} \draw[thick, fill=white] (\theta:\ra)  arc (\theta:\theta+360:0.15*\ra);
\foreach \theta in {60, 180, 300} \draw[thick, fill=white] (\theta:\ra)  arc (\theta:\theta+360:0.1*\ra);
\foreach \theta in {30,90,...,330} \draw[thick, fill=white] (\theta:\ra)  arc (\theta:\theta+360:0.075*\ra);
\foreach \theta in {15,45,...,345}  \draw[thick, fill=white] (\theta:\ra)  arc (\theta:\theta+360:0.03*\ra);
\end{scope}

%mark and label boundary points
\foreach \theta in {0,15,...,345} \fill[white] (\theta:\ra)  circle (1pt);
\node at (0:1.15*\ra) {$e$}; 
\node at (120:1.15*\ra) {$s$}; 
\node at (240:1.15*\ra) {$t$}; 
\end{tikzpicture} & \begin{tikzpicture}
%draw Poincare disk
\draw[thick, dashed, fill=blue!05] (0,0) circle (\ra);

\begin{scope}
%clip
\clip (0,0) circle (\ra);
%fill regions
\foreach \theta in {0,120,240} \hgfill{\theta}{\theta+120}{blue!10};
\foreach \theta in {0,60,...,300} \hgfill{\theta}{\theta+60}{blue!15};
\foreach \theta in {0,30,...,330} \hgfill{\theta}{\theta+30}{blue!20};
\foreach \theta in {0,15,...,345} \hgfill{\theta}{\theta+15}{blue!30};

%draw walls

\foreach \theta in {0,120,240} \hgline{blue}{\theta}{\theta+120};
\foreach \theta in {0,60,...,300} \hgline{blue}{\theta}{\theta+60};
\foreach \theta in {0,30,...,330} \hgline{blue}{\theta}{\theta+30};
\foreach \theta in {0,15,...,345} \hgline{blue}{\theta}{\theta+15};

\end{scope}

%mark and label boundary points
\foreach \theta in {0,15,...,345} \fill (\theta:\ra)  circle (2pt);
\node at (0:\ra+0.35) {$e$}; 
\node at (120:\ra+0.35) {$s$}; 
\node at (240:\ra+0.35) {$t$}; 
\end{tikzpicture} \\
$\PStabL{\Gamma_2A_2}$ & $\PStabQ{\Gamma_2A_2}$ \\
\end{tabular}
\end{center}
\caption{Illustrations of $\PStabL{\CC}$ and $\PStabQ{\CC}$ for $\CC$ the derived category of the $A_2$ quiver, and the associated Ginzburg $2$-Calabi--Yau category. In each case $\PStabL{\CC}$ is a real surface with boundary (black lines) whose interior $\Stab{\CC}$ (shaded) is a contractible non-compact Riemann surface biholomorphic to $\C$ (pink) in the former case and to the Poincar\'e disk (blue) in the latter. This is divided by walls (respectively red or blue lines) into chambers (indicated by shading) in which the set of stable objects is constant. The walls correspond to non-semisimple algebraic hearts in $\CC$ up to shift. The phases of the two simple objects of the heart coincide on the wall and the stability of extensions between them changes on crossing it. The endpoints of the wall lie on the boundary strata where the two simple objects become massless (indicated by labels). The stratum consists of the allowed phases for that massless object increasing in the clockwise direction; in these two examples there is no restriction on the massless phase. An object is stable in a chamber if the closure of the chamber intersects the boundary stratum where that object is massless. The spaces $\PStabQ{\CC}$ on the right are obtained by collapsing each boundary stratum to a point. See \cref{subsec:A2} for details.}

\label{fig:A2}
\end{figure}

\begin{figure}
\begin{center}
\begin{tabular}{cc}
\begin{tikzpicture}
%draw Poincare disk
\draw[thick, dashed, fill=red!5] (0,0) circle (\ra);

\begin{scope}
%clip
\clip (0,0) circle (\ra);
%fill chambers
\hgfill{90}{230}{red!10};
\hgfill{90}{170}{red!15};
\hgfill{90}{130}{red!20};
\hgfill{90}{110}{red!25};

\hgfill{90}{-50}{red!10};
\hgfill{90}{10}{red!15};
\hgfill{90}{50}{red!20};
\hgfill{90}{70}{red!25};

\hgfill{95}{100}{red!30};
\hgfill{100}{110}{red!30};
\hgfill{110}{130}{red!30};
\hgfill{130}{170}{red!30};
\hgfill{170}{230}{red!30};

\hgfill{85}{80}{red!30};
\hgfill{80}{70}{red!30};
\hgfill{70}{50}{red!30};
\hgfill{50}{10}{red!30};
\hgfill{10}{-50}{red!30};

\hgfill{230}{310}{red!30};

%draw walls
\hgline{red}{95}{100};
\hgline{red}{100}{110};
\hgline{red}{110}{130};
\hgline{red}{130}{170};
\hgline{red}{170}{230};

\hgline{red}{85}{80};
\hgline{red}{80}{70};
\hgline{red}{70}{50};
\hgline{red}{50}{10};
\hgline{red}{10}{-50};

\hgline{red}{230}{310};

\hgline{red}{90}{95};
\hgline{red}{90}{100};
\hgline{red}{90}{110};
\hgline{red}{90}{130};
\hgline{red}{90}{170};
\hgline{red}{90}{230};
\hgline{red}{90}{-50};
\hgline{red}{90}{10};
\hgline{red}{90}{50};
\hgline{red}{90}{70};
\hgline{red}{90}{80};
\hgline{red}{90}{85};

%draw bpundary components
\foreach \theta in {230,310}  \draw[thick, fill=white] (\theta:\ra)  arc (\theta:\theta+360:0.175*\ra);
\foreach \theta in {170,10}  \draw[thick, fill=white] (\theta:\ra)  arc (\theta:\theta+360:0.15*\ra);
\foreach \theta in {130,50}  \draw[thick, fill=white] (\theta:\ra)  arc (\theta:\theta+360:0.1*\ra);
\foreach \theta in {110,70}  \draw[thick, fill=white] (\theta:\ra)  arc (\theta:\theta+360:0.075*\ra);
\foreach \theta in {100,80}  \draw[thick, fill=white] (\theta:\ra)  arc (\theta:\theta+360:0.05*\ra);
\draw[thick, fill=white] (90:\ra)  arc (90:450:0.025*\ra);

\end{scope}

%mark and label boundary points
\foreach \theta in {90,100, 110, 130,170,230,310, 10,50,70,80} \fill[white] (\theta:\ra)  circle (1pt);
\node at (90:1.15*\ra) {$s$}; 
\node at (230:1.15*\ra) {$t_0$}; 
\node at (310:1.15*\ra) {$t_1$}; 
\end{tikzpicture} & \begin{tikzpicture}
%draw Poincare disk
\draw[thick, dashed, fill=red!5] (0,0) circle (\ra);

\begin{scope}
%clip
\clip (0,0) circle (\ra);
%fill chambers
\hgfill{90}{230}{red!10};
\hgfill{90}{170}{red!15};
\hgfill{90}{130}{red!20};
\hgfill{90}{110}{red!25};

\hgfill{90}{-50}{red!10};
\hgfill{90}{10}{red!15};
\hgfill{90}{50}{red!20};
\hgfill{90}{70}{red!25};

\hgfill{95}{100}{red!30};
\hgfill{100}{110}{red!30};
\hgfill{110}{130}{red!30};
\hgfill{130}{170}{red!30};
\hgfill{170}{230}{red!30};

\hgfill{85}{80}{red!30};
\hgfill{80}{70}{red!30};
\hgfill{70}{50}{red!30};
\hgfill{50}{10}{red!30};
\hgfill{10}{-50}{red!30};

\hgfill{230}{310}{red!30};

%draw walls
\hgline{red}{95}{100};
\hgline{red}{100}{110};
\hgline{red}{110}{130};
\hgline{red}{130}{170};
\hgline{red}{170}{230};

\hgline{red}{85}{80};
\hgline{red}{80}{70};
\hgline{red}{70}{50};
\hgline{red}{50}{10};
\hgline{red}{10}{-50};

\hgline{red}{230}{310};

\hgline{red}{90}{95};
\hgline{red}{90}{100};
\hgline{red}{90}{110};
\hgline{red}{90}{130};
\hgline{red}{90}{170};
\hgline{red}{90}{230};
\hgline{red}{90}{-50};
\hgline{red}{90}{10};
\hgline{red}{90}{50};
\hgline{red}{90}{70};
\hgline{red}{90}{80};
\hgline{red}{90}{85};

\end{scope}

%mark and label boundary points
\foreach \theta in {90, 95,100,110,130,170,230,310,10,50,70,80,85} \fill (\theta:\ra) circle (2pt);
\node at (90:1.15*\ra) {$s$}; 
\node at (230:1.15*\ra) {$t_0$}; 
\node at (310:1.15*\ra) {$t_1$}; 
\end{tikzpicture} \\
$\PStabL{\Lambda_{2,1,0}}$ & $\PStabQ{\Lambda_{2,1,0}}$ \\
&\\
\begin{tikzpicture}
%draw Poincare disk
\draw[thick, dashed, fill=red!10] (0,0) circle (\ra);

\begin{scope}
%clip{red!20}
\clip (0,0) circle (\ra);
%fill chambers
\hgfill{95}{100}{red!20};
\hgfill{100}{110}{red!20};
\hgfill{110}{130}{red!20};
\hgfill{130}{170}{red!20};
\hgfill{170}{230}{red!20};

\hgfill{85}{80}{red!20};
\hgfill{80}{70}{red!20};
\hgfill{70}{50}{red!20};
\hgfill{50}{10}{red!20};
\hgfill{10}{-50}{red!20};

\hgfill{230}{310}{red!20};

%draw walls
\hgline{red}{95}{100};
\hgline{red}{100}{110};
\hgline{red}{110}{130};
\hgline{red}{130}{170};
\hgline{red}{170}{230};

\hgline{red}{85}{80};
\hgline{red}{80}{70};
\hgline{red}{70}{50};
\hgline{red}{50}{10};
\hgline{red}{10}{-50};

\hgline{red}{230}{310};

%draw bpundary components
\foreach \theta in {230,310}  \draw[thick, fill=white] (\theta:\ra)  arc (\theta:\theta+360:0.185*\ra);
\foreach \theta in {170,10}  \draw[thick, fill=white] (\theta:\ra)  arc (\theta:\theta+360:0.15*\ra);
\foreach \theta in {130,50}  \draw[thick, fill=white] (\theta:\ra)  arc (\theta:\theta+360:0.1*\ra);
\foreach \theta in {110,70}  \draw[thick, fill=white] (\theta:\ra)  arc (\theta:\theta+360:0.075*\ra);
\foreach \theta in {100,80}  \draw[thick, fill=white] (\theta:\ra)  arc (\theta:\theta+360:0.05*\ra);
\end{scope}

%mark and label boundary points
\foreach \theta in {100, 110, 130,170,230,310, 10,50,70,80} \fill[white] (\theta:\ra)  circle (1pt);
\draw[dotted, fill=white] (90:\ra)  circle (0.05*\ra);
\node at (90:1.15*\ra) {$\mathcal{O}_x$}; 
\node at (230:1.2*\ra) {$\mathcal{O}$}; 
\node at (310:1.2*\ra) {$\mathcal{O}(1)$}; 
\end{tikzpicture} & \begin{tikzpicture}
%draw Poincare disk
\draw[thick, dashed, fill=red!10] (0,0) circle (\ra);

\begin{scope}
%clip
\clip (0,0) circle (\ra);
%fill chambers
\hgfill{95}{100}{red!20};
\hgfill{100}{110}{red!20};
\hgfill{110}{130}{red!20};
\hgfill{130}{170}{red!20};
\hgfill{170}{230}{red!20};

\hgfill{85}{80}{red!20};
\hgfill{80}{70}{red!20};
\hgfill{70}{50}{red!20};
\hgfill{50}{10}{red!20};
\hgfill{10}{-50}{red!20};

\hgfill{230}{310}{red!20};

%draw walls
\hgline{red}{95}{100};
\hgline{red}{100}{110};
\hgline{red}{110}{130};
\hgline{red}{130}{170};
\hgline{red}{170}{230};

\hgline{red}{85}{80};
\hgline{red}{80}{70};
\hgline{red}{70}{50};
\hgline{red}{50}{10};
\hgline{red}{10}{-50};

\hgline{red}{230}{310};

\end{scope}

%mark and label boundary points
\draw[dotted, fill=white] (90:\ra)  circle (0.05*\ra);
\foreach \theta in {95,100,110,130,170,230,310,10,50,70,80,85} \fill (\theta:\ra) circle (2pt);
\node at (90:1.15*\ra) {$\mathcal{O}_x$}; 
\node at (230:1.175*\ra) {$\mathcal{O}$}; 
\node at (310:1.175*\ra) {$\mathcal{O}(1)$}; 
\end{tikzpicture} \\
$\PStabL{\mathbb{P}^1}$ & $\PStabQ{\mathbb{P}^1}$ \\
&\\
\begin{tikzpicture}
%draw Poincare disk
\draw[thick, dashed, fill=red!10] (0,0) circle (\ra);

\begin{scope}
%clip
\clip (0,0) circle (\ra);
%fill chambers
\foreach \theta in {0,240}  \hgfill{\theta}{\theta+120}{red!20};

%draw walls
\foreach \theta in {0,240} \hgline{red}{\theta}{\theta+120};

%draw boundary components
\foreach \theta in {120,240} \draw[thick, dashed, fill=white] (\theta:\ra)  arc (\theta:\theta+360:0.15*\ra);
\draw[thick, fill=white] (0:\ra)  arc (0:3600:0.15*\ra);
\draw[thick] (120:\ra)  arc (120:-50:0.15*\ra);
\draw[thick] (240:\ra)  arc (240:410:0.15*\ra);
\end{scope}

%mark and label boundary points
\foreach \theta in {0,120,240} \fill[white] (\theta:\ra)  circle (1pt);
\node at (0:1.15*\ra) {$t$}; 
\node at (120:1.15*\ra) {$s$}; 
\node at (240:1.15*\ra) {$e$}; 
\end{tikzpicture} &\begin{tikzpicture}
%draw Poincare disk
\draw[thick, dashed, fill=red!10] (0,0) circle (\ra);

\begin{scope}
%clip
\clip (0,0) circle (\ra);

%fill chambers
\foreach \theta in {0,240}  \hgfill{\theta}{\theta+120}{red!20};

%draw walls
\foreach \theta in {0,240} \hgline{red}{\theta}{\theta+119};

%separate boundary points not in limit of walls
\foreach \theta in {120,240} \fill[red!10] (\theta:\ra)  circle (5pt);

\end{scope}

%mark and label boundary points
\foreach \theta in {0,120,240} \fill (\theta:\ra)  circle (2pt);
\node at (0:1.15*\ra) {$t$}; 
\node at (120:1.15*\ra) {$s$}; 
\node at (240:1.15*\ra) {$e$}; 
\end{tikzpicture} \\
$\PStabL{Q}$ & $\PStabQ{Q}$
\end{tabular}
\end{center}
\caption{Illustrations of $\PStabL{\CC}$ and $\PStabQ{\CC}$ for $\CC$ respectively the derived category of representations of the quiver $\Lambda_{2,1,0}$, coherent sheaves on $\mathbb{P}^1$ (equivalently the Kronecker quiver), and nilpotent representations of the quiver $Q$ with two vertices with a loop at one vertex and an arrow from the vertex with a loop to the other. See \cref{fig:A2} for explanation. In the second case the skyscrapers $\mathcal{O}_x$ are stable in the central chamber but never massless. In the third case there are restrictions on the phases of massless objects: $\phi(s)< \phi(t)$ when $s$ is massless and $\phi(e)>\phi(t)+1$ when $e$ is massless. See \cref{subsec:discrete,,subsec:projective line,,subsec:nilrep} for details.}
\label{fig:projective line}
\label{fig:lambda}
\label{fig:nilpotent}
\end{figure}

\subsection{A discrete derived category}
\label{subsec:discrete}
Let $\Lambda_{2,1,0}$ be the bound quiver with two vertices, one arrow in each direction, and the zero relation given by the composite of these arrows. The (principal component of the) stability space $\Stab{\Lambda_{2,1,0}}\cong \C^2$ was first described in \cite{Woolf10}, see also \cite{BPP16,QW18} for proofs that the stability space is connected and generalisations to other discrete derived categories. 

Let $s$ be the simple representation at the vertex with no relation, and $t=t_0$ the other simple representation. The object $s$ is $2$-spherical and $t_0$ is exceptional. Since $\twist_{s}(s)=s[-1]$ the twist $\twist_{s}$ generates an infinite cyclic subgroup of automorphisms. Set $t_n =\twist_{s}^n(t_0)$. There are unique non-split extensions $0\to s \to t_{-1} \to t_0 \to 0$ and $0\to t_0 \to t_1 \to s \to 0$. In particular there is a stability condition in which $s$, $t_{-1}$ and $t_0$ are the only stable objects up to shift. This lives in a chamber of $\Stab{\Lambda_{2,1,0}}$ with three walls. Crossing the wall where $t_{-1}$ destabilises we enter a chamber in which $t_1$ is stable. As in the previous example this chamber is the image of the initial one under the action of the twist $\twist_{s}$, and similarly crossing the wall where $t_0$ destabilises we enter the image of the initial chamber under $\twist_{s}^{-1}$. However, if we cross the wall where the spherical object $s$ destabilises then we enter a chamber in which only $t_0$ and $t_1$ are stable. Unlike the previous chambers which consist of a single cell, this is the union of a sequence of cells, and cell-walls upon which the phases of $t_0$ and $t_1[n]$ agree for $n\in \N$. In summary, $\Stab{\Lambda_{2,1,0}}$ has one free orbit of chambers with three stable objects (one spherical and two exceptional) and one free orbit of chambers with two stable objects (both exceptional) under the action generated by $\twist_{s}$. 

Each heart of a stability condition in $\Stab{\Lambda_{2,1,0}}$ is algebraic, see \cite[Lems.~7.4 and 7.5]{BPP17}, so by \cref{prop:massless stables = algebraic simples} the boundary strata correspond to the simple objects of these hearts (up to isomorphism and shift), \ie there is one stratum where each of $s$ and $t_n$ for $n\in \Z$ is massless. It follows from \cref{prop:lax closure} that the massless object of each boundary stratum can take any phase in $\R$. 

The twist $\twist_{s}$ preserves the stratum labelled by $s$ and sends that labelled by $t_n$ to that labelled by $t_{n+1}$. Its square $\twist_{s}^2$ acts trivially on the Grothendieck group, so the strata labelled by the $t_n$ map to one of two points in charge space according to whether $n$ is even or odd. The space $\PStabQ{\Lambda_{2,1,0}}$ is obtained by collapsing each boundary stratum to a point. The twist $\twist_{s}$ acts by an isometry without fixed points on $\PStab{\Lambda_{2,1,0}} \cong \C$ and therefore acts by a translation of $\C$. This is illustrated in \cref{fig:lambda}.

\subsection{The projective line}
\label{subsec:projective line}
The stability space $\PStab{\PP^1}\cong \C$ was first described in \cite{Okada06}, see also \cite{Macri07}. The primordial \emph{geometric} stability condition has heart the coherent sheaves, with stable objects the skyscrapers $\mathcal{O}_x$ for $x\in \PP^1$ and line bundles $\mathcal{O}(k)$ for $k\in \Z$. The chamber containing it has a countable sequence of walls corresponding to the Kronecker hearts. Crossing one of these walls we enter a chamber in which only $\mathcal{O}(k)$ and $\mathcal{O}(k+1)$ are stable, the skyscrapers and all other line bundles destabilise. This chamber is the union of a sequence of cells separated by cell-walls on which the phases of $\mathcal{O}(k+1)$ and $\mathcal{O}(k)[n]$ for $n>0$ agree.

By \cref{prop:massless stables = algebraic simples} the boundary strata of $\PStabL{\mathbb{P}^1}$ correspond to the simple objects (up to isomorphism and shift) in the Kronecker and semisimple hearts as these and their irreducible tilts are algebraic. Thus there is one boundary stratum for each line bundle $\mathcal{O}(k)$. An example lax stability condition of the stratum corresponding to $\mathcal{O}$ is given by the slicing $P_b$ of \cref{ex:slicings} with charge map $Z(\mathcal{O})=0$ and $Z(\mathcal{O}(1)) = i$.

These boundary strata accumulate over the charge annihilating the skyscrapers. However, \cref{prop:massless stables = algebraic simples} shows there is no corresponding boundary stratum in $\PStabL{\PP^1}$ because the skyscrapers are not simple in any algebraic heart and so cannot become massless --- see also \cref{ex:lax supported versus weak}. We indicate this omitted stratum by a dotted circle in \cref{fig:projective line} and label it by $\mathcal{O}_x$ to indicate that the skyscrapers are stable in the adjacent chamber. The space $\PStabQ{\mathbb{P}^1}$ is obtained by collapsing each boundary component to a point.

The infinite cyclic group generated by the automorphism $-\otimes \mathcal{O}(1)$ preserves the chamber containing the geometric stability condition (but does not fix any stability condition in this chamber) and acts freely and transitively on the chambers in which only two objects are stable. It also acts freely and transitively on the boundary points. It follows that it acts by a translation on $\PStab{\PP^1}\cong \C$. See \cref{fig:projective line} for an illustration.
 
Superficially, this closely resembles the previous example. However, there are several important (and inter-related) differences. In this case there is a chamber bounded by a countably infinite family of walls; there are stable objects whose mass does not vanish; the images of the boundary points accumulate in charge space.

\subsection{Nilpotent representations}
\label{subsec:nilrep}

Let $Q$ be the quiver with two vertices connected by an arrow and with a loop:
$\begin{tikzcd}[column sep = small, cramped]
\arrow[loop left] 1 \ar[r] & 2
\end{tikzcd}$.
The category $\nilrep(\kk Q)$ of nilpotent representations of the quiver $Q$ is an algebraic hereditary abelian category with two simple objects $s$ and $t$, where $s$ has a self extension and $t$ is exceptional. Let $\CC = \Db(\nilrep(\kk Q))$. We use $[s], [t]$ as a basis for $\Lambda=K(\CC)$.
For more details, see \cite[p428]{KY14} and \cite[Ex.~2.9]{BCPW24}.

The indecomposable objects of $\nilrep(\kk Q)$ are $M(n_0; n_1, \ldots, n_k)$ with $0 \leq k-1 \leq n_0 \leq k$, $n_1 \geq 0$ and $n_j \geq 1$ for each $j > 1$, in which $n_0$ is the number of occurrences of $t$ in a composition series, $n_1$ is the number of occurrences of $s$ above the first $t$ in the composition series, $n_2$ is the number of occurrences of $s$ above the second $t$ and so on, for example:
\[
  M(0;1) = s, \quad M(1;0) = t,
  \quad
  M(1;2) = \raisebox{-1.7ex}{%
             \begin{tikzpicture}[every node/.style={scale=0.7, font=\sffamily}]
               \node at (0.0,-0.44) {t}; \node at (0,-0.22) {s}; \node at (0,0) {s};
             \end{tikzpicture}},
  \quad
  M(2;2,1) = \raisebox{-2.4ex}{%
             \begin{tikzpicture}[every node/.style={scale=0.7, font=\sffamily}]
              \node at (0,0) {s}; \node at (0,-0.22) {s}; \node at (-0.2,-0.44) {t}; \node at (0.2,-0.44) {s}; \node at (0.2,-0.66) {t};
            \end{tikzpicture}},
\quad
  M(1;2,1) = \raisebox{-2.4ex}{%
             \begin{tikzpicture}[every node/.style={scale=0.7, font=\sffamily}]
              \node at (0,0) {s}; \node at (0,-0.22) {s}; \node at (-0.2,-0.44) {t}; \node at (0.2,-0.44) {s};
            \end{tikzpicture}}%
            .              
\]
Since $\nilrep(\kk Q)$ is algebraic the stability function $Z\colon \Lambda \to \C$ defined by $Z(s) = -1$ and $Z(t) = i$ determines a stability condition. If $M \neq t$ is an indecomposable nilpotent representation then $Z(M)$ lies in the closed region between the ray $e^{3i\pi/4} \R_{\geq 0}$ and the negative real axis. One can check that $M=M(n_0; n_1, \ldots, n_k)$ is semistable if, and only if, either $n_0=0$ or $n_0=k>0$ and
\[
\frac{n_1+\cdots+n_j}{j} \geq \frac{n_1+\cdots +n_k}{k} \qquad \text{for}\ 1\leq j<k;
\]
moreover $M$ is stable if, and only if, $M=s$ or $n_0=k>0$ and all the inequalities above are strict. In particular $M$ is semistable when $n_0=k>0$ and $n_1\geq \cdots\geq n_k$ so that we can choose a semistable $M$ with $\arg Z(M) = \pi \phi$ for any $\phi = p/q \in [3/4,1] \cap \Q$ where $p,q\in\N$.

It follows that the chamber containing this stability condition has two walls, one where $\phi(s)=\phi(t)$ corresponding to the heart $\nilrep(\kk Q)$ and one where $\phi(e)=\phi(t)+1$ corresponding to an algebraic heart with simple objects $e\coloneqq M(1;1)$ and $t[1]$. This latter heart is equivalent to nilpotent representations of the quiver 
$\begin{tikzcd}[column sep = small, cramped]
\arrow[loop left] 1 & 2 \ar[l]
\end{tikzcd}$.
Crossing the first wall one enters a chamber in which only $s$ and $t$ are stable; similarly crossing the second one enters a chamber in which only $e$ and $t$ are stable. Note that $t$ is stable at all points. The central chamber in which there are countably many stable objects is bounded by the line $\phi(s)=\phi(e)$. (Note that the existence of this boundary implies that $\PStab{Q}$ is a hyperbolic Riemann surface.)

The projective lax closure $\PStabL{Q}$ has three boundary strata where respectively $s$, $t$ and $e$ become massless. The $s$-massless stratum is
\[
 \{ (P,Z) \mid s, t \ \text{stable},\ Z(s)=0 \neq Z(t),\ \phi(s)<\phi(t) \}
 \]
and the $e$-massless stratum is $\{ (P,Z)\mid e, t \text{ stable},\  Z(e)=0\neq Z(t),\ \phi(e)>\phi(t)+1 \}$. There are no constraints on the phase of $t$ when it is massless. In particular the closure of each wall meets the $t$-massless stratum, but neither closure meets either of the other two boundary strata.

The projective quotient stability space $\PStabQ{Q}$ also has three boundary strata given by the points at which $s$, $e$ and $t$ become massless. The point where $t$ is massless is in the closure of both walls, those where $s$ and $e$ are massless are in the closure of neither.

As announced in \cref{rmk:lax support} this example illustrates why we must test massive $\delta$-slim objects, rather than only massive stable objects, in the lax support property to obtain an open condition on lax pre-stability conditions with the same massless category. To see why, consider the set of lax pre-stability conditions in which $s$ and $t$ are stable and $s$ is massless. These satisfy lax support when $\phi(s)<\phi(t)$, and therefore also satisfy the weaker condition that there is a support constant for all massive stable objects. When $\phi(s)=\phi(t)$ the slicing has a single non-zero slice $\nilrep(\kk Q)$ and $t$ is the only massive stable object. In particular, there is trivially a support constant for massive stable objects. However, when $\phi(s)>\phi(t)$ 
the only semistable indecomposable objects are $t = M(1;0)$, $M(0;n)$ and $M(n_0;n_1,\ldots,n_{n_0})$. 
In particular, each $M(1;n)$ is massive, stable and of the same phase; the $M(1;n)$ all have the same mass but $\norm{M(1;n)}\to \infty$ as $n\to \infty$ so that there is no support constant. In summary, there is a support constant for massive stable objects if and only if $\phi(s)\leq \phi(t)$ so that this weaker notion of support does not propagate.

\section{Comparisons with other constructions}
\label{sec:comparisons} 

\noindent
We list other extensions of stability spaces occurring in the literature and, if possible, compare with our constructions.

\subsection{Very weak stability conditions of Bayer, Macrì, Stellari and Piyaratne, Toda}
\label{subsec:very weak}

Almost identical generalisations of Bridgeland's stability conditions were introduced in \cite[App.~B]{BMS16} by Arend Bayer, Emanuele Macr\`\i\ and Paolo Stellari and in \cite[\S2.1]{PT} by Dulip Piyaratne and Yukinobu Toda. Following the latter, a \defn{very weak pre-stability condition} on a triangulated category $\CC$ with a homomorphism $v\colon K(\CC) \to \Lambda$ onto a lattice is a pair $(Z,\HH)$ where $\HH$ is a bounded heart and $Z \in \Hom{\Lambda}{\CC}$ such that
\begin{enumerate}
\item if $h\in\HH$ then $Z(h) \in \U \cup \R_{\leq0}$ where $\U \subset \C$ is the strict upper half-plane, and
\item HN filtrations exist in $\HH$ for the slope function $\mu \coloneqq - \tfrac{\Re Z}{\Im Z} \colon K(\HH) \to \R \cup \{\infty\}$ with $\mu(h)=\infty$ when $\Im Z(h)=0$.
\end{enumerate}
The authors point out that $\{h\in\HH \mid Z(h)=0\}$ is a Serre subcategory of $\HH$ and that $Z$ induces a stability function on the abelian quotient category with respect to the appropriate quotient lattice. A very weak pre-stability condition $(Z,\HH)$ induces a slicing $P$, and in fact $(P,Z)$ is a lax pre-stability condition as in \cref{def:lax pre-stability condition} such that all massless objects are in $P(1)$. 

The article \cite{PT} defines a \defn{very weak stability condition} to be a pair $(Z,\HH)$ as above that moreover satisfies the support property in the vein of Kontsevich--Soibelman \cite{KS08} which is equivalent to the sypport property for the quotient pre-stability condition. This is a weaker condition than our lax support property. They consider the set $\Stabvw{\CC}$ of very weak stability conditions, topologised as a subset of $\Hom{\Lambda}{\CC} \times \Slice{\CC}$. They include points we do not because their support property is weaker, and we include points they do not because we allow more general slicings on the massless subcategory.

The discussion in \cite{BMS16} is almost identical to that of \cite{PT}, with the following differences: \cite{BMS16} restrict to bounded derived categories of smooth, projective threefolds and their charge function maps the heart to the upper half-plane $\U \cup \R_{\leq0}$ rotated by $-\pi/2$. They topologise the set of very weak stability conditions by giving it the coarsest topology such that the functions $\sigma \mapsto Z(c)$ and $\sigma \mapsto \phi_\sigma^\pm(c)$ are continuous for all $c\in\CC$.

These generalised stability conditions are used in both \cite{BMS16} and \cite{PT} as technical tools to prove the existence of strict stability conditions on the derived categories of certain varieties. The weaker notions are used to provide a framework in which tilt-stability always deforms. The above articles, as well as the one discussed next, are about generalising the classical Bogomolov--Giesker inequality for semistable vector bundles to settings involving stability conditions.

\subsection*{Weak stability conditions of Collins, Lo, Shi, Yau}
A recent article in this line is \cite{CLSY} by Tristan Collins, Jason Lo, Yun Shi and Shing-Tung Yau. They define a \defn{weak stability condition} as a triple $\sigma = (Z,\HH,\{\phi(h)\}_{h\in\ker(Z)\cap\HH})$ with $Z$ and $\HH$ as before and additionally $\phi \colon \ker(Z)\cap\HH \to (0,1]$, such that the function $\phi$ satisfies the weak seesaw property on $\ker(Z)\cap\HH$ (an exact sequence $0 \to h_1 \to h \to h_2 \to 0$ in $\ker(Z)\cap\HH$ implies $\phi(h_1) \geq \phi(h) \geq \phi(h_2)$ or $\phi(h_1) \leq \phi(h) \leq \phi(h_2)$) and the HN property.
In other words, massless objects $h$ in the heart $\HH$ get assigned phases $\phi(h) \in (0,1]$ in a compatible fashion.
\cite[Prop.~3.5]{CLSY} shows that this datum can be repackaged as $(Z,P,\{\phi(c)\}_{c\in\ker(Z)\cap P(0,1]})$ where $P$ is a slicing.

A lax pre-stability condition $(Z,P)$ induces a weak stability condition but not necessarily the other way around because \cite{CLSY} does not assume the slicing to be locally finite and disregards the support property. The goal of \cite{CLSY} is to prove the Bridgeland stability of certain objects in algebro-geometric settings, and weak stability conditions occur as limits of Bridgeland stability conditions (so that massless objects obtain phases as limits of phases in strict stability conditions).

\subsection{Bapat, Deopurkar and Licata's `Thurston compactification'}
\label{subsec:Bapat et al comparison}

In \cite{BDL20}, the authors categorify Thurston's compactification of Teichmüller space.
For a Hom-finite $\kk$-linear triangulated category $\CC$, they consider the map
\[
m\colon \PStab{\CC} \to \PP(\R^S), \quad \sigma\cdot \C \mapsto [m_\sigma(c) \mid c\in S ]
\]
where $S$ is the set of objects of $\CC$ which are semistable in some stability condition and define the Thurston compactification to be the closure $\overline{M(\CC)}$ of its image $M(\CC)$. They also consider analogues for other subsets $S$ of objects and $q$-deformed versions which we will not discuss here.

When $\PP(\R^S)$ is given the topology induced from the product topology on $\R^S$ the map $m$ is continuous, for instance by \cref{prop:mass and phase continuity}. Automorphisms of $\CC$ act on $\PP(\R^S)$ by pre-composing a real-valued function on the objects of $S$ with the inverse automorphism. 

The map $m$ is equivariant for $\Aaut{\Lambda}{\CC}$ because $m_{\alpha \cdot \sigma(c)} = m_\sigma(\alpha^{-1}(c))$.

The space $\overline{M(\CC)}$ is compact when $S$ contains a classical generator of $\CC$ \cite[Prop.~4.1]{BDL20} and the map $m$ is injective when $\CC = \Db(\Ginzburg Q)$ is the $2$-Calabi--Yau category associated to a quiver $Q$ \cite[Prop.~6.14]{BDL20}. When $Q=A_2$ they show that $M(\CC)$ is homeomorphic to $\Stab{\CC}$ and that $\overline{M(\CC)}$ is homeomorphic to a closed Euclidean ball \cite[Prop.~7.23]{BDL20}.

Motivated by the description of boundary points of Thurston's compactification as functionals given by unsigned intersections with closed curves, they show that for a spherical object $s$,
\[
\overline{\mathrm{hom}}(s) \coloneqq \bigg[ \sum_{n\in \Z} \dim_\kk \Homm{\CC}{s}{c[n]} \bigm| c\in S\bigg] \in \PP(\R^S)
\]
defines a point in the boundary $\partial \overline{M(\CC)}$ \cite[Cor.~4.13]{BDL20}. When $\CC=\Db(\Ginzburg A_2)$ these functionals form a dense subset of the boundary. 

The map $m$ extends to $m \colon \PStabL{\CC} \to \PP(\R^S)$ with the same definition; the extension is also continuous by \cref{prop:mass and phase continuity}. Because the masses depend only on the associated quotient stability condition, we get an induced map $m\colon \PStabQ{\CC} \to \PP(\R^S)$.
By construction the images of both $\PStabL{\CC}$ and $\PStabQ{\CC}$ are contained within $\overline{M(\CC)}$.

\begin{example}
Let $\CC = \Db(A_2)$. By \cref{subsec:A2}, the semistable objects of a point in $\Stab{A_2}$ are, up to shift, two or all three objects from $\mathcal{S} \coloneqq \{s,e,t\}$ where $s$ and $t$ are the two simple representations and $e$ the extension between them. It follows that the masses of $s$, $e$, and $t$ determine the masses of all objects. This remains true for $\StabL{A_2}$. So it suffices to consider
\[
\PStab{A_2} \to \PP(\R^3), \quad \sigma \mapsto [m_\sigma(s) : m_\sigma(e) : m_\sigma(t)].
\]
The image is cut out by the inequalities $x_0,x_1,x_2>0$ (the masses are strictly positive) together with the cyclic permutations of the inequality $x_0-x_1+x_2\leq 0$ (the mass of an extension is bounded by the sum of the masses of its factors). If we normalise so that $x_0+x_1+x_2=1$ then the image can be viewed as the shaded triangle, with vertices omitted, in the $2$-simplex in \cref{fig:A2 Thurston compactification}. In particular we see that the map is not injective because when, for example, $e$ is unstable the masses of $s$ and $t$ do not suffice to determine their phases. The three boundary strata in $\PStabL{A_2}$ where the masses of $s$, $e$ and $t$ respectively vanish are mapped to the three omitted vertices. So in this example, $\PStabQ{A_2}$ surjects onto Bapat, Deopurkar and Licata's compactification, and the boundaries coincide. The boundary points correspond precisely to the functionals $\overline{\mathrm{hom}}(c)$ for $c\in \mathcal{S}$.

\begin{figure}
\begin{tikzpicture}[scale=1.0]
\draw[dashed] (90:3)--(-30:3)--(-150:3)--cycle;
\draw[red, fill=red!20] (30:1.5)--(-90:1.5)--(-210:1.5)--cycle;
\draw[fill] (30:1.5) circle (2pt);
\draw[fill] (-90:1.5) circle (2pt);
\draw[fill] (-210:1.5) circle (2pt);
\node at (30:2.5) {$x_2=0$};
\node at (-90:2) {$x_0=0$};
\node at (-210:2.5) {$x_1=0$};
\end{tikzpicture}
\caption{The image of $\PStab{A_2} \to \Delta^2$, $\sigma \mapsto ( \lambda m_\sigma(s), \lambda m_\sigma(e), \lambda m_\sigma(t) )$ where $\lambda = m_\sigma(s)+ m_\sigma(e)+ m_\sigma(t)$ is shaded red. The chamber of the stability space in which $s$, $e$ and $t$ are stable is mapped homeomorphically to the interior. The chamber in which only $s$ and $t$ are stable, together with its bounding wall, are projected down onto the edge $x_0-x_1+x_2 = 0$, and similarly for the other two chambers. The three boundary points in $\PStabQ{A_2}$ where the masses of $s$, $e$ and $t$ respectively vanish are mapped to the three black vertices.}
\label{fig:A2 Thurston compactification}
\end{figure}
\end{example}

\begin{example}
Let $\CC = \Db(\Ginzburg A_2)$ and let $\mathcal{S}$ be the set of equivalence classes of spherical objects in $\Db(\Ginzburg A_2)$ up to isomorphism and shift. Then by \cite[Prop.~7.8]{BDL20} the map 
\[
m \colon \PStab{\Ginzburg A_2} \to \PP(\R^\mathcal{S}), \quad \sigma \mapsto [m_\sigma(c) \mid c\in \mathcal{S}]
\]
is a homeomorphism onto its image which we denote $M(\Ginzburg A_2)$. After choosing an element $s\in \mathcal{S}$ to map to $[1:0]$, the action of the Artin--Tits braid group induces a bijection $\mathcal{S} \cong \Z^2 / \pm$. Using this identification the map 
$
\mathcal{S} \to \PP(\R^\mathcal{S})$, $s \mapsto [ \overline{\mathrm{hom}}(s)(c) \mid c\in \mathcal{S} ]
$
 extends uniquely to a homeomorphism from $\PP(\R^2)$ onto the boundary of $M(\Ginzburg A_2)$ in $\PP(\R^S)$ by \cite[Props.~7.16 and 7.21]{BDL20}. The closure $\overline{M(\Ginzburg A_2)}$ is homeomorphic to a closed disk. The functional $\overline{\mathrm{hom}}(s)$ is the unique fixed point of the spherical twist $\twist_{s}$.
 
Now consider the extension $m\colon \PStabQ{\Ginzburg A_2} \to \overline{M(\Ginzburg A_2)}$. Recall from \cref{subsec:A2} and \cref{fig:A2} that the projective quotient stability space $\PStabQ{\Ginzburg A_2}$ contains one boundary point for each $s\in\mathcal{S}$ at which the objects in the class $s$ become massless. This boundary point is fixed by $\twist_{s}$ because $\twist_{s}$ acts by a shift on $s$. The equivariance of $m$ implies that this point is mapped to $\overline{\mathrm{hom}}(s)$. (At first sight this looks odd because $\sum_{n\in \Z} \dim_\kk\Hom{s}{s[n]}=2$ is non-zero. However we are working in an infinite-dimensional projective space 
\[
\lim_{n\to \infty} \frac{ m_\sigma(\twist^n_{s}(c)) } {n} = m_\sigma(s) \, \overline{\mathrm{hom}}(s)(c)
\]
for any $c\in \mathcal{S}$ and stability condition $\sigma$ in which $s$ is stable by \cite[Thm.~4.9]{BDL20}. Since $\twist_{s}$ fixes $s$ up to a shift, this implies that $\overline{\mathrm{hom}}(s)(s)=0$ as expected.) We conclude that
$
m \colon \PStabQ{\Ginzburg A_2} \to \overline{M(\Ginzburg A_2)}
$
is a continuous embedding, restricting to a homeomorphism between the interiors, and whose image is dense in the boundary. This provides a modular interpretation of the boundary points $\overline{\mathrm{hom}}(s)$ as the boundary points of $\PStabQ{\Ginzburg A_2}$.
\end{example}

\subsection{Bolognese's metric completion}
\label{subsec:Bolognese comparison}

In \cite{Bolognese23} Bolognese constructs an alternative extension of $\Stab{\CC}$ using a metric completion. She assumes that $\cm \colon \Stab{\CC} \to \Hom{\Lambda}{\C}$ is a cover of the complement of a locally finite union of submanifolds in $\Hom{\Lambda}{\C}$. Fixing an inner product on the $\R$-vector space $\Hom{\Lambda}{\C}$, she equips $\Stab{\CC}$ with the geodesic metric $d_B$ induced from the pullback of the associated metric. Since $\Stab{\CC}$ is locally homeomorphic to $\Hom{\Lambda}{\C}$ with its norm topology, this metric induces the usual topology on $\Stab{\CC}$.

A Cauchy sequence $(\sigma_n)$ in $(\Stab{\CC}, d_B)$ is called \defn{$\cm$-local} if it eventually lies in an open subset $U\subset \Stab{\CC}$ homeomorphic to its image via $\cm$. It has the \defn{limiting support property} if $\liminf\limits_{n\to \infty} C_n >0$ where
$C_n \coloneqq \inf\{ K>0 : \forall c\in P_n(\phi), \lim\limits_{m\to \infty} Z_m(c) \neq 0 \implies |Z_n(c)| > K\norm{c} \}$.

% This is \cite[Def.~4.3]{Bolognese23}}
% The set of such constants $K$ is non-empty because each $\sigma_n$ is in $\Stab{\CC}$ and so satisfies the support property.
% This property is well defined on equivalence classes of Cauchy sequences by \cite[Lem.~4.4]{Bolognese23}.
%
% Bolognese shows that any Cauchy sequence is equivalent to a $\cm$-local one.
% If two $\cm$-local Cauchy sequences are equivalent then they are $\cm$-local with respect to the same open $U$ \cite[Thm.~3.7]{Bolognese23}.

A $\cm$-local Cauchy sequence determines a thick subcategory of those objects which become massless in the limit, and a stability condition on the quotient category \cite[Prop.~4.2 and Thm.~6.1]{Bolognese23}. Moreover, $\cm$-local Cauchy sequences are equivalent if and only if they determine the same massless subcategory and stability condition on the quotient.

Denote by $\hatstab{\CC}$ the subspace of the metric completion of $(\Stab{\CC}, d_B)$ of equivalence classes of $\cm$-local Cauchy sequences with the limiting support property.
%
% As a topological space $\hatstab{\CC}$ is independent of the choice of inner product \cite[Lem.~3.6]{Bolognese23}.
%
One can check that a $\cm$-local Cauchy sequence $(\sigma_n)$ converges to a lax pre-stability condition $\sigma$ in the product metric on $\Slice{\CC}\times \Hom{\Lambda}{\C}$.
This provides a continuous map $\hatstab{\CC} \to \Slice{\CC}\times \Hom{\Lambda}{\C}$. 

Let $\Bstab{\CC}$ be the intersection of the image of this map with the closure of $\Stab{\CC}$. To compare Bolognese's construction with ours, we need to compare this with $\StabL{\CC}$. Unfortunately the relationship is not obvious since our lax support property is phrased in terms of massive stable objects in some $P(\phi-\delta,\phi+\delta)$ and Bolognese's limiting support property is phrased in terms of semistable objects in some $P_n(\phi)$ which remain massive in the limit. Since the HN filtration of a $\sigma_n$-semistable object with respect to $\sigma$ may contain massless objects, and similarly the other way round, there is no direct argument relating the two notions of support. 

If $\Bstab{\CC} = \StabL{\CC}$ then there should be a homeomorphism $\hatstab{\CC} \cong \StabQ{\CC}$.
% This holds for the $A_1$ quiver.
More generally an inclusion in either direction should extend to a map between $\hatstab{\CC}$ and $\StabQ{\CC}$ in the corresponding direction.

\subsection{Halpern-Leistner and Robotis' augmented stability conditions}
\label{Halpern-Leistner--Robotis}

In \cite{HR25}, the authors introduce an extension of $\Stab{\CC}$ through semi-orthogonal decompositions.
They first observe \cite[Prop.~2.3]{HR25} that a pre-stability condition $\sigma = (P,Z)$ on $\CC$ is uniquely determined by an aggregated mass and slope function
\[ \ell_\sigma \colon \CC \setminus 0 \to \C, \quad
   \ell_\sigma(E) = \log(m_\sigma(E)) + \frac{i\pi}{m_\sigma(E)} \sum_{\theta\in\R} \theta \, \abss{Z(H^\theta_P(E))} \]
where $m_\sigma(E)$ is the sum of masses of $\sigma$-semistable factors $H_P^\theta(E)$ of $E$. This function is compatible with the right $\C$-action, $\ell_{\sigma\cdot z} = \ell_\sigma + z$, and so descends to $\PStab{\CC} = \Stab{\CC} / \C$.

\newcommand{\HR}[3]{\langle #1 | #2\rangle_{#3}}
\newcommand{\HRroot}{\text{root}}
\newcommand{\AStab}[1]{\mathcal{A}\mathrm{Stab}(#1)}
\newcommand{\AaStab}[1]{\mathcal{A}\mathrm{Stab}^\mathrm{lax}(#1)}
They consider $\PStab{\CC}$ as the corresponding set of functions $\ell\colon \CC \setminus 0 \to \PP^1\setminus\{p_\infty\}$ where $p_\infty \in \PP^1 = \C\PP^1$ is a marked point, and write $\sigma = \HR{\CC}{\ell}{\PP^1}$. Skipping many details, the marked projective line is replaced by a \defn{multiscale line} \cite[Def.~3.1]{HR25}, which (ignoring additional data) is a projective curve $\Sigma$ of arithmetic genus 0, \ie a nodal tree of projective lines, together with a marked smooth point $p_\infty \in \Sigma$. This gives $\Sigma$ the combinatorial structure of a rooted tree. Let $\Sigma_\HRroot$ denote the irreducible component containing $p_\infty$.

A \defn{multiscale decomposition} \cite[Def.~3.6]{HR25} of $\CC$ consists of a multiscale line $\Sigma$ together with a non-zero thick subcategory $\CC_{\leq v} \subseteq \CC$ for each leaf (terminal vertex) $v$ of the rooted tree $\Sigma$, satisfying a list of conditions. If the rooted tree has height 1, \ie all non-root irreducible components $\Sigma_v$ meet $\Sigma_\HRroot$, and the intersection points $\Sigma_v \cap \Sigma_\HRroot$ have pairwise different imaginary values in $\C = \Sigma_\HRroot \setminus \{p_\infty\}$, this specialises to a semi-orthogonal decomposition $\CC = \clext{\CC_{\leq v_1},\ldots,\CC_{\leq v_k}}$. If instead all $\Sigma_v \cap \Sigma_\HRroot$ have the same imaginary value then the thick subcategories form a filtration $C_{\leq v_1} \subsetneq \cdots \subsetneq \CC_{\leq v_k} = \CC$; see Examples~3.7, 3.8 and Figure 4 in \cite{HR25}.

An \defn{augmented stability condition} \cite[Def.~3.36]{HR25} is a triple $\HR{\CC_\bullet}{\ell_\bullet}{\Sigma}$ consisting of a multiscale decomposition and, for each leaf $v$, a function $\ell_v \colon \text{gr}_v(\CC_\bullet) \setminus 0 \to \Sigma_v \setminus \{p_v\}$ that induces a stability condition on $\text{gr}_v(\CC_\bullet) = \CC_{\leq v}/\CC_{<v}$ and where $p_v \in \Sigma_v$ is the intersection with the other irreducible component meeting $\Sigma_v$ (or $p_\infty$ if $\Sigma = \PP^1$). Let $\AStab{\CC}$ be the set of augmented stability conditions on $\CC$ up to real oriented isomorphism of the underlying multiscale lines. By construction, this is an extension $\PStab{\CC} \subset \AStab{\CC}$. As for its geometry, Halpern--Leistner and Robotis conjecture this space to be a manifold with corners; see \cite[Conj.~A]{HR25} and \cite[Thm.~6.7]{HR25} for a partial result.

Denote by $\AaStab{\CC}$ the subset where the multiscale line consists of either single $\PP^1$ or the otherwise smallest possible tree consisting of the root and two terminal leaves, $v$ and $w$, and the associated multiscale decomposition corresponds to a filtration $\CC_{\leq v} \subsetneq \CC_{\leq w} = \CC$. We conjecture that there is a continuous map
$\AaStab{\CC} \to \PStabQ{\CC}$ extending the identity on $\PStab{\CC}$ by $\HR{\CC_\bullet}{\ell_\bullet}{\Sigma} \mapsto \sigma$ with $\NN_\sigma = \CC_{\leq v}$ corresponding to $\ell_w \in \PStab{\CC/\NN_\sigma} = \PStab{\text{gr}_w(\CC_\bullet)} = \PStab{\CC/\CC_{\leq v}}$. This map forgets the information of $\ell_v \in \PStab{\CC_{v}} = \PStab{\CC_{\leq v}}$.

%%% Not included but leave in text:
%
% $\AStab{\PP^1}$ (Section 6.3, Lemma 6.15): admissible boundary points either correspond to exceptional pairs $\clext{\cO(k),\cO(k+1)}$ or to filtrations $\clext{\cO(k)} \subsetneq \CC = \Db(\PP^1)$. Then $\mathcal{A}\mathrm{Stab^{adm}}(\PP^1) = (-\infty,\infty] \times \R$, extending $\PStab{\PP^1} = \C = (-\infty,\infty) \times \R$. There is one inadmissable point in $\AStab{\PP^1}$ given by the filtration $\{\text{torsion complexes}\} \subsetneq \Db(\PP^1)$; it is the limit of any path in $\C$ with $\Re(z) \to -\infty$.

\bibliographystyle{arxivalpha}
\bibliography{compactification-bibliography}

\bigskip
{\small\texttt{
  \noindent        
  \begin{tabular}{lll}
    \textit{\textrm{Contact:}}
      & nathan.broomhead@plymouth.ac.uk, & d.pauksztello@lancaster.ac.uk, \\
      & david.ploog@uis.no,              & jonathan.woolf@liverpool.ac.uk
  \end{tabular}}}

\end{document}